\documentclass[10pt]{article}
\usepackage[T1]{fontenc}
\usepackage{lmodern}
\usepackage[utf8]{inputenc}

\usepackage{graphicx}
\usepackage{epstopdf}

\usepackage{tikz-cd}
 
\graphicspath{{images/}}
\usepackage{verbatim}
\usepackage{amsmath,amsthm}
\usepackage{mathtools}
\usepackage{xspace}
\usepackage{pifont}
\usepackage{graphicx}
\usepackage{amssymb}
\usepackage{epic, eepic}
\usepackage{dsfont}
\usepackage{amssymb}
\usepackage{makeidx}
\usepackage{mathrsfs}
\usepackage{exscale}
\usepackage{color} 
\usepackage{overpic} 
\usepackage{bm}
\usepackage{bbm}
\usepackage{booktabs} 
\usepackage{color, colortbl}
\usepackage{subcaption}
\usepackage[numbers]{natbib}
\usepackage{authblk}
\RequirePackage[colorlinks,citecolor=blue,urlcolor=blue]{hyperref}
\usepackage{amsmath,afterpage}
\usepackage{epsf}
\usepackage{graphics,color}
\RequirePackage[colorlinks,citecolor=blue,urlcolor=blue]{hyperref}
\usepackage{natbib}
\usepackage[normalem]{ulem}

\def\0{\mathbf{0}}

\def\eps{\varepsilon}

\def\lam{\lambda}
\def\rr{\rightarrow}

\def \< {\langle}
\def \> {\rangle}

\def\ol{\overline}
\def\ul{\underline}

\def\to{\rightarrow}
 \def\ol{\overline}    \def\ul{\underline}

\def\eps{\varepsilon}

   \def\lam{\lambda}

\def\E{{\bf E}}

\def\black#1{\textcolor{black}{#1}}

\def\red#1{\textcolor{red}{#1}}

\newtheorem{theorem}{Theorem}[section]
\newtheorem{lemma}[theorem]{Lemma}

\newtheorem{proposition}[theorem]{Proposition}

\newtheorem{corollary}[theorem]{Corollary}

\newtheorem{remark}[theorem]{Remark}
\newtheorem{example}{Example}[section]

\newtheorem{definition}[theorem]{Definition}

\newtheorem{assumption}[theorem]{Assumption} 
\numberwithin{equation}{section}
\newcommand{\hatd}[1]{{}}

\newcommand{\bd}{\begin{displaymath}}
\newcommand{\ed}{\end{displaymath}}
\newcommand{\be}{\begin{equation}}
\newcommand{\ee}{\end{equation}}
\newcommand{\bq}{\begin{eqnarray}}
\newcommand{\eq}{\end{eqnarray}}
\newcommand{\bn}{\begin{eqnarray*}}
\newcommand{\en}{\end{eqnarray*}}
\newcommand{\dl}{\delta}

\def\wt{\widetilde}

\def\E{{\mathbb{E}}}

\newcommand{\id}{{\rm id}}

 \usepackage[normalem]{ulem}
	
\usepackage[shortlabels]{enumitem}
\usepackage[linesnumbered, ruled,vlined]{algorithm2e}
\SetKwInput{KwInput}{Input}                
\SetKwComment{Comment}{/* }{ */}
\def\YZ#1{{\textcolor{blue}{Yufei: #1}}}

\newcommand{\tr}{\textnormal{tr}}

\DeclareMathOperator*{\argmin}{arg\,min}

\newcommand{\lc}
{\mathrel{\raise2pt\hbox{${\mathop<\limits_{\raise1pt\hbox
{\mbox{$\sim$}}}}$}}}
\newcommand{\gc}
{\mathrel{\raise2pt\hbox{${\mathop>\limits_{\raise1pt\hbox{\mbox{$\sim$}}}}$}}}

\newcommand{\ec}
{\mathrel{\raise2pt\hbox{${\mathop=\limits_{\raise1pt\hbox{\mbox{$\sim$}}}}$}}}

\def\bb{\begin{equation}} \def\ee{\end{equation}}
\def\bbn{\begin{equation*}} \def\een{\end{equation*}}

\def\beqn{\begin{eqnarray}}  \def\eqn{\end{eqnarray}}

\def\beqnx{\begin{eqnarray*}} \def\eqnx{\end{eqnarray*}}

\def\bn{\begin{enumerate}} \def\en{\end{enumerate}}

\def\bd{\begin{description}} \def\ed{\end{description}}



\def\cG{\mathcal{G}}

\def\sP{\mathbb{P}}

\def\sR{{\mathbb R}}

\title{An Offline Learning Approach to Propagator Models
}
\author{Eyal Neuman} 
\author{Wolfgang Stockinger}
\author{Yufei Zhang\footnote{Corresponding author. \href{mailto:yufei.zhang@imperial.ac.uk}{yufei.zhang@imperial.ac.uk}}}
\affil{Department of Mathematics, Imperial College London}

\begin{document}

 \vspace{-0.5cm}
\maketitle

\begin{abstract}
 We consider an offline learning problem for an agent who first estimates an unknown price impact kernel from a static dataset, and then designs strategies to liquidate  a risky asset while creating transient price impact. We propose a novel approach for a nonparametric estimation of the propagator from a dataset containing 
correlated
 price trajectories, trading signals and  metaorders. We quantify the accuracy of the estimated propagator using a metric which depends explicitly on 
 the dataset.  We show that a trader who tries to minimise her execution costs by using a greedy strategy purely based on the estimated propagator will encounter suboptimality due to so-called spurious correlation between the trading strategy and the estimator and due to intrinsic uncertainty resulting from a biased cost functional. By adopting an offline reinforcement learning approach, we introduce a pessimistic loss functional taking the uncertainty of the estimated propagator into account,
 with an optimiser which eliminates the spurious correlation,  and derive an asymptotically optimal bound on the execution costs even without precise information on the true propagator. 
Numerical experiments are included to demonstrate the effectiveness of the proposed propagator estimator and the pessimistic trading strategy. 
 
\end{abstract} 

\begin{description}
\item[Mathematics Subject Classification (2010):] 
62L05,
60H30, 91G80, 68Q32, 93C73, 93E35, 62G08
\item[JEL Classification:] C02, C61, G11
\item[Keywords:] optimal portfolio liquidation, price impact, propagator models, predictive signals, Volterra stochastic control, 
offline reinforcement learning, 
nonparametric estimation, 
pessimistic principle, 
regret analysis
\end{description}


\section{Introduction}  \label{sec-mot-res} 
Price impact refers to the empirical fact that execution of a large order affects the risky asset's price in an adverse and persistent manner and is leading to less favourable prices for the trader. Accurate estimation of transactions' price impact is instrumental for designing profitable trading strategies. Propagator models serve as a central tool in describing these phenomena mathematically (see \citet{bouchaud-gefen, gatheral2010no}). These models express price moves in terms of the influence of past trades, and give a reliable reduced form view on the limit order book reaction for trades execution. The model's tractability provides a convenient formulation for stochastic control problems arising from optimal portfolio liquidation \citep{AS12,GSS,AJ-N-2022,ANV-23}. 

For an agent who wishes to liquidate a large order, a common practice for avoiding undesirable execution costs due to price impact, is to split the so called \emph{metaorder} into smaller child orders described by $\{u_{t_j} :\ 0=t_1< \ldots <t_M=T \}$ over a time interval of length $T$. Propagator models assume that there exists a matrix $G = (G_{i,j})$, known as the \emph{propagator} or the \emph{price impact kernel}, such that the asset's execution price $S$ is given by
 \begin{equation} \label{p-proc} 
S_{t_{i+1}} =  S_{t_1} + \sum_{j =1}^{i} G_{i,j}  u_{t_j} 
+A_{t_i}+  \varepsilon_{t_i} ,
\quad i=1,\ldots, M,
\end{equation}
where the process $A+\eps$ represents the fundamental (or unaffected) asset price. Here $A$ is often referred to as a trading signal and $\eps$ is a mean-zero noise process (see e.g., \citet[Chapter 13]{bouchaud_bonart_donier_gould_2018}). The propagator $G_{i,j}$ typically decays for $j \gg i$ and hence the sum on the right-hand side of \eqref{p-proc} is often referred to \emph{transient price impact}. A well-known example à la Almgren and Chriss, discusses the case where all entries of $G$ are similar, then this sum represents permanent price impact. If $G = \lam \mathbb{I}$, where $\lam>0$ and $\mathbb{I}_M$ is the $\mathbb{R}^{M \times M}$ identity matrix, then the sum represents temporary price impact (see \citep{AlmgrenChriss1,OPTEXECAC00}). 

In the aforementioned setting, the trader can only observe the visible price process $S$, her own trades $u$ and the trading signal $A$. In order to quantify the price impact and to design a profitable trading strategy, the trader needs a precise estimation of the matrix $G$.
Some estimators for the convolution case (i.e., where  $G_{i,j} = G_{i-j})$ were proposed in \citep{bouchaud-gefen,Forde:2022aa, Toth_17,Benzaquen_22} and in Chapter 13.2 of \cite{bouchaud_bonart_donier_gould_2018}. These regression based estimation methods are already a common practice in the industry, however they ignore numerous mathematical issues, such as the illposedness of the least-squares estimation problem,
dependencies between price trajectories and spurious correlations between the estimator and greedy strategies. Hence the convergence of these price impact estimators remains unproved and rigorous results on the convergence rate are considered as long-standing open problem. 

In \cite{neuman2023statistical} a novel approach for nonparametric estimation of the price impact kernel was proposed for the continuous time version of \eqref{p-proc}. Sharp bounds on the convergence rate which are characterised by the singularity of the propagator were derived therein. The estimation phase in \cite{neuman2023statistical} was designed for an online reinforcement learning framework where the agent interacts with the environment while trading. However, there are a few crucial drawbacks for the estimation methods developed in \cite{neuman2023statistical}: the trader's strategy had to be similar and deterministic in all interactions with the market, and moreover the price trajectories in each episode were assumed to be independent. These assumptions do not apply in crowded markets for example, where all agents are following similar trading signals (see \cite{N-V-2021}). Lastly, but most importantly, online learning for portfolio liquidation is very expensive. In order to get one sample path of $(S,u,A)$ in \eqref{p-proc} the trader needs to execute an order according to a highly suboptimal strategy, which leads to undesirable liquidation costs. For this reason most propagator calibrations are done offline, while fine tuning of the estimator should take place using online algorithms. 

In order to resolve the above issues, the goals of this work are as follows:
\begin{itemize}
    \item[\textbf{(i)}] to provide an offline estimator for the propagator $G$ in \eqref{p-proc} based on a static dataset of historical trades and prices and to derive optimal convergence  rate for this estimator. 
    \item[\textbf{(ii)}] to propose a stochastic control framework that will eliminate the correlation between the offline estimator $\hat G$ and a greedy trading strategy based on $\hat G$, and to derive a tight bound on its performance gap. 
    \end{itemize}
We will first provide details on objective (i).  We choose to work in a discrete time setting, which is compatible with financial datasets and is inline with  common practice in the financial industry. Our approach for the propagator estimation is offline since most of the existing data available to researchers is historical. Further, as mentioned before, online learning in order execution is expensive due to suboptimal interaction with the market using greedy trading strategies. Hence, deriving estimates for the price impact based on existing datasets could save substantial trading costs. In the following, we assume that there exists a dataset $\mathcal{D}$ with $N$ price and signal trajectories, including the corresponding agents' metaorders following the dynamics in \eqref{p-proc},  
\be \label{data-b} 
\mathcal{D} =  \left\{(S^{(n)}, u^{(n)}, A^{(n)})
\mid     n=  1, \ldots, N\right\},
\ee
see Definition \ref{def:offline_data} for details. In Theorems \ref{thm:confidence_volterra} and \ref{thm:confidence_convolution}, we improve the convergence rate results of Theorem 2.10 of \cite{neuman2023statistical}. Our contribution to propagator estimations is as follows:
\begin{itemize} 
\item[\textbf{(i)}] We relax the assumption on independence between the price trajectories, by assuming that the noise process $\eps^{(n)}$ is conditionally sub-Gaussian (see Assumption \ref{assum:offline_data}).   
\item[\textbf{(ii)}] We allow for any price adaptive strategies $u^{(n)}$ in the dataset $\mathcal{D}$, in contrast to choosing only one deterministic strategy that will repeat itself (i.e., $u^{(n)} = u^{(m)}$ for all $0\leq m,n \leq N$) as in Theorem 2.10 of \cite{neuman2023statistical} . 
\item[\textbf{(iii)}] We allow for estimation of non-convolution propagators $G$ and we show that in the convolution case the convergence rate is substantially faster (see Remark \ref{rem-con-rate}). 
\item[\textbf{(iv)}] We present a conceptually new norm,  
\be \label{norm-v-n}
\|\hat G  - G\|_{V_{N,\lambda}} := \|(\hat G  - G)V_{N,\lambda}^{\frac{1}{2}}   \|,
\ee
which not only quantifies the estimation error in terms of the quality of the dataset, but will also have a crucial role in the pessimistic offline learning approach which we present in the following. Here $\hat{G}$ denotes the estimated propagator, and $V_{N,\lambda}$ is a matrix that measures the covariance between the $u^{(n)}$'s in $\mathcal D$ (see \eqref{eq:project_admissible}). 
\end{itemize}

Next we discuss objective (ii) of this work, to propose a stochastic control framework that will eliminate the correlation between a greedy strategy and the estimated price impact kernel. Precise estimation of the propagator is crucial in portfolio liquidation and portfolio selection problems in order minimise the trading costs. The following cost functional measures the expected costs of a trade of size $x_0$ executed within the time interval $[0,T]$  (see~\citep{cartea15book,GSS2,GSS,Lehalle-Neum18} among others):  
 \begin{equation} \label{cost} 
\begin{aligned}
J(u;G) 
&=  \mathbb{E}\left[ \sum_{i=1}^{M}  \sum_{j=1 }^{i} G_{i,j}  u    _{t_j} u_{t_i}
+ \sum_{i=1}^M A_{t_i} u_{t_i}  \right],
\end{aligned} 
\end{equation}  
where $u=(u_{t_i})_{i=1}^M$ satisfies a fuel constraint  i.e., $\sum_{i=1}^N u_{t_i} = x_0$. The first term in the expectation on the right-hand side of \eqref{cost} represents the price impact costs and the second term represents the fundamental price of the trade.  

The two step procedure starting with offline estimation of the propagator $G$ and then using the estimator $\hat G$  in order to obtain a greedy strategy $\hat u$ (i.e., to optimize $J(u;\hat{G})$ over a set of admissible controls) introduces some statistical issues and challenges which fit into the framework of \emph{offline reinforcement learning} (offline RL). As mentioned before typical training schemes of online RL algorithms rely on interaction with the environment, where in the case of trading this interaction is expensive and in other areas such autonomous driving, or healthcare it is dangerous.  Recently, the area of offline RL has emerged as a promising candidate to overcome this barrier. Offline RL algorithms learn a policy from a static dataset which is collected before the agent executes any greedy policy. However, these algorithms may suffer from several pathological issues which classified as follows:
\begin{itemize} 
\item[\textbf{(i)}] Intrinsic uncertainty: the dataset possibly fails to cover the trajectory induced by the optimal policy, which however carries the essential information.  
\item[\textbf{(ii)}] Spurious correlation: the dataset possibly happens to cover greedy trajectories unrelated to the optimal policy, which by chance induces low execution costs. This leads to an estimation error of the true parameters, and to an underperforming greedy policy correlated with the estimated parameters. 
 \end{itemize} 

 These issues were outlined and partially quantified for Markov decision processes with unknown distribution of transition probabilities and rewards in \citep{chang2021mitigating, jin2021pessimism,NEURIPS2020_f7efa4f8, Li-Tang-21,Uehara-21}. We demonstrate how these issues emerge in offline propagator estimation. In the spirit of Section 3.1 of \cite{jin2021pessimism}, we decompose the suboptimality of any strategy into the following ingredients: the spurious correlation, intrinsic uncertainty and optimization error.  We denote by $( G^{\star}, u^{\star},J)$ the true propagator, the optimal strategy and the associated cost functional $J$ in \eqref{cost}, respectively. We use the notation $(\hat G, \hat u, \hat J)$ for the estimator of the true kernel $G^\star$, a convex cost functional $\hat J(\cdot;\hat G)$  using $\hat G$, and a greedy strategy $\hat u$ minimising this functional. 
Note that we do not insist that $\hat J $ and $J$ necessarily agree, that is we allow $ \hat J(\cdot,  G) \not=  J(\cdot,  G)$ for some or all $G$'s.  

We further define the suboptimality of an arbitrary algorithm that aims to estimate $G^{\star}$ and also produces a 
strategy $\hat u$ as follows:
$$
\textrm{SubOpt} = J(\hat u; G^{\star}) - J( u^{\star}; G^{\star}). 
$$
We decompose the suboptimality of such algorithm into the following components  along the same lines as Lemma 3.1 of \cite{jin2021pessimism}, 
\be  \label{sub-decom} 
\begin{aligned}
\textrm{SubOpt} &=\underbrace{J(\hat u; G^{\star}) - \hat J(  \hat u ; \hat G)}_\text{(i): Spurious Correlation} + \underbrace{\hat J( u^{\star}; \hat G) - J( u^{\star}; G^{\star})}_\text{(ii): Intrinsic Uncertainty}  +\underbrace{\hat J(  \hat u ; \hat G) -\hat J(  u^\star ;  \hat G)}_\text{(iii): Optimization Error}. 
\end{aligned}
\ee 

Note that term (i) in \eqref{sub-decom} is the most challenging to control, as $\hat u$, $\hat G$ and hence $\hat J$ simultaneously depend on the dataset $\mathcal{D}$. This leads to a spurious correlation between them. In Example \ref{exmp-corr} of Appendix \ref{sec-examples}, we show that such spurious correlation can lead to a greedy strategy $\hat u$, which is substantially suboptimal in a specific tractable case. More involved examples are given in Section \ref{sec-numerics}. We refer to Section 3.2 of \cite{jin2021pessimism} for an additional negative example for multi-armed bandit problems. 

In particular, due to the correlation between $\hat u$ and $\hat G$, term (i) can have a large expectation, with respect to the probability measure induced by the dataset, even under the assumption that  $\hat G$ is an unbiased estimator of $G^{\star}$.
In contrast, term (ii) is less challenging to control, as $u^\star$ is the minimiser of $J(\cdot;  G^{\star})$, hence it does not depend on $\mathcal D$ and in particular on $\hat J(\cdot; \hat G)$. In Section 4.3 of  \cite{jin2021pessimism}, it was shown that the intrinsic uncertainty is impossible to eliminate for linear Markov decision processes. Finally we notice that the optimization error in term (iii) is nonpositive as long as $\hat u$ is greedy (i.e., optimal) with respect to $\hat J(\cdot; \hat G)$. 

Note that even with the tools developed in this paper, the estimation of price impact can be quite sensitive to various conditions in the market, which are sometimes difficult to quantify precisely. As pointed out before such deviations in the estimations of the true  propagator will lead to undesirable costs due to spurious correlation (see Example \ref{exmp-corr}). For example in \citep{Madhavan2001,Mich-Neum20}  permanent and temporary price impact on the Russell 3000 annual additions and deletions events were computed. 
It was found that after 2008, the temporary market impact often presented a negative sign, presumably amenable to the 2010s bull market which signed a positive trend in the equity stock market \citep{wsj_2019,ft_2019,reuters_2019}.  While a regression model can be applied in order to fully test the effect of market trends on price impact, both the reference market index S$\&$P500 and  individual stock returns present significant autocorrelations across different lags \cite{EGADEESH1990}, hence this type of analysis is quite involved. 

In order to minimise the ingredients of the suboptimality in \eqref{sub-decom} we adopt a \emph{pessimistic offline RL} framework, which penalises trading strategies  visiting states or using actions, which were not explored in the static dataset.  We define the uncertainty quantifier which is a key concept in this theory. In the following we denote by $\mathbb P' $ the probability measure governing the dataset (see Section \eqref{sec-pessim} for the precise definition). 
We refer to $\mathcal A$ as the class of admissible controls with respect to the cost functional $J$ in \eqref{cost}, where a precise definition is given in \eqref{admiss}. 

\begin{definition} [$\dl$-uncertainty quantifier] \label{def-quant} We say that $\Gamma:\mathcal A \rr \mathbb{R}_+$ is a $\dl$-uncertainty quantifier with respect to $\mathbb{P}'$ if the following event 
$$
A = \left \{| J( u; G^{\star}) - J(   u ; \hat G)| \leq \Gamma(u), \textrm{ for all } u\in \mathcal A \right \},
$$
satisfies $\mathbb{P}' (A) \geq 1-\dl$ for $\delta \in (0,1)$.
\end{definition} 

Assuming that we have $\dl$-uncertainty quantifier $\Gamma$ we now propose a candidate for $\hat J$ so that the spurious correlation in \eqref{sub-decom} is nonpositive, 
\be \label{j-h-def}
\hat J (u ; \hat G) =  J (u ; \hat G) + \Gamma(u), \quad   u \in \mathcal A. 
\ee
We also require that $\Gamma(u)$ is convex so that the cost functional \eqref{j-h-def} is convex. Indeed by Definition \ref{def-quant} and \eqref{j-h-def}, the spurious correlation is now bounded by 
\be \label{spurious-neg}
J(\hat u; G^{\star}) - \hat J(  \hat u ; \hat G) = J(\hat u; G^{\star}) - J (\hat u ; \hat G) - \Gamma(\hat{u})\leq 0,
\ee
with probability $1-\delta$ for $\delta \in (0,1)$. 

Using $\hat J$ as in \eqref{j-h-def} concludes that suboptimality in \eqref{sub-decom} only corresponds to term (ii) which characterizes the intrinsic uncertainty. Our definitions and main results in Sections \ref{sec-pessim} and \ref{sec-res-pess} propose such a $\dl$-uncertainty quantifier $\Gamma$ which is sufficiently efficient and helps us to establish a tight upper bound for the suboptimality in \eqref{sub-decom}. It is important to notice that deriving $\Gamma$ as in \eqref{def-quant} is highly nontrivial since the true propagator $G^{\star}$ is not an observable. In order to derive an uncertainty quantifier we use the bounds, which were developed in the estimation part of the paper, on the distance between the estimator $\hat G$ and $G^{\star}$ in \eqref{norm-v-n}, under the norm $\|\cdot\|_{V_{N,\lambda}}$. Specifically, in \eqref{j-p1} we choose $\hat J(\cdot ;\hat G) = J_{P,1}(\cdot ; G_{N,\lam})$ and in Corollary \ref{cor-quant} we show that the $\dl$-uncertainty quantifier is given by,  
\be \label{ell-1-def} 
\Gamma(u) = \E[\ell_1(V_{N,\lambda},u)],  \quad \textrm{where }  \ \ell_1(V_{N,\lambda},u) \coloneqq  C(N)
   \big\| V_{N,\lambda}^{-\frac{1}{2}} u  \big\|, 
\ee
for some constant $C(N) >0$.
The optimiser of $\hat J(\cdot ;\hat G)$ is referred to as a pessimistic optimal strategy and it is denoted by $u^{(1)}$. 
In Example \ref{exmp-l1-pen} we demonstrate the effect of $\ell_1$ on the pessimistic optimal strategy in a simple and tractable setup. More realistic examples are given in Section \ref{sec-numerics}. This is the first application of pessimistic offline RL in the area of quantitative finance and also in continuous state stochastic control. In the reinforcement learning literature, mathematical results that quantify these issues of offline RL are scarce, and they focus on Markov decision processes (MDPs). In the following we summarise our contribution in this direction:  
\begin{itemize} 
\item[\textbf{(i)}] In Theorem \ref{THM:performance_bound}(i) we prove that the spurious correlation of $u^{(1)}$ is nonpositive and that the suboptimality of $u^{(1)}$, which is essentially due to the intrinsic uncertainty, is bounded by $$    \textrm{SubOpt}  =  J( u^{(1)};   G^{\star}) - J( u; G^{\star}) \leq 2 \mathbb{E} \left[  \ell_1(V_{N,\lambda},u) \right],$$
for any admissible strategy $u$ including $u^{\star}$ which is unknown.     
\item[\textbf{(ii)}] In Theorem  \ref{thm-l-bound}(i), we prove that  
$\E[\ell_1(V_{N,\lambda},u)] = O(N^{-1/2}(\log N)^{1/2})$, where $N$ corresponds to the the number of samples in the static dataset \eqref{data-b}, under the assumption of a well-explored dataset (see Assumption \ref{assum:concentration_inequ} and Remark \ref{rem-asump-mat}) . Together with \eqref{sub-decom} and \eqref{spurious-neg} this leads to the asymptotic behaviour of the performance gap, 
\be  \label{sub-res} 
 \textrm{SubOpt} =O(N^{-1/2}(\log N)^{1/2}),
\ee 
where the constants of this rate are given explicitly. 
Special treatment and slightly refined results are given for convolution kernels in Theorems \ref{THM:performance_bound}(ii) and \ref{thm-l-bound}(ii). In Remark \ref{rem-asump-mat}, we show that for a convolution type propagator the assumption of a well-explored dataset should hold with asymptotically high probability for large $N$, in contrast to a Volterra type propagator. 
\item[\textbf{(iii)}] Note that the expected convergence rate for a generic least-square estimator $\hat G$ to the true kernel $G$ is of order $O(N^{-1/2})$, hence the statement of Theorem \ref{thm-l-bound} is sharp up to a logarithmic factor. In fact the logarithmic correction in \eqref{sub-res} is a result of the estimation scheme of $\hat G$, as described in \eqref{eq:G_n_lambda_volterra}, which is subject to quality of coverage of the dataset (see \eqref{eq:G_unconstraint_minimiser}). For a sufficiently regular dataset the regularization is not needed and the logarithmic term in Theorem \ref{thm:confidence_volterra} will vanish. This will give us asymptotically sharp bounds in Theorem \ref{thm-l-bound}(i), that is $\textrm{SubOpt} =O(N^{-1/2} )$.  
\end{itemize} 

We compare the contribution of this work to \cite{jin2021pessimism}, where partial results on pessimistic offline RL with respect to spurious correlation and intrinsic uncertainty were derived for MDPs. The bound on suboptimality in Theorem \ref{THM:performance_bound} coincides with the corresponding bound established in Theorem 4.2 of \cite{jin2021pessimism} that deals with a much simpler case of a class of MDPs where the rewards and distribution of the transition probabilities are Markovian. Note that stochastic control problems that involve propagators not only take place in a continuous state space but are also non-Markovian (see e.g., \cite{AJ-N-2022}). The convergence rate of $O(N^{-1/2}\log N)$ in Theorem \ref{thm-l-bound} coincides with the convergence rate  established in Corollary 4.5 of \cite{jin2021pessimism} for \emph{linear} MDP under similar assumptions.  

\paragraph{Structure of the paper:} 
In Section \ref{sec-l-p} we setup the trading model and describe our main results on propagator estimation in Theorems \ref{thm:confidence_volterra} and \ref{thm:confidence_convolution}, and on the performance gap for pessimistic learning in Theorems \ref{THM:performance_bound} and \ref{thm-l-bound}. Section \ref{sec-numerics} is dedicated to numerical illustrations of the propagator estimation and of pessimistic learning strategies. In Section \ref{sec-mart}, we develop the mathematical foundations for our results on convergence of the propagator estimators. Sections \ref{sec-pf-est}--\ref{sec-pf-opt} are dedicated to the proofs of the main results. In Appendix \ref{sec-examples} we provide some simple and tractable examples for the framework of pessimistic  RL in portfolio liquidation. In Appendix \ref{sec-examples_noisy} we provide further numerical experiments for estimation of Volterra kernels. 

  \section{Problem formulation and main results} \label{sec-l-p} 

In this section we present our results on the optimal liquidation model which was briefly described in \eqref{cost}, the price impact estimation with offline data and finally on the pessimistic control problem which corresponds to \eqref{j-h-def}. 
\subsection{The optimal liquidation model} \label{sec-liq-model}  
We fix $T>0$ along with an equidistant partition of the interval $[0,T]$, $\mathbb{T} \coloneqq \lbrace 0=t_1,\ldots,t_M=T \rbrace$. Let $(\Omega, \mathcal F, \mathbb F=\{\mathcal F_{t_i}\}_{i=1}^M, \mathbb{P})$ be a filtered probability space and let $\mathcal F_0$ be the null $\sigma$-field.  Let $(A_{t_i})_{i=1}^M$  and $(\eps_{t_i})_{i=1}^M$ be $\mathbb F$-adapted stochastic processes satisfying
$$
 \mathbb{E}[\eps_{t_i}^2]  < \infty, \quad  \mathbb{E}[A_{t_i}^2] < \infty, \quad \textrm{for all } i=1,\ldots,M. 
$$  
where we further assume that 
$$
\mathbb{E}[\eps_{t_i}\mid \mathcal{F}_{t_{i-1}}]=0, \quad \textrm{for all } i=1,\ldots,M. 
$$
We consider a trader with a target
position of $x_0 \in \mathbb{R}$  shares in a risky asset. The
number of shares the trader holds at time $t_n$ is prescribed as
$$
X_{t_n} =  \sum_{i=1}^{n} u_{t_i}, \quad n=1,\ldots,M, 
$$
where $u = (u_{t_1}, \ldots, u_{t_M})^\top$ denotes her trading speed, choosen from a set of admissible strategies
\be \label{admiss} 
\mathcal A = \left\{ (u_{t_i})_{i=1}^M \textrm{ are } \mathbb F\textrm{-adapted and }  \mathbb{E}[u_{t_i}^2] < \infty,  \ \textrm{for all } i=1,\ldots,M \right\}. 
\ee
We will further impose the fuel constraint $\sum_{i=1}^{M} u_{t_i} = x_0$, hence $x_0$ can be regarded as the target inventory to be  
{achieved}
by the terminal time. 
We fix an 
$M \times M$ matrix $G=(G_{i,j})_{i,j=1}^{M}$, which is called the propagator, such that
\be \label{g-def} 
\sum_{i=1}^{M}  \sum_{j=1 }^{M} (G_{i,j} +G_{j,i})  x_{j} x_{i} > 0, \quad \textrm{for all } x \in \mathbb{R}^M, \quad G_{i,j}= 0, \quad \textrm{for } 1 \leq i < j \leq M. 
\ee
We assume that the trader's trading activity causes transient price impact on the risky asset's execution price in the sense that her orders are filled at prices 
\begin{equation}\label{eq:price_dynamics_G_star}
S_{t_{i+1}} =  S_{t_1} + \sum_{j =1}^{i} G_{i,j}  u_{t_j} 
+A_{t_i}+  
\varepsilon_{t_i} ,
\quad i=1,\ldots, M,
\end{equation}
where $S_{t_1}$ is an $\mathcal F_{t_1}$-measurable, square integrable random variable, which describes the initial price.  Note that in this setup the process $A+\eps$ represents the fundamental (or unaffected) asset price, where $A$ is often referred to as a trading signal see e.g., Section 2 of \cite{NeumanVoss:20}.  

Following a similar argument as in Section 2 of \cite{Lehalle-Neum18} we consider the following performance functional, which measures the costs of the trader's strategy in the presence of transient price impact and a signal, 
\begin{equation} \label{per-func} 
\begin{aligned}
J(u;G) 
&=  \mathbb{E}\left[ \sum_{i=1}^{M}  \sum_{j=1 }^{i} G_{i,j}  u_{t_j} u_{t_i}
+ \sum_{i=1}^M A_{t_i} u_{t_i}  \right],
\end{aligned} 
\end{equation} 
where $u=(u_{t_i})_{i=1}^M \in \mathcal A$ satisfies a fuel constraint. We therefore wish to solve the following cost minimization problem,  
\begin{equation} \label{opt-prog} 
   \begin{cases}
&\min_{u \in \mathcal A} J(u;G), \\
&\textrm{s.t.}  \sum_{i=1}^N u_{t_i} = x_0 .
  \end{cases} 
\end{equation}

In order to derive the optimiser to \eqref{opt-prog},  we introduce some additional definitions and notation.  
\paragraph{Notation} 
For $G$ as in \eqref{g-def}, we denote by $2\eta >0$ the smallest eigenvalue of the positive definite matrix $G +G^{\top} $ and define  
\be \label{g-dec} 
  G = \eta  \mathbb{I}_M + \wt G,
\ee 
so that $ \wt G$ satisfies 
\be \label{g-non-neg-def} 
\sum_{i=1}^{M}  \sum_{j=1 }^{M} (\wt G_{i,j} + \wt G_{j,i})  x_{j} x_{i} \geq 0, \quad \textrm{for all } x \in \mathbb{R}^M, \quad \wt G_{i,j}= 0, \quad \textrm{for } j>i. 
\ee

Here $\mathbb{I}_M$ is the $M \times M$ identity matrix, where we often omit the subscript $M$ when there is no ambiguity. We define the following $M$-vectors: $\boldsymbol{1} = (1,\ldots,1)^{\top}$ and $A=(A_{t_1},\ldots,A_{t_M})^\top$. Lastly we denote $ \mathbb{E}_{s} [\cdot] =  \mathbb{E}[ \cdot | \mathcal F_s]$ for any $s \in \mathbb{T}$. 
For $\wt G$ as in \eqref{g-dec}, let
\be \label{tile-g-i} 
\wt G^{(\ell)} := (\wt G^{(\ell)}_{i,j})  =  ( \wt G_{i,j} 1_{\{\ell \leq  i\}}), \quad \ell = 1, \ldots, M.  
\ee
We further define the matrices 
$$
D^{(\ell)} =  2\eta \mathbb I_M+\wt G^{(\ell)}    + 
(\wt G^{\top})^{(\ell)},
$$
and $K = (K_{i,j})_{i,j=1}^{M}$ with
\be \label{k-def-og} 
K_{i,j}=   \wt G_{i,j} -   \mathrm{1}_{\lbrace t_j <  t_i \rbrace} (\wt G^{\top}    (D^{(i)} )^{-1}\wt G^{(i)})_{i,j }.
\ee
 
The following theorem derives an explicit solution to the constrained optimal liquidation problem in \eqref{opt-prog}.
\begin{theorem} \label{prop-opt-strat} 
There exists a unique admissible strategy solving \eqref{opt-prog} which is given by 
$$
u_{t_i} = \sum_{j\leq i}(2\eta  \mathbb{I}_M +  K)^{-1}_{i,j}\left(g_{t_j}  +a_{t_j}  \mathbb{E}_{t_j }[\lambda]\right)    , \quad i=1,\ldots,M,
$$
where 
\be
\begin{aligned}
a_{t_j} &=   -1+\sum_{k=1}^M  ( \wt G^{\top} (D^{(j)} )^{-1})_{j,k}1_{\{t_j \leq   t_k\}}, \\
g_{t_j} &=  -A_{t_j}+ \sum_{k=1}^M  ( \wt G^{\top} (D^{(j)})^{-1})_{j,k}  1_{\{t_j \leq  t_k\}}   \mathbb{E}_{t_j}[ A_{t_k}], 
\end{aligned} 
\ee
and, for $r = 1, \ldots, M-1$, 
$$
\mathbb{E}_{t_{r+1}}[\lam]= \frac{(b_r-h_r)\E_{t_{r}}[\lam]- (c_{r+1}-c_r)}{b_{r+1}} ,  \quad   \mathbb{E}_{t_1}[\lam] = \frac{x_0-c_1}{b_1}, 
$$
with 
\begin{align*} 
& \wt A_{t_i} = ((2\eta  \mathbb{I}_M +  K)^{-1}  g)_i, \quad 
c_r =   \sum_{i=1}^M \E_{t_r} [\wt A_{t_i}] ,   \\
& b_{r}  =\sum_{i=1}^M \sum_{r  \leq j \leq M} (2\eta  \mathbb{I}_M +  K)^{-1}_{i,j} a_{t_j}, \quad  h_r = \sum_{i=1}^M(2\eta  \mathbb{I}_M +  K)^{-1}_{i,r} a_{t_r}. 
\end{align*} 
 
\end{theorem} 
The proof of Theorem \ref{prop-opt-strat} is given in Section \ref{sec-pf-opt}. 
 
\begin{remark} \label{rem-sol-no-sig} 
In the case where the signal $A\equiv0$ the execution problem reduces to the deterministic optimization problem which was presented in Proposition 1 of  Alfonsi et al. \cite{alf-sch12}. The solution to the problem in this special case is a straightforward generalization of equation (5) therein and it is given by 
 \begin{align*}
    u^{\star} =  \frac{x_0}{\boldsymbol{1}^{\top}(G + G^{\top} )^{-1} \boldsymbol{1}}(G + G ^{\top})^{-1}\boldsymbol{1},
\end{align*}
which is well-defined since $G  + G^{\top}$ is symmetric positive definite hence  is invertible. 
 \end{remark}

\begin{remark} 
The proof of Theorem \ref{prop-opt-strat} generalises the methods that were recently developed in \citep{AJ-N-2022,ANV-23} in order to tackle such non-Markovian problems as \eqref{opt-prog} in two directions. First, we provide a framework for solving problems with constraints by introducing a stochastic Lagrange multiplier. Moreover we manage to solve the non-regularized problem, that is, without the linear temporary price impact term $\lam u_{t_i}$ in the right-hand side of  \eqref{eq:price_dynamics_G_star} (compare with equation (2.4) in \cite{AJ-N-2022}). This allows us to study the propagator model in its original form as introduced in Section 3.1 of \cite{bouchaud-gefen}.    
\end{remark} 
\subsection{Main results on price impact estimation with offline data}  \label{sec-data} 
 In this section, we provide an offline estimator for the propagator $G$ in \eqref{eq:price_dynamics_G_star} based on a static dataset of historical trades and prices and  derive optimal convergence rates for this estimator. In order to tackle the estimation problem, we assume that the agent has access to an offline dataset, which consists of price trajectories and metaorders across several episodes.
The precise properties of the  dataset are given below. 
\begin{definition}[Offline dataset]\label{def:offline_data}
    Let  $N\in \mathbb{N}$ be  the number of episodes in the dataset and 
    consider  
$$\mathcal{D} =  \left\{(S^{(n)}_{t_i})_{i=1}^{M+1}, (u^{(n)}_{t_i})_{i=1}^{M}, (A^{(n)}_{t_i})_{i=1}^{M}
\mid     n=  1, \ldots, N\right\},
$$
where $S^{(n)} \coloneqq (S^{(n)}_{t_i})_{i=1}^{M+1}$, $u^{(n)}  \coloneqq (u^{(n)}_{t_i})_{i=1}^{M}$
and $A^{(n)} \coloneqq (A^{(n)}_{t_i})_{i=1}^{M}$ are price trajectories, trading strategies and   trading signals available in   the $n$-th episode, respectively.  

We call $\mathcal{D}$   an offline dataset for  \eqref{eq:price_dynamics_G_star}
if   $(S^{(n)}, u^{(n)}, A^{(n)})_{n=1}^N$ are realisations  of random variables 
  defined on a probability space $(\Omega',\mathcal{F}',\mathbb{P}') $
  satisfying the following properties:
for each $n\in \mathbb{N}$, 
$u^{(n)}$ is measurable with respect  to the $\sigma$-algebra $\mathcal{G}_{n-1}$  
where 
$$
\mathcal{G}_{n-1}=\sigma\left\{(S^{(k)})_{k=1}^{n-1},(u^{(k)})_{k=1}^{n-1}, (A^{(k)})_{k=1}^n\right\},
$$
and 
there exist $\mathcal G_n$-measurable random variables
$\eps^{(n)}=(\varepsilon^{(n)}_{t_i})_{  i=1}^{M}$ such that, 
\begin{equation}
\label{eq:data_price_n}
S^{(n)}_{t_{i+1}} -S^{(n)}_{t_1} =  \sum_{j =1}^{i} G^\star_{i,j}  u^{(n)}_{t_j} + A^{(n)}_{t_i}+  \varepsilon^{(n)}_{t_i} ,
\quad i=1,\ldots, M,    
\end{equation}
where $G^\star $ is the true (unknown) propagator and
$\mathbb{E}^{\sP'}[\eps_{t_i}^{(n)}\mid \mathcal{G}_{n-1}]=0$ for all $i=1,...,M$. 
 \end{definition}

\begin{remark}
\label{rmk:data_set}
Definition \ref{def:offline_data} accommodates dependent trading strategies $(u^{(n)})_{n=1}^N$ that may not optimize the control objective \eqref{per-func}. This setting is relevant for various practical scenarios such as crowded markets where agents with different objective functionals follow similar signals (see \citep{C-N-M-2023, Mich-Neum-MK,N-V-2021}). Moreover, these strategies may be generated adaptively in the data collecting process, meaning that the trader has incorporated historical observations and currently observed signals in order to design the trading strategy $u^{(n)}$ for the $n$-th episode.
\end{remark}

\paragraph{Convention.} For any $v, u \in \mathbb R^M$ we denote by $\langle v,u \rangle$ the inner product in $\mathbb R^M$ and by  $ \|v\|$ the Euclidean norm. For any matrix $B = (B_{i,j})^{M}_{i,j =1} $ we denote by $\|B\|_{\mathbb{R}^{M\times M}}  $ its  
Frobenius norm, i.e., 
\be \label{frob}
\|B\|_{\mathbb{R}^{M\times M}} = \sqrt{\sum_{i=1}^M\sum_{j=1}^M |B_{i,j}|^2} = \sqrt{\tr (B^\top B)},
\ee
where $\tr(B)$ represents the trace of $B$. We further define the Frobenius inner product,
\be \label{frob-in}
\langle A,B\rangle_{\mathbb{R}^{M\times M} } =\tr (A^\top B), \quad A,B \in \mathbb{R}^{M\times M}. 
\ee
We define the following class of admissible propagators. 
    The true propagator $G^\star \in \mathbb{R}^{M\times M}$ 
is assumed to be in the following set  with some  known constant $\kappa>0$:
\be  \label{assum:volterra_kernel}
\mathscr{G}_{\textrm{ad}} \coloneqq 
\left\{ 
G=(G_{i,j})_{i,j=1}^M
\,\middle\vert\, 
\begin{aligned}
& 
\textnormal{$x^\top (G+G^\top)x \ge  \kappa \|x\|^2 $  for all $x\in \sR^M$,}
\\
&
\textnormal{$G_{i,j}\ge 0$  for all $ i, j = 1, \ldots, M$,}
\\
&
\textnormal{$G_{i,j}=0$  for all $1\le i< j\le M$.}
 \end{aligned}
\right\}.
 \ee

\begin{remark}
Note that the conditions in $\mathscr{G}_{\textrm{ad}}$ agree with the assumptions on the propagator which  were made in \eqref{g-def}. Indeed the constant $\kappa$ coincides with $2\eta$ in \eqref{g-dec} and it can be regarded as a lower bound on the temporary price impact. The  
lower diagonal structure of $G$ ensures non-anticipative structure (see \eqref{eq:data_price_n}). The fact that the entries of $G$ are nonnegative ensures that a sell (buy) transaction, which causes a negative (positive) change in the inventory, will push the price downwards (upwards) in \eqref{eq:data_price_n}.

\end{remark}

\begin{remark}
We present a couple of typical examples for price impact Volterra kernels, whose projection on a finite grid belongs to $\mathscr{G}_{\textrm{ad}}$. More examples for convolution kernels are discussed in Remark \ref{rem-examples}. 
 The following non-convolution kernel was introduced in order to model price impact in bond trading (see Section 3.1 of  \cite{brigo2020}): 
 $$
G(t,s) = f(T-t) K(t-s)  \mathrm{1}_{\{s<t\}},
$$
where $K:\mathbb{R}_{+} \to \mathbb{R}_+$ is a usual nonnegative definite decay kernel and $f$ is a bounded function satisfying $f(0)=0$, due to the terminal condition on the bond price.  
In another example which was proposed in \cite{FruthSchoenebornUrusov} for modelling order books with time-varying liquidity, the propagator is given by $G(t,s) = L(t) \exp \left(-\int_s^t \rho(u) du \right)$, where $L,\rho$ are bounded strictly positive functions. \end{remark}

In the following assumption we classify the noise in \eqref{eq:data_price_n} as conditionally sub-Gaussian. 
\begin{assumption} \label{assum:offline_data}
The agent  has an offline dataset $\mathcal{D}$   
of size $N\in \mathbb{N}$ as in Definition \ref{def:offline_data} and
 there exists a known constant $R>0$
 such that for all $n=1,\ldots, N$,
 $$
 \mathbb{E}^{\sP'}\left[\exp\left({\langle v, \eps^{(n)} \rangle}\right)\mid \mathcal{G}_{n-1} \right]
\le \exp\left( \frac{R^2\|v\|^2}{2}\right),
\quad \textrm{for all } v\in \mathbb{R}^M.
$$
That is, 
   $\eps^{(n)}$ in \eqref{eq:data_price_n} is $R$-conditionally sub-Gaussian with respect to $\mathcal{G}_{n-1}$.
\end{assumption}

Based on the dataset $\mathcal{D}$, 
 the unknown price impact coefficients $(G^\star_{i,j})_{1\le j\le i\le M}$ 
 can be estimated via 
   a least-squares method. Note that for any $n\in \mathbb{N}$,
the observed data $y^{(n)} = (S^{(n)}_{t_{i+1}}-S^{(n)}_{t_1}-A^{(n)}_{t_i})_{i=1}^{M}\in \mathbb{R}^M$ and 
$u^{(n)}= (u^{(n)}_{t_i})_{i=1}^{M}\in \mathbb{R}^M$
satisfy
\begin{equation}
\label{eq:price_impact_linear_regression}
y^{(n)} =G^\star u^{(n)} +\eps^{(n)} , 
\quad \textnormal{with $\eps^{(n)}= (\eps^{(n)}_{t_i})_{i=1}^{M}$.}
\end{equation}
This motivates us to estimate $G^\star$ by minimising 
the following 
   quadratic loss over all admissible price impact coefficients:   
\begin{equation} \label{eq:G_n_lambda_volterra}
G_{N,\lambda}\coloneqq\argmin_{G \in \mathscr{G}_{\textrm{ad}}}
\left(
\sum_{n=1}^N \|y^{(n)}-G u^{(n)}\|^2+\lambda\|G\|_{\mathbb{R}^{M \times M}}^2
\right),
\end{equation}
where 
 $\lambda>0$
 is a given regularisation parameter and 
$\mathscr{G}_{\textrm{ad}} $ 
is given in \eqref{assum:volterra_kernel}.

\begin{remark} 
The regularisation parameter  $\lambda>0$ ensures that the loss function \eqref{eq:G_n_lambda_volterra} is strongly convex in $G$. This along with  the fact that $\mathscr{G}_{\textrm{ad}}$ is nonempty, closed and convex implies that 
${G}_{N,\lambda}$ is uniquely defined. The necessity of such regularisation can be clearly seen later in \eqref{eq:G_unconstraint_minimiser}, as it ensures the invertibility of the matrix on the right-hand side of the equality. The need for this regularisation is often ignored in the propagator estimation literature (see e.g., Section 13 of \cite{bouchaud_bonart_donier_gould_2018}). 
\end{remark}   
The estimator $G_{N,\lambda}$ 
  can be computed by projecting the  unconstrained least-squares estimator onto the   set $\mathscr{G}_{\textrm{ad}}$.
Indeed, if one sets   
\begin{align}
\label{eq:G_unconstraints}
\begin{split}
\tilde{G}_{N,\lambda}
&\coloneqq\argmin_{G \in \mathbb{R}^{M \times M}}
\left(
\sum_{n=1}^N \|y^{(n)}-G u^{(n)}\|^2+\lambda\|G\|_{\mathbb{R}^{M \times M}}^2
\right), 
\end{split}
\end{align}
then 
\begin{equation} \label{eq:project_admissible}
G_{N,\lambda}
=\argmin_{G \in \mathscr{G}_{\textrm{ad}}}\left\|\left(G-\tilde{G}_{N,\lambda}\right) V_{N,\lambda}^
{\frac{1}{2}}\right\|^2_{\mathbb{R}^{M \times M}},
\quad \textnormal{with
$V_{N,\lambda}=\sum_{n=1}^N u^{(n)} (u^{(n)})^{\top } +\lambda \mathbb{I}_{M }$}, 
\end{equation}
where we recall that $\mathbb{I}_{M}$ denotes the $M \times M$ identity matrix. 
The unconstrained problem \eqref{eq:G_unconstraints}
can be solved explicitly as
\begin{equation}
 \label{eq:G_unconstraint_minimiser}  
 \tilde{G}_{N,\lambda}= \sum_{n=1}^N y^{(n)} (u^{(n)})^{\top} V_{N,\lambda}^{-1},
\end{equation}
while the  optimisation problem  \eqref{eq:project_admissible} can be    efficiently solved by  standard  convex optimisation packages (see e.g., \cite{diamond2016cvxpy}).

The following theorem provides a confidence region of the estimated  coefficient ${G}_{N,\lambda}$ in terms of the observed data. 

\begin{theorem}\label{thm:confidence_volterra}
Assume that $G^\star \in \mathscr G_{ad}$ and that Assumption \ref{assum:offline_data}  holds.
    For any $\lambda>0$, let 
    ${G}_{N,\lambda}$ and $V_{N,\lambda}$ be defined in 
    \eqref{eq:G_n_lambda_volterra}.
    Then for all $\lambda>0$ and $\delta \in (0,1)$,
    with $\sP'$-probability at least $1-\delta$,
    \be \label{bnd-err-1} 
\begin{aligned}
&
\left\|({G}_{N,\lambda}  - G^\star)V_{N,\lambda}^{\frac{1}{2}}   \right\|_{\mathbb{R}^{M \times M}}
\le  R
\left(
\log\left(
\frac{\lambda^{-M^2} (\det(V_{N,\lambda} ))^{M}}{\delta^2}
\right)\right)^{\frac{1}{2}}
+\lambda \left \|    G^\star
V_{N,\lambda}^{-\frac{1}{2}}
    \right\|_{\mathbb{R}^{M \times M}}. 
\end{aligned}
\ee
 \end{theorem}
The proof of Theorem \ref{thm:confidence_volterra} is given in Section \ref{sec-pf-est}. 

\begin{remark}
Theorem \ref{rmk:data_set} 
quantifies     the accuracy of the estimator 
${G}_{N,\lambda}$  with an explicit dependence   on the given dataset $\mathcal{D}$  and 
the magnitude of $G^\star$.
The result is applicable to a general discrete-time Volterra propagator $G^\star$ and non-Markovian and correlated observations (as stated in Remark \ref{rmk:data_set}). 
It is worth noting  that this error estimate substantially improves the result presented in \cite[Theorem 2.10]{neuman2023statistical}, as the latter  assumes the propagator to be a (continuous-time) convolution kernel and requires the observed  trading speeds $(u^{(n)})_{n=1}^N$ to be fixed (i.e., $u^n=u^m$ for all $n,m =1, \ldots, M$) and deterministic vectors.
\end{remark} 
Examples where the Volterra price impact kernel $G^\star$ is in fact a convolution kernel arise from empirical studies and are quite popular in the literature (see e.g., \cite{bouchaud-gefen, gatheral2010no,Ob-Wan2005}). In the sequel, we incorporate this structural property of the price impact coefficient to enhance our estimation procedure.
We therefore assume that for the true price impact kernel $G^\star$ there exists   
    $K^\star = (K^\star_i)_{i=0}^{M-1} $ such that 
    \begin{equation}\label{eq:price_impact_kernel}
 G^\star = \begin{pmatrix}
     K^\star_0    & 0 & \cdots & 0\\
    K^\star_1     & K^\star_0   & \cdots & 0\\
     \vdots & \vdots & \ddots & \vdots\\
     K^\star_{M-1 }     &  K^\star_{M-2 }     & \cdots  &  K^\star_{0 }
 \end{pmatrix},
\end{equation}
and $K^\star  $ 
is in the 
following  class $ \mathscr{K}_{\textrm{ad}}$  of   admissible  convolution  kernels:
\begin{equation} \label{assume:convolution_kernel} 
\mathscr{K}_{\textrm{ad}} \coloneqq  
\left\{ 
K=(K_{i})_{i=0}^{M-1}
\,\middle\vert\, 
\begin{aligned}
&\textnormal{$K_i-K_{i-1} \le K_{i+1}-K_{i} \le 0$  for all $1\le i \le M-2$,}
\\
&\textnormal{and the associated $G$ defined by \eqref{eq:price_impact_kernel} is in $\mathscr{G}_{\textrm{ad}}$.}
 \end{aligned}
\right\}.
\end{equation}

 \begin{remark}
Note that \eqref{assume:convolution_kernel}
 assumes that the price impact $G^\star_{i,j}$ is determined by a convex and decreasing kernel evaluated at equally distributed time grids.
 Specifically this means that there exists a function  
    $K^\star:\mathbb{R}_{+} \to  \mathbb{R}_{+}$, called a resilience function,  
     such that  
    $G^\star_{i,j}=K^\star(t_{i}-t_j)= K^\star(t_{i-j})$  (with $t_0=0$)  for all $1\le j\le i\le M$, then it is easy to see that the first constraint in \eqref{assume:convolution_kernel}
 holds 
if $K^\star$ is convex and decreasing.   
\end{remark}   

\begin{remark}\label{rem-examples} 
We present some typical examples for price impact convolution kernels, whose  projection on a finite grid belongs to $\mathscr{K}_{\textrm{ad}}$. In \cite{bouchaud-gefen, gatheral2010no} among others, the following kernel was introduced:   
$ K(t)=\frac{\ell_{_{0}}}{(\ell_{0}+t)^{\beta}}$, for some constants $\beta, \ell_0>0$. The example of $K(t)=\frac{1}{t^{\beta}}1_{\{t>0\}}$ for  $0< \beta < 1/2,$ were proposed by Gatheral in \cite{gatheral2010no}.  
 The case where $K (t) =  e^{-\rho t}$, for some constant $\rho >0$, was proposed by Obizhaeva and Wang \cite{Ob-Wan2005}.  
\end{remark}

The convolution structure in \eqref{eq:price_impact_kernel} simplifies the estimation of the matrix $G^\star\in \mathbb{R}^{M \times M} $ to the estimation of  $K^\star \in \mathbb{R}^{M}$. Let   $\mathcal{D} =  \{(S^{(n)}_{t_i})_{i=1}^{M+1}, (u^{(n)}_{t_i})_{i=1}^{M}, (A^{(n)}_{t_i})_{i=1}^{M}
\mid     n=  1, \ldots, N\}$  be a given dataset.
By \eqref{eq:price_impact_kernel} and \eqref{eq:data_price_n}, we have for any 
   $n=1,\ldots, N$, 
\begin{align}
\label{eq:price_convolution_model}
S^{(n)}_{t_{i+1}}-S^{(n)}_{t_1}-A^{(n)}_{t_i}
&=
\sum_{j=1}^i K^\star_{i-j }       u^{(n)}_{t_j} +\eps^{(n)}_{t_i}   
= \sum_{j=0}^{i-1} K^\star_{j }         u^{(n)}_{t_{i-j}} +\eps^{(n)}_{t_i}, 
\quad  i=1,\ldots, M,
\end{align}
which can be equivalently  written as  $y^{(n)} =U_n K^\star +\eps^{(n)} $,  
with 
\begin{equation} \label{eq:y_n_U_n}
    y^{(n)} = 
\begin{pmatrix}
S^{(n)}_{t_{2}}-S^{(n)}_{t_1}-A^{(n)}_{t_1} 
\\
\vdots
\\
S^{(n)}_{t_{M+1}}-S^{(n)}_{t_1}-A^{(n)}_{t_M} 
\end{pmatrix},
\;
U_n = \begin{pmatrix}
     u^{(n)}_{t_1} & 0 & \cdots & 0\\
    u^{(n)}_{t_2} & u^{(n)}_{t_1} & \cdots & 0\\
     \vdots & \vdots & \ddots & \vdots\\
     u^{(n)}_{t_M} & u^{(n)}_{t_{M-1}} & \cdots  &  u^{(n)}_{t_1}
 \end{pmatrix},
 \;
 \eps^{(n)} = 
\begin{pmatrix}
\eps^{(n)}_{t_1}
\\
\vdots
\\
\eps^{(n)}_{t_M}
\end{pmatrix}.
\end{equation}

The data series of $S^n$ in \eqref{eq:price_convolution_model} motivates us to consider the following constrained least-squares estimator:
\begin{equation}
\label{eq:K_n_lambda}
K_{N,\lambda}\coloneqq\argmin_{K \in \mathscr{K}_{\textrm{ad}}}
\left(
\sum_{n=1}^N \|y^{(n)}-U_n K \|^2+\lambda\|K\|^2
\right),
\end{equation}
where 
 $\lambda>0$
 is a given regularisation parameter,
and 
$\mathscr{K}_{\textrm{ad}} $ 
is given in \eqref{assume:convolution_kernel}.
Similar to $G_{N,\lambda}$ in \eqref{eq:G_n_lambda_volterra},
$K_{N,\lambda}$
can be computed by projecting the unconstrained least-squares estimator onto the parameter set $\mathscr{K}_{\textrm{ad}}$:
\begin{align}
\label{eq:project_admissible_conv}
    K_{N,\lambda} 
=\argmin_{K \in \mathscr{K}_{\textrm{ad}}}\left\| W_{N,\lambda}^
{\frac{1}{2}}\left(K-\tilde{K}_{N,\lambda}\right)\right\|^2,
\end{align} 
where 
\be \label{w-def}
W_{N,\lambda}\coloneqq \sum_{n=1}^N U^\top_n U_n  +\lambda \mathbb{I}_{M },
\quad 
\tilde{K}_{N,\lambda} \coloneqq W_{N,\lambda}^{-1} \sum_{n=1}^N U^\top_n y^{(n)}.
\ee
with $y^{(n)}\in \mathbb{R}^M$ and $U_n\in \mathbb{R}^{M \times M}$ defined in \eqref{eq:y_n_U_n}.
Given $K_{N,\lambda}=((K_{N,\lambda})_i)_{i=0}^{M-1}$,
the associated propagator 
$G_{N,\lambda}$
is then defined  according to \eqref{eq:price_impact_kernel}: 
\begin{equation}
\label{eq:G_n_lambda_convolution}
 G_{N,\lambda} = \begin{pmatrix}
     (K_{N,\lambda})_0    & 0 & \cdots & 0\\
    (K_{N,\lambda})_1     & (K_{N,\lambda})_0   & \cdots & 0\\
     \vdots & \vdots & \ddots & \vdots\\
     (K_{N,\lambda})_{M-1 }     &  (K_{N,\lambda})_{M-2 }     & \cdots  &  (K_{N,\lambda})_{0 }
 \end{pmatrix}.   
\end{equation}

The following theorem   is analogue to Theorem \ref{thm:confidence_volterra} and 
provides a confidence region of the estimator ${K}_{N,\lambda}$.  

\begin{theorem}\label{thm:confidence_convolution}
Suppose Assumption \ref{assum:offline_data} holds and that $K^{\star}\in \mathscr{K}_{\textrm{ad}}$.
    For each $\lambda>0$, let 
    ${K}_{N,\lambda}$ be defined in 
    \eqref{eq:K_n_lambda}.
    Then for all $\lambda>0$ and $\delta \in (0,1)$,
    with $\sP'$-probability at least $1-\delta$,
    \be \label{diff-k} 
\begin{aligned}
&
\left\|W_{N,\lambda}^{\frac{1}{2}} ({K}_{N,\lambda}  - K^\star)  \right\|
\le  R
\left(
\log\left(
\frac{\lambda^{-M}  \det(W_{N,\lambda} ) }{\delta^2}
\right)\right)^{\frac{1}{2}}
+\lambda \left \| 
W_{N,\lambda}^{-\frac{1}{2}}
K^\star
    \right\|,    
\end{aligned}
\ee
with  $W_{N,\lambda} =\sum_{n=1}^N U_n^\top  U_n+\lambda\mathbb{I}_M$. 
\end{theorem}
The proof of Theorem \ref{thm:confidence_convolution} is given in Section \ref{sec-pf-est}. 

\begin{remark} \label{rem-con-rate} 
By imposing the structural property \eqref{eq:price_impact_kernel}, we achieve  in Theorem  \ref{thm:confidence_convolution} a better dependence in the number of trading periods $M$ compared to Theorem \ref{thm:confidence_volterra}. Specifically compare the power of $\lam >0$ in the first term on the right-hand side of \eqref{bnd-err-1} to the corresponding term in \eqref{diff-k}.  
 We present in Remark \ref{rem-asump-mat}  some additional important advantages for using a convolution kernel in the context of the regret analysis of the pessimistic control problem.  See Figure \ref{FIGURE LABEL_d}  for a numerical comparison of the estimators \eqref{eq:G_unconstraints}
 and  \eqref{eq:K_n_lambda}.
 \end{remark} 

\subsection{Pessimistic control design with offline data} \label{sec-pessim}
In the following two subsections, we propose a stochastic control framework that will eliminate the spurious correlation, which was defined in \eqref{sub-decom}, between the offline estimator ${G}_{N,\lambda}$ and a greedy trading strategy based on ${G}_{N,\lambda}$,  and derive a tight bound on the associated performance gap. 
 
An important consideration in designing trading strategies arise from the fact that pre-collected data may not provide a uniform exploration of the parameter space, and hence  certain entries of the unknown propagator $G^\star$ may have been estimated with limited accuracy. 
Consequently, if a trading strategy is solely based on the estimated propagator ${G}_{N,\lambda}$ in \eqref{eq:G_n_lambda_volterra}, its performance in real deployment may be suboptimal. To overcome this challenge, in this section we introduce an additional loss term that takes into account the statistical uncertainty of the estimated propagator. This is often referred to as the principle of pessimism in the face of uncertainty in the offline reinforcement learning literature (see e.g., \cite{jin2021pessimism,chang2021mitigating}).

We recall that the offline dataset $\mathcal{D}$ from Definition \ref{def:offline_data} is built out of $N$ trajectories of prices, trading speeds and signals $(S^{(n)}, u^{(n)}, A^{(n)})_{n=1}^N$  which are defined on the probability space $(\Omega',\mathcal{F}', \mathbb{P}') $. In Section \ref{sec-liq-model}  we have specified the filtered probability space $(\Omega, \mathcal F, \mathbb F=\{\mathcal F_{t_i}\}_{i=1}^M, \mathbb{P})$ in which the optimal liquidation problem \eqref{opt-prog} takes place. We define the following product space
\begin{equation*}
(\ol{\Omega}, \ol{\mathcal{F}}, \ol{\mathbb{P}})
= (\Omega' \times\Omega,\mathcal{F}'\times\mathcal{F}
        ,\mathbb P' \times \mathbb P).
\end{equation*}
We further define {\color{black}$\mathbb P_{\mathcal{D}}$} to be the probability measure on the product space conditioned on a realisation of $\mathcal{D}$ and $\E_{\mathcal D}$ as the corresponding conditional expectation. 

Next we introduce the pessimistic stochastic control problem in the spirit of \eqref{j-h-def} that corresponds to the optimal liquidation problem \eqref{opt-prog}.  
For the remainder of this section we assume that $N$ and the dataset $\mathcal D$ in \eqref{def:offline_data} are fixed and that for any $\lam >0$, $V_{\lam,N}$ is defined as in \eqref{eq:project_admissible}.

Let $L^\mathcal{A}_{\eqref{l-bnd}},    L^\mathscr{G}_{\eqref{l-bnd}} \geq 1$. We define a subclass of the admissible strategies $\mathcal A$ in \eqref{admiss} and admissible kernels $\mathscr{G}_{ad}$ in \eqref{assum:volterra_kernel} as follows: 
\be \label{l-bnd}
\mathcal{A}_{b} = \left\{ u \in \mathcal{A}: \black{ \|u \| }\leq L^\mathcal{A}_{\eqref{l-bnd}} \right\},
\quad \mathscr{G}_{b} = \left\{ G \in \mathscr{G}_{ad}: \|G \|_{\mathbb{R}^{M \times M}} \leq  L^\mathscr{G}_{\eqref{l-bnd}} \right\}. 
\ee
\begin{remark} 

We choose $L^\mathcal{A}_{\eqref{l-bnd}}$ to be large enough so that $\mathcal{A}_{b}$ contains strategies satisfying  the fuel constraint of \eqref{opt-pes}. The constant  $L^\mathscr{G}_{\eqref{l-bnd}}$ is chosen so that $ G^\star \in \mathscr{G}_{b}$. Since $\mathscr{G}_{b}$ provides a radius for the parameter space, using Theorem \ref{prop-opt-strat} and our assumption that the signal $A$ is known to the trader, we can also choose $L^\mathcal{A}_{\eqref{l-bnd}}$ such that the optimal strategy \eqref{prop-opt-strat} with $G^\star$, is in the set $\mathcal{A}_{b}$.
\end{remark}

Now, for an arbitrary admissible control strategy $u \in\mathcal{A}_{b}$, we define the following penalization functional  
 \begin{align}\label{fun_regularization}
   \ell_1(V_{N,\lambda},u) \coloneqq L^\mathcal{A}_{\eqref{l-bnd}}C_{\eqref{c-const}}(N) 
   \left\|
   V_{N,\lambda}^{-\frac{1}{2}} u
   \right\|, 
\end{align}
where $C_{\eqref{c-const}} (N) = C_{\eqref{c-const}}(N,M,\delta,\lambda,   R, L^\mathscr{G}_{\eqref{l-bnd}}, V_{N,\lambda})$ bounds the estimation error from Theorem \ref{thm:confidence_volterra} and it is given by, 
\be \label{c-const} 
\begin{aligned}
C_{\eqref{c-const}} (N)
\coloneqq  R \left( 
\log\left(
\frac{\lambda^{-M^2} (\det(V_{N,\lambda}))^{M}}{\delta^2}
\right)\right)^{\frac{1}{2}}
+\lambda   L^\mathscr{G}_{\eqref{l-bnd}}  \left \|  
  V_{N,\lambda}^{-\black{1/2}}  \right\|_{\mathbb{R}^{M \times M}}.
\end{aligned} 
\ee 
  Here we have used the fact that for any two matrices $A,B \in \mathbb{R}^{M \times M}$ we have in the Frobenius norm
\be \label{frob-id}
\|A  B \|_{\mathbb{R}^{M \times M}}  \leq \|A  \|_{\mathbb{R}^{M \times M} } \| B \|_{\mathbb{R}^{M \times M}}.
\ee
%
%
The regularized cost functional for the pessimistic Volterra liquidation problem is given by, 
\begin{equation} \label{j-p1} 
  {J}_{P,1}(u;G_{N,\lambda}) \coloneqq \mathbb{E}_{\mathcal D} \left[ \sum_{i=1}^{M}   \sum_{j=1 }^{i} (G_{N,\lambda})_{i,j}  u_{t_j} u_{t_i}   + \sum_{i=1}^M A_{t_i} u_{t_i} + \ell_1(V_{N,\lambda},u) \right].
\end{equation}
The cost functional \eqref{j-p1} replaces the unknown propagator in \eqref{per-func} with the estimated propagator $G_{N,\lambda}$ and includes an additional cost term $\ell_1$  to discourage the agent from taking actions that are not supported by the offline data. We identify the expectation of $\ell_1$ as a $\dl$-uncertainty quantifier (see \eqref{ell-1-def}). This statement will be made precise in Section \ref{sec-res-pess}. The penalty strength is determined by the inverse covariance matrix $V^{-1/2}_{N,\lambda}$, where a larger $V^{-1/2}_{N,\lambda}$ implies higher uncertainty in the estimated model $G_{N,\lambda}$ and a stronger penalty $\ell_1$.  As we will show in Section \ref{sec-res-pess}, this pessimistic loss yields a  trading strategy that can compete with any other admissible strategy, with regret depending explicitly on the offline dataset.

In the case where the price impact kernel $G^{\star}$ is given by a convolution kernel as in \eqref{eq:price_impact_kernel},
the impact of a   trading speed $u = (u_1,\ldots,u_M)^\top  \in \mathbb R^{M}$
on the price  $S$ includes the following matrix $U\in \sR^{M\times M}$ (see \eqref{eq:price_convolution_model}):
\be \label{t-op} 
   U\coloneqq  \begin{pmatrix}
     u_{1} & 0 & \cdots & 0\\
    u_{2} & u_{1} & \cdots & 0\\
     \vdots & \vdots & \ddots & \vdots\\
     u_{M} & u_{M-1} & \cdots  &  u_{1}
 \end{pmatrix}. 
\ee
Similarly to $\mathscr{G}_{b}$ in \eqref{l-bnd}, we define the following subclass of convolution kernels, recalling that $K$ is a vector, 
\be \label{k-bnd}
\mathscr{K}_{b} = \left\{ K \in \mathscr{K}_{ad}: \|K \| \leq  L^\mathscr{K}_{\eqref{k-bnd}} \right\}. 
\ee
We then introduce the following regularization 
\begin{align}\label{eq:cost_conv}
\ell_2(W_{N,\lambda},u) \coloneqq L^{\mathcal A}_{\eqref{l-bnd}}C_{\eqref{c-2-const} }(N)  \left \| U W_{N,\lambda}^{-1/2}   \right \|_{\mathbb{R}^{M\times M}},
\end{align}
where $W_{N,\lambda}$ was defined in \eqref{w-def} and 
$C_{\eqref{c-2-const}} (N) = C_{\eqref{c-2-const}}(N,M,\delta,\lambda,R  ,
L^\mathscr{K}_{\eqref{k-bnd}}, W_{N,\lambda})$ 
 is the estimation error from Theorem \ref{thm:confidence_convolution},
\be \label{c-2-const}
\begin{aligned}
C _{\eqref{c-2-const}} (N)  \coloneqq     R
\left(
\log\left(
\frac{\lambda^{-M}  \det(W_{N,\lambda} ) }{\delta^2}
\right)\right)^{\frac{1}{2}}
+\lambda  L^{\mathscr K}_{\eqref{k-bnd}} \left \| 
W_{N,\lambda}^{\black{-1/2}}
    \right\|_{\mathbb{R}^M \times \mathbb{R}^M} .
\end{aligned}
\ee
Recall that the estimated convolution kernel 
  $K_{N,\lambda}$ was defined in \eqref{eq:K_n_lambda}.
The regularized objective functional in the convolution case is given by, 
\begin{equation} \label{j-p2} 
  J_{P,2}(u;K_{N,\lambda}) \coloneqq \mathbb{E}_{\mathcal D}\left[ \sum_{i=1}^{M}   \sum_{j=1 }^{i} (K_{N,\lambda})_{i - j}  u_{t_j} u_{t_i}   + \sum_{i=1}^M A_{t_i} u_{t_i} + {\ell}_2(W_{N,\lambda},u) \right],
\end{equation}
where we used the relation \eqref{eq:G_n_lambda_convolution}
between  the   estimated convolution kernel 
  $K_{N,\lambda}$ and the associated 
 propagator  $G_{N,\lambda}$.
Note that since the estimated propagator is of convolution form, the cost $\ell_2$ in \eqref{eq:cost_conv} measures the impact of the statistical  uncertainty of $K_{N,\lambda}$ on a  trading speed $u$
through the corresponding matrix $U$ (compare to~$\ell_1$ in \eqref{fun_regularization} and \eqref{j-p1}). 

We define the following pessimistic optimization problems for $i=1,2$:  
\begin{equation} \label{opt-pes} 
   \begin{cases}
&\min_{u \in \mathcal A_b} J_{P,i}(u;G_{N,\lambda}), \\
&\textrm{s.t.}  \sum_{i=1}^N u_{t_i} = x_0,
  \end{cases} 
\end{equation} 

\begin{remark} 
The existence of a unique minimizer $\hat{u} \in \mathcal{A}$ for ${J}_{P,i}$ follows directly from \cite[Lemma 2.33]{bonnans2013perturbation}
and the facts that 
$\mathcal{A}\ni u\mapsto {J}_{P,i}(u, G_{N,\lambda})\in \mathbb{R}$ is strongly convex, and $\left\{u\in \mathcal{A}_b\mid \sum_{i=1}^N u_i=x_0\right\} $ is nonempty if $L^\mathcal{A}_{\eqref{l-bnd}}$ is sufficiently large. Indeed,   the fact that $G_{\lam,N}$ in \eqref{eq:G_n_lambda_volterra} is chosen from the class of admissible kernels \eqref{assum:volterra_kernel} ensures the strong convexity of $J_{P,i}$.   
\end{remark}

\subsection{Main results on performance of pessimistic strategies} \label{sec-res-pess}

The following theorem provides an upper bound on the performance gap between the pessimistic solution to \eqref{opt-pes} and any arbitrary admissible strategy including the optimal strategy from Theorem \ref{prop-opt-strat} using the true propagator $G^\star$, in terms of the mean of $\ell_i$. We also derive the asymptotic behaviour of this bound with respect to large values of the sample size $N$ in Theorem \ref{thm-l-bound}. Throughout this section we assume that the dataset $\mathcal{D}$ is fixed according to Definition \ref{def:offline_data}.

We recall that the performance functional  $J(u;G^{\star})$
is related to the 
original control problem \eqref{per-func},
and 
the performance functional  
$J_{P,i}(u;G_{N,\lambda})$ is  related to the 
pessimistic control problem  \eqref{opt-pes}  conditioned on $\mathcal{D}$. Further, $G_{N,\lambda}$ which is determined by $\mathcal D$ was defined in \eqref{eq:G_n_lambda_volterra} and \eqref{eq:G_n_lambda_convolution}. The class of admissible Volterra kernels $\mathscr{G}_{ad}$ was defined in \eqref{assum:volterra_kernel},
and the subclass of convolution kernels with kernels in $\mathscr{K}_{ad} $, was defined in \eqref{assume:convolution_kernel}.

\begin{theorem}\label{THM:performance_bound}
Let $\hat{u}^{(i)}$ be the minimizer of $J_{P,i}(u;G_{N,\lambda})$, $i=1,2$.  Then under Assumption \ref{assum:offline_data}, for all $\lambda>0$ and $\delta \in (0,1)$ we have with $\mathbb P'$-probability at least $1-\delta$,  
\begin{itemize} 
\item[\textbf{(i)}] if $G^{\star} \in \mathscr{G}_{ad} $ is a Volterra kernel,  
$$
J(\hat u^{(1)};G^{\star})- J(u;G^{\star}) \leq  2 \mathbb{E}_{\mathcal D}  \left[  \ell_1(V_{N,\lambda},u) \right], \quad \textrm{for all } u\in \mathcal{A}_{b}, 
$$
\item[\textbf{(ii)}] 
if $G^{\star} $ is a convolution kernel with $K^\star \in  \mathscr{K}_{ad}$,  
$$
J(\hat u^{(2)};G^{\star})- J(u;G^{\star}) \leq  2 \mathbb{E}_{\mathcal D}  \left[  \ell_2(W_{N,\lambda},u) \right], \quad \textrm{for all } u\in \mathcal{A}_{b}. 
$$
\end{itemize} 
\end{theorem}

One of the by-products of the proof of Theorem \ref{THM:performance_bound} is the fact that $\Gamma(u)=  \mathbb{E} \left[  \ell_i(W_{N,\lambda},u) \right]$ is a $\dl$-uncertainty quantifier in the sense of Definition \ref{def-quant}. 

\begin{corollary} \label{cor-quant}
Under Assumption \ref{assum:offline_data}, for all $\lambda>0$ and $\delta \in (0,1)$ we have with $\mathbb P'$-probability at least $1-\delta$,  
\begin{itemize}
\item[\textbf{(i)}] if $G^{\star} \in \mathscr{G}_{ad} $ is a Volterra kernel,  
$$
| J( u; G^{\star}) - J(   u ; \hat G)| \leq\mathbb{E}_{\mathcal D} \left[  \ell_1(W_{N,\lambda},u) \right],  \quad \textrm{for all } u\in \mathcal{A}_{b},
$$
\item[\textbf{(ii)}] 
if $G^{\star} $ is a convolution kernel with $K^\star \in  \mathscr{K}_{ad}$,  
$$
| J( u; G^{\star}) - J(   u ; \hat G)| \leq \mathbb{E}_{\mathcal D} \left[  \ell_2(W_{N,\lambda},u) \right],  \quad \textrm{for all } u\in \mathcal{A}_{b}.
$$
\end{itemize} 
Hence, $\Gamma(u)=  \mathbb{E} \left[  \ell_i(W_{N,\lambda},u) \right]$ is a $\dl$-uncertainty quantifier. 
\end{corollary}
The proofs of Theorem \ref{THM:performance_bound} and Corollary \ref{cor-quant} are given in Section \ref{sec-pf-pessim}.

In order to derive a convergence rate for the error bound on the regret, which was established in Theorem \ref{THM:performance_bound}, we make the following assumptions. Recall that $(u^{(n)})_{n=1}^N$ are the strategies recorded in the dataset $\mathcal{D}$ (see Definition \ref{def:offline_data}) and that for any $u^{(n)}$, $U_n$ is the matrix defined in \eqref{eq:y_n_U_n}.   

We give the definition of the Loewner order, i.e., the partial order defined by the convex cone of positive semi-definite matrices.
\begin{definition} \label{ineq-matrix} For any two symmetric  matrices $A,B\in  \mathbb{R}^{M \times M}$ we say that $A \leq B$ if for any  $x \in \mathbb{R}^M$, we have $x^{\top}(B-A)  x \ge 0$. 
\end{definition} 
\begin{assumption}\label{assum:concentration_inequ}
For any $\delta \in (0,1)$ there exists a (known) constant $C_{\ref{assum:concentration_inequ}}(\delta)>0$ such that the following bound holds with probability $1-\delta$,  
\begin{itemize}
\item[\textbf{(i)}]   if $G^{\star} \in \mathscr{G}_{ad} $ is a Volterra kernel 
\begin{align*}
  - \frac{C_{\ref{assum:concentration_inequ}}(\delta)}{\sqrt{N}} \mathbb{I}_{M}  \leq \frac{1}{N}\sum_{n=1}^N u^{(n)} (u^{(n)})^{\top}   -\Sigma  \leq  
\frac{C_{\ref{assum:concentration_inequ}}(\delta)}{\sqrt{N}} \mathbb{I}_{M}, 
\end{align*}
\item[\textbf{(ii)}]   if $G^{\star} $ is a convolution kernel with $K \in  \mathscr{K}_{ad}$, 
\begin{align*}
  - \frac{ C_{\ref{assum:concentration_inequ}}(\delta)}{\sqrt{N}} \mathbb{I}_{M}  \leq \frac{1}{N}\sum_{n=1}^N U_n^{\top} U_n  - \hat  \Sigma  \leq  
\frac{   C_{\ref{assum:concentration_inequ}}(\delta)}{\sqrt{N}} \mathbb{I}_{M}, 
 \end{align*}
\end{itemize} 
where $ \Sigma$ and $\hat{\Sigma}$ are symmetric positive-definite matrices, not depending on $(N,\dl)$.
\end{assumption}

\begin{remark} \label{rem-asump-mat} 
There is an important advantage in verifying Assumption \ref{assum:concentration_inequ}(ii) for convolution kernels compared to Assumption \ref{assum:concentration_inequ}(i) for Volterra kernels. Note that in case (ii), $U_n^{\top} U_n$ is a product of two triangular matrices (see \eqref{eq:y_n_U_n}) which is positive definite if the first entry $u_{t_1}^{(n)} \not = 0$. This means that $N^{-1}\sum_{n=1}^N U_n^{\top} U_n$ is positive definite if $\sum_{n=1}^N  u_{t_1}^{(n)} \not = 0$, which is an event with probability $1$, if the transaction size is assumed to be continuous. In reality the transaction size is quantized, and this event will have probability asymptotically close to $1$ for large $N$. This means that normalising this sum of matrices, with respect to a symmetric positive-definite matrix $\hat{\Sigma}$ is a very natural assumption, which is expected to hold for almost any dataset of metaorders. On the other hand, in case (i) the the product $u^{(n)} (u^{(n)})^{\top}$ yields a matrix of rank $1$, hence in order for the assumption to hold the number of samples $N$ has to be much larger than the number of grid points $M$.
\end{remark}

\paragraph{Notation:}We denote by $\ul \xi_{\Sigma }$ (resp.~$\ul \xi_{\hat \Sigma}$) the minimal eigenvalue of $\Sigma$ (resp.~$\hat{\Sigma}$),
and by $\ol \xi_{\Sigma }$ (resp.~$\ol \xi_{\hat \Sigma}$) the maximal eigenvalue of $\Sigma$ (resp.~$\hat{\Sigma}$). 

In the following we show that our pessimistic strategy $\hat u^{(i)}$ is $O(N^{-1/2}(\log N)^{1/2})$ close to any competing strategy, where $N$ corresponds to the sample size of the dataset specified in Definition \ref{def:offline_data}. 

\begin{theorem} \label{thm-l-bound} 
Let $\delta \in (0,1)$ and let $\hat{u}^{(i)}$ be the minimizer of $J_{P,i}(u;G_{N,\lambda})$, $i=1,2$ with $\lam=  C_{\ref{assum:concentration_inequ}}(\delta)N^{1/2}$. Then under Assumptions \ref{assum:offline_data} and \ref{assum:concentration_inequ} we have with $\mathbb{P}'$-probability at least $1-2\delta$,
\begin{itemize} 
\item[\textbf{(i)}] if $G^{\star} \in \mathscr{G}_{ad} $ is a Volterra kernel, 
$$
J(\hat u^{(1)};G^{\star})- J(u;G^{\star}) \leq  \frac{1}{\sqrt{N}}    \left( C_1(M,\delta,\Sigma,  N)  + C_2(M,\delta,\Sigma) \right), \quad \textrm{for all } u\in \mathcal{A}_{b}, 
$$
\item[\textbf{(ii)}] if $G^{\star} $ is a convolution kernel with $K \in  \mathscr{K}_{ad}$,  
$$
J(\hat u^{(2)};G^{\star})- J(u;G^{\star}) \leq  \frac{1}{\sqrt{N}} \left( C_1(M,\delta,\hat \Sigma,N)   +   \hat C_2(M,\delta,\hat \Sigma) \right), \quad \textrm{for all } u\in \mathcal{A}_{b}, 
$$
\end{itemize} 
where $C_1(N) = O(\sqrt{\log N})$ and $C_2 = O(1)$ are given by,  
\begin{align*}
 C_1(M,\delta,\Sigma,  N) &= 2R\ul \xi_{\Sigma}^{-1/2} (L^{\mathcal A}_{\eqref{l-bnd}})^2 M \sqrt{\log  \left(N   \frac{1}{\dl^2}\left(1 + \frac{  
 \ol \xi_{\Sigma} }  {2C_{\ref{assum:concentration_inequ}}(\dl) } \right) \right)} , \\
C_2(M,\delta,\Sigma) & = 2\ul \xi_{\Sigma}^{-1}  (L^{\mathcal A}_{\eqref{l-bnd}})^2 L^\mathscr{G}_{\eqref{l-bnd}} C_{\ref{assum:concentration_inequ}}(\delta)  \sqrt{M}, \\
\hat C_2(M,\delta,\Sigma) & = 2\ul \xi_{\Sigma}^{-1}  (L^{\mathcal A}_{\eqref{l-bnd}})^2 L^\mathscr{K}_{\eqref{k-bnd}} C_{\ref{assum:concentration_inequ}}(\delta)  M.
 \end{align*}
\end{theorem}

The proof of Theorem \ref{thm-l-bound} is given in Section \ref{sec-pf-pessim}.

\begin{remark} Theorem \ref{thm-l-bound} derives upper bounds of order $N^{-1/2}(\log N)^{1/2}$ for the performance difference between the optimal pessimistic strategies $\hat u^{(i)}$ and any arbitrary strategy of a competitor having access to a similar dataset. Note that $\mathcal A_b$ also includes the optimal strategy using the unknown $G^{\star}$ from Theorem \ref{prop-opt-strat}. This framework introduces a novel approach for nonparametric estimation of financial models, which is particularly effective in the case where the quality of the common dataset is poor or it contains a relatively low number of  samples, the statistical estimators can be biased and resulting greedy strategies may create unfavorable costs. See Figure \ref{FIGURE LABEL_c} for numerical illustrations.
\end{remark} 

\begin{remark} 
Our results using the pessimistic learning approach can be compared to the well-known robust finance approach, in which strong assumptions are made on the parametric model, and the unknown parameters are assumed to be within a certain radius from the true parameters. Our nonparametric framework suggests that the radius of feasible models is in fact measured by matrix norms induced by $V_{N,\lambda}^{-1/2}$ and $W_{N,\lambda}^{-1/2}$ on the left-hand side of \eqref{bnd-err-1} and \eqref{diff-k}, respectively. In contrast to the robust finance approach theses norms are determined directly from the dataset and are not chosen as a hyperparameter (see \eqref{eq:project_admissible} and \eqref{w-def}). 
\end{remark} 

\begin{remark}
We briefly compare the results of Theorems \ref{thm-l-bound} and \ref{THM:performance_bound} to existing results in the offline reinforcement learning literature. These results typically focus on the much simpler setup of Markov decision processes (MDPs) with unknown transition probabilities and rewards. Note that stochastic control problems related to propagators as in \eqref{per-func} and \eqref{j-p1} not only take place in a continuous state space but are also non-Markovian (see e.g., \cite{AJ-N-2022}). In \cite{jin2021pessimism} some results on pessimistic offline RL with respect to minimization of the spurious correlation and intrinsic uncertainty were derived for MDPs. The bound on suboptimality in Theorem \ref{THM:performance_bound} coincides with the corresponding bound for  MDPs established in Theorem 4.2 of \cite{jin2021pessimism}. The convergence rate of order $N^{-1/2}\log N$ in Theorem \ref{thm-l-bound} is compatible with the convergence rate of order $N^{-1/2}$ established in Corollary 4.5 of \cite{jin2021pessimism} for \emph{linear} MDP, under similar assumptions as in Assumption \ref{assum:concentration_inequ}. The logarithmic correction in Theorem \ref{thm-l-bound} is a result of the estimation scheme for $G_{N,\lam}$ (see \eqref{eq:G_n_lambda_volterra}), which is subject to the regularisation of the dataset in \eqref{eq:G_unconstraint_minimiser}. This estimation procedure is completely independent from the results of \cite{jin2021pessimism} and for a sufficiently regular dataset the regularization is not needed and the logarithmic term in Theorem \ref{thm:confidence_volterra} will vanish. 


\end{remark}

\section{Numerical Illustration} \label{sec-numerics} 
In this section, we examine the performance of the propagator estimators in Section \ref{sec-data} and  the pessimistic strategies presented in Section \ref{sec-res-pess}.
Using a synthetic dataset, we illustrate the following characteristics of our methods: 
\begin{itemize}
    \item Directly estimating a Volterra propagator using \eqref{eq:G_n_lambda_volterra} 
may result in large estimation errors unless  the dataset contains  sufficiently noisy trading strategies. By imposing a convolution structure and shape constraints of the estimated model,  \eqref{eq:K_n_lambda} significantly improves the estimation accuracy, even with a smaller sample size.

\item 
Minimising the execution costs in \eqref{per-func} after substituting the estimated propagator instead of the true one, yields a greedy strategy that is very sensitive to the accuracy of the estimated model and also creates unfavorable transaction costs. The pessimistic strategy takes the model uncertainty into account and achieves more stable performance and drastically reduces the execution costs. 
\end{itemize}

We  start by describing the   construction of the synthetic dataset  $\mathcal{D}$  for our experiments. For fixed $N$-trading days, we split each trading day into $5$ minute bins. Hence, for a trading day of $6.5$ hours we have $M= 78$. We assume that the unaffected price process $P$ has the following dynamics: 
\begin{equation*}\mathrm{d}P_{t} = I_t  \mathrm{d}t + \sigma_P \mathrm{d}W^P_{t},\quad P_{0} = p_0 >0, 
\end{equation*}
where $p_0$ is 
the  initial price,
$I$ is the expected return
following    an 
Ornstein-Uhlenbeck dynamics (cf. \cite[Section 2.3]{Lehalle-Neum18}), 
\begin{equation}
\label{eq-ornstein-uhlenbeck}
\mathrm{d}I_{t} = -\mu I_{t}\mathrm{d}t + \sigma \mathrm{d}W_{t},\quad I_{0} = 0,
\end{equation}
and the values of 
 $\sigma_P$, $\mu$ and $\sigma$ are  given in Table \ref{table:1}.  
The signal $A_{t_i}$ in \eqref{eq:data_price_n} at time $t_i \in \lbrace 0=t_1<t_2<\ldots<t_M =1 \rbrace$ is then given by
\begin{align*}
    A_{t_i} = \int_{0}^{t_i} I_u \mathrm{d}u.
\end{align*}
We consider a market with 3 types of traders, trading simultaneously over one year (i.e. $N= 252$ trading days). For simplicity, we construct a dataset 
with buy strategies,
by sampling 
the target inventory $x_0$ of each type of trade uniformly from $[500,2000]$. 
Including sell strategies in our dataset will not change our estimation. 
We assume that 
the traders do not have precise information on the true price impact parameters
and  adopt    the following commonly used strategies: 
\begin{itemize} 
\item TWAP trades, 
 aiming at $x_0$ stocks and buying 
at a constant rate throughout each day: 
$$
u^1_{t_i} = \frac{x_0}{M}, \quad i=1,\ldots,M. 
$$
\item Execution strategies according to \citet{Ob-Wan2005} as described in Remark \ref{rem-sol-no-sig}, 
$$
u^2=    \frac{x_0}{\boldsymbol{1}^{\top}(G + G^{\top} )^{-1} \boldsymbol{1}}(G + G ^{\top})^{-1}\boldsymbol{1}, 
$$
where $G_{i,j}=\hat{\kappa} e^{-\hat \rho (t_i-t_j)}\mathrm{1}_{\{ i \geq j \}}$, with $\hat \rho$ and $\hat \kappa$ sampled uniformly from $[\rho/2, 3\rho/2]$ and $[\kappa/2, 3\kappa/2]$, respectively. The values of  $\rho$ and $\kappa$ are given in table \ref{table:1}. 
\item Purely trend following strategies taking into account only temporary price impact as in Remark 3.4 of \cite{Lehalle-Neum18}, 
$$
u^3_{t_i} = \frac{I_{t_i}}{2 \hat \mu \hat{\kappa}}\big(1-e^{-\hat{\mu}(T-t_i)}\big), \quad i=1,\ldots,M, 
$$  
with   $\hat \kappa$ and $\hat \mu$   sampled uniformly from $[\kappa/2, 3\kappa/2]$ and $[\mu/2, 3\mu/2]$, respectively. The values of $\mu$ and $\kappa$ are given in Table \ref{table:1}. 
 \end{itemize}

  \begin{table}[ht]
\caption{Model Parameters} 
\centering 
\begin{tabular}{c c} 
\hline\hline 
Parameter & Value \\ [0.5ex]
\hline 
Price volatility $\sigma_{P}$ & 0.0088\\
Signal volatility $\sigma$ & 0.06\\
Signal mean reversion $\mu$ & 0.1\\
Trading Cost $\kappa $ & $ 0.01 $  \\
Resilience  $\rho $ & $ 0.04$  \\
\hline 
\end{tabular}\label{table:1} 
\end{table}

Given the above three trading strategies, 
for each $n=1,\ldots, N$, 
we generate the observed price trajectories according to the following dynamics (see \eqref{eq:data_price_n}): \begin{align}
\label{eq:price_experiments}
S^{(n)}_{t_{i+1}}   &=  \sum_{j =1}^{i} G^\star_{i,j}  u^{(n)}_{t_j} + P^{(n)}_{t_i} =  \sum_{j =1}^{i} G^\star_{i,j}  u^{(n)}_{t_j} + p_0 + A^{(n)}_{t_i} + \sigma_{P} W_{t_i},
\quad i=1,\ldots, M,    
\end{align}
where  $(u^{(n)}_{t_j})_{j=1}^M$ is a realisation of  
$u_{t_i} \coloneqq  u^1_{t_i}+   u^{2}_{t_i}+  u^{3}_{t_i}, i=1, \ldots,M.$

Here, we construct the market by giving equal weight to each of the type of trade,
but   similar results can be obtained by varying different weights to every strategy. 
We consider 
the true parameter $G^\star $ in \eqref{eq:price_experiments} to be a convolution-type propagator from one of the following two classes: 
\be \label{g-power} 
G^{\star,(1)}_{i,j} = K^{\star,(1)}(t_i-t_j) = \frac{  \kappa }{(t_i-t_j+1)^{\beta^{\star}}}, \quad i \geq j, \quad \textrm{for } \beta^{\star} \in \lbrace 0.1, 0.2, 0.3, 0.4 \rbrace,
\ee 
and 
\be \label{g-exp} 
G^{\star,(2)}_{i,j} = K^{\star,(2)}(t_i-t_j) = \kappa e^{-\rho^{\star}(t_i -t_j)}, \quad i \geq j,   \quad \textrm{for } \rho^{\star} \in \lbrace 0.1, 0.2, 0.3, 0.4\rbrace, 
\ee
where $\kappa$ is given in Table \ref{table:1}.

Given the above synthetic dataset $\mathcal D$, we  examine  the accuracy of the estimators
\eqref{eq:G_n_lambda_volterra} 
and \eqref{eq:K_n_lambda},
where the former performs a fully nonparametric estimation of  
the Volterra propagator, 
and the latter imposes a convolution structure of the estimated model. 
These estimators are implemented as in \eqref{eq:project_admissible}
and \eqref{eq:project_admissible_conv} with $\lambda=10^{-3}$, where 
the projection step 
is carried out   using the Python optimisation package CVXPY  \cite{diamond2016cvxpy}.
Our numerical results show that 
in the present setting, 
the Volterra estimator \eqref{eq:project_admissible} 
yields large estimation errors, 
as the observed trading speeds in the dataset  $\mathcal D$
are not random enough to fully explore the   parameter space (see Remark \ref{rem-asump-mat}).   
This is demonstrated in Figure \ref{FIGURE LABEL_d},
where we consider  
   a   propagator  $G^\star$ associated with the 
power-law kernel  \eqref{g-power} with $\kappa=0.01$ and $\beta^\star= 0.4$,
and present  the  accuracy
of the two estimators. 
Figure \ref{FIGURE LABEL_d} (right) shows that 
the    Volterra estimator  \eqref{eq:project_admissible} yields large
 component-wise relative errors,
 and hence it is difficult to   recover the power-law kernel   from the estimated model.
We refer the reader to Figure \ref{FIGURE LABEL_e} in Appendix \ref{sec-examples_noisy},
which shows that if the historical trading speeds in the  dataset are  sufficiently noisy, then the   estimator \eqref{eq:project_admissible}   recovers the true propagator accurately. 

\begin{figure}[!ht]
\centering
\includegraphics[trim=30 5 30 30, clip, width= 0.48\linewidth]{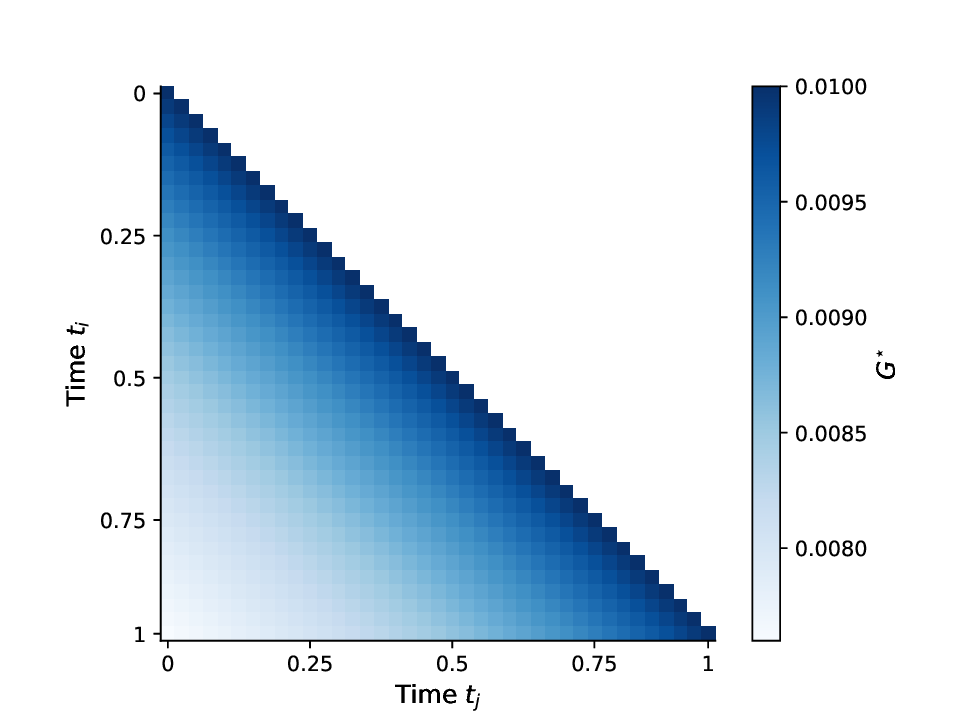}
\includegraphics[trim=30 5 30 30, clip, width=0.48\linewidth]{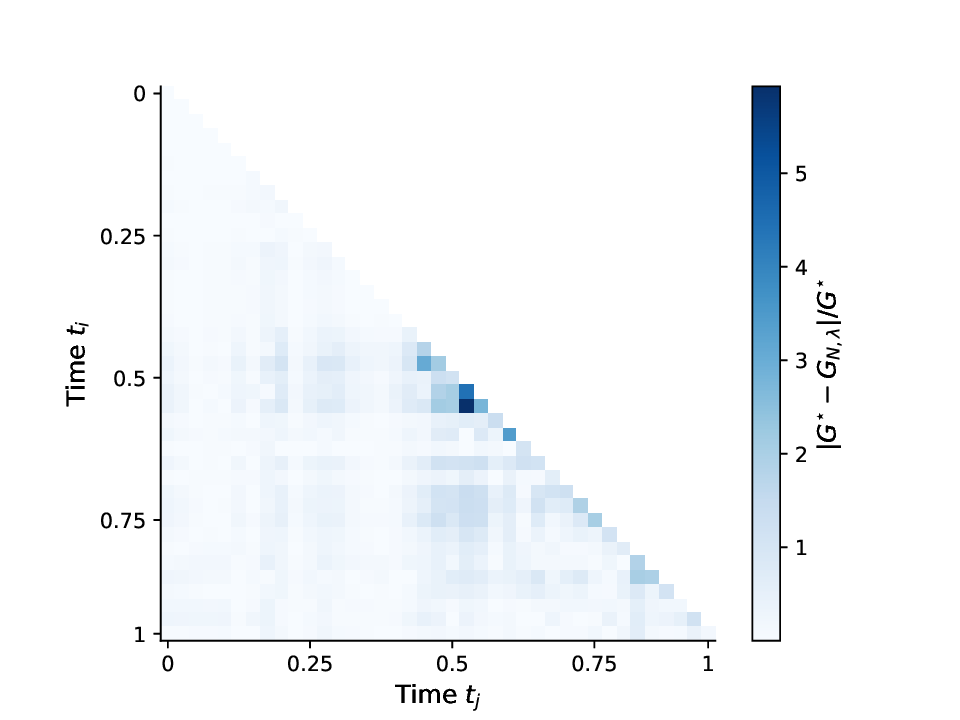}  
\includegraphics[trim=30 5 30 30, clip, width=0.48\linewidth]{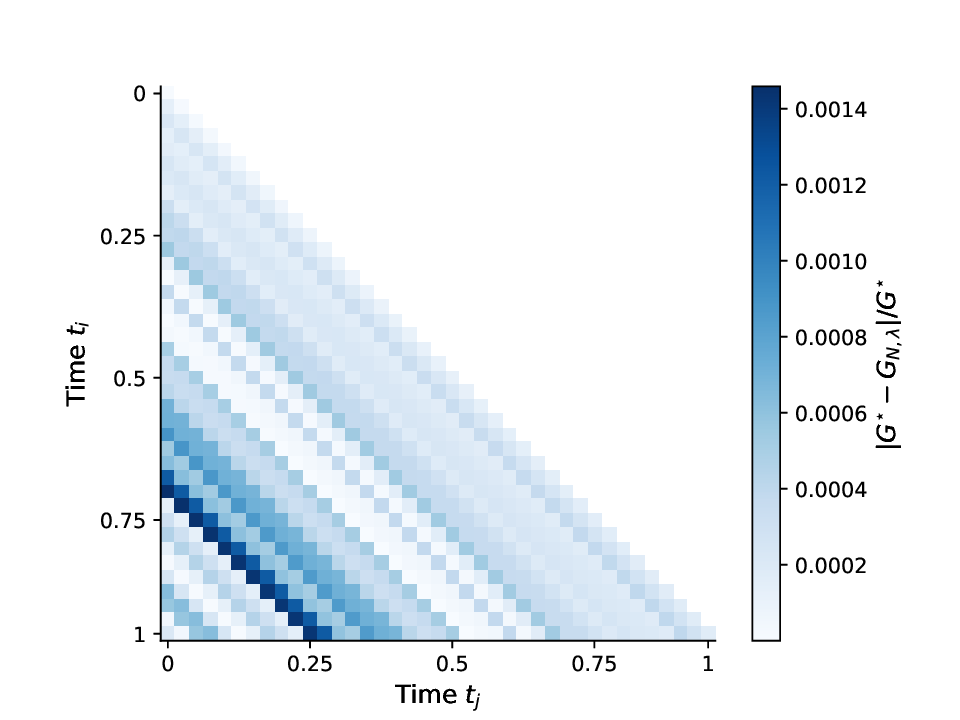}
\caption{
The true power-law propagator $G^{\star}(t)= 0.01 (t+1)^{-0.4}$ (left) and relative errors of the Volterra estimator (right) and the convolution estimator (bottom).
}
\label{FIGURE LABEL_d} 
\end{figure}
 
In contrast, 
Figure \ref{FIGURE LABEL_d} (bottom) clearly shows that 
by imposing a convolution structure of the estimated model, the estimator 
\eqref{eq:project_admissible_conv} achieves a much higher accuracy using the same dataset.  The estimator also   recovers the true convolution kernel accurately, even with a smaller sample size.
This can be seen from 
  Figure \ref{FIGURE LABEL}, where we  plot $K^{\star,(1)}$ in \eqref{g-power}  and $K^{\star,(2)}$
  in \eqref{g-exp}  for the above listed values of $\rho^{\star}, \beta^{\star}$ (see Figures \ref{SUBFIGURE LABEL 1} and \ref{SUBFIGURE LABEL 2}) and compare them to the estimated kernels obtained by \eqref{eq:project_admissible_conv} (see  Figures \ref{SUBFIGURE LABEL 3} and \ref{SUBFIGURE LABEL 4}).
One can clearly observe that, 
 for both considered kernels, 
the proposed estimator captures the behavior of  the true propagators accurately.
To quantify the accuracy of our estimators and to demonstrate the importance of the projection step,
 we define the following relative errors for different sample sizes $N$:
$$ \text{err}^{(i)} \coloneqq \max_{j \in \lbrace 1, \ldots, M \rbrace} \frac{|(K^{\star,(i)})_j - (\tilde{K}_{N,\lambda})_j|}{(K^{\star,(i)})_j},
\quad
\text{err}^{(i)}_{\textrm{proj}} \coloneqq \max_{j \in \lbrace 1, \ldots, M \rbrace} \frac{|(K^{\star,(i)})_j - (K_{N,\lambda})_j|}{(K^{\star,(i)})_j}, 
$$
 where $K^{\star,(1)}$ refers to the power-law kernel  \eqref{g-power} with
 $\beta^{\star} = 0.1$,   
 $K^{\star,(2)}$ refers to 
 the exponential kernel  \eqref{g-exp} with $\rho^{\star} = 0.1$,
 $\tilde{K}_{N,\lambda}$ is the estimated kernel using the plain least-squares estimator \eqref{w-def}, 
 and 
 ${K}_{N,\lambda}$ is obtained by \eqref{eq:project_admissible_conv} with the additional projection step.
 Table \ref{table:2} summarises the estimation accuracy for  sample sizes $N \in \lbrace 63, 126, 252 \rbrace$, which shows that both estimators achieve   relative errors of order $10^{-3}$. 
Moreover, one can see that by imposing monotonicity and convexity on the estimated model, the estimator \eqref{eq:project_admissible_conv} improves the accuracy of the plain least-squares estimator at least by a factor of 2. 

\begin{figure}[!ht]
\centering
\begin{subfigure}{.46\textwidth}
    \centering
    \includegraphics[width=1\linewidth]{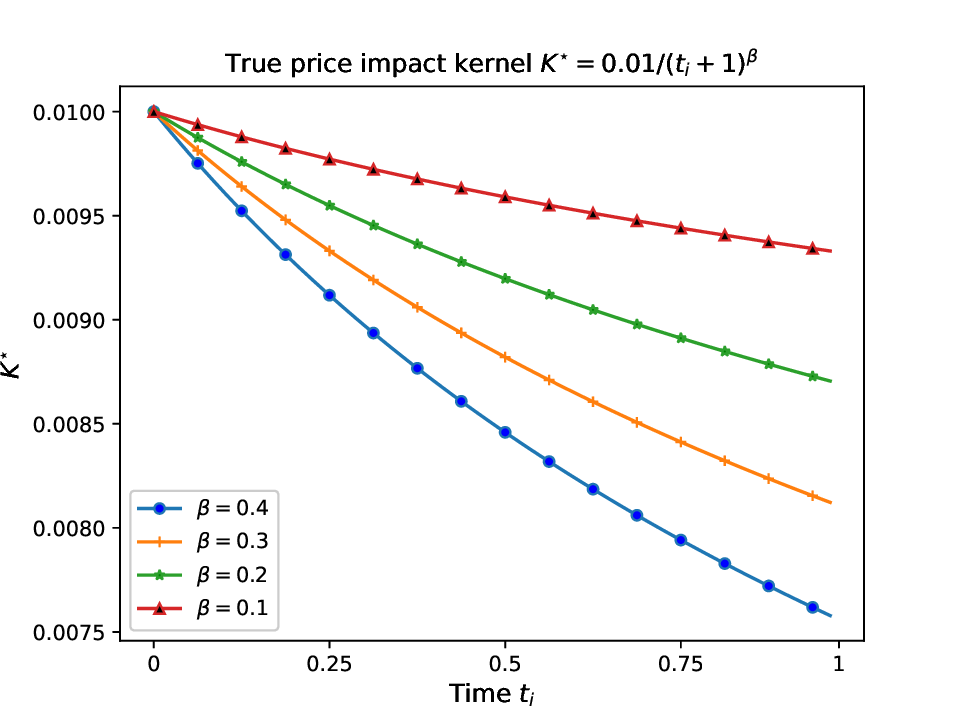}  
    \caption{}
    \label{SUBFIGURE LABEL 1}
\end{subfigure}
\begin{subfigure}{.46\textwidth}
    \centering
 \includegraphics[width=1\linewidth]{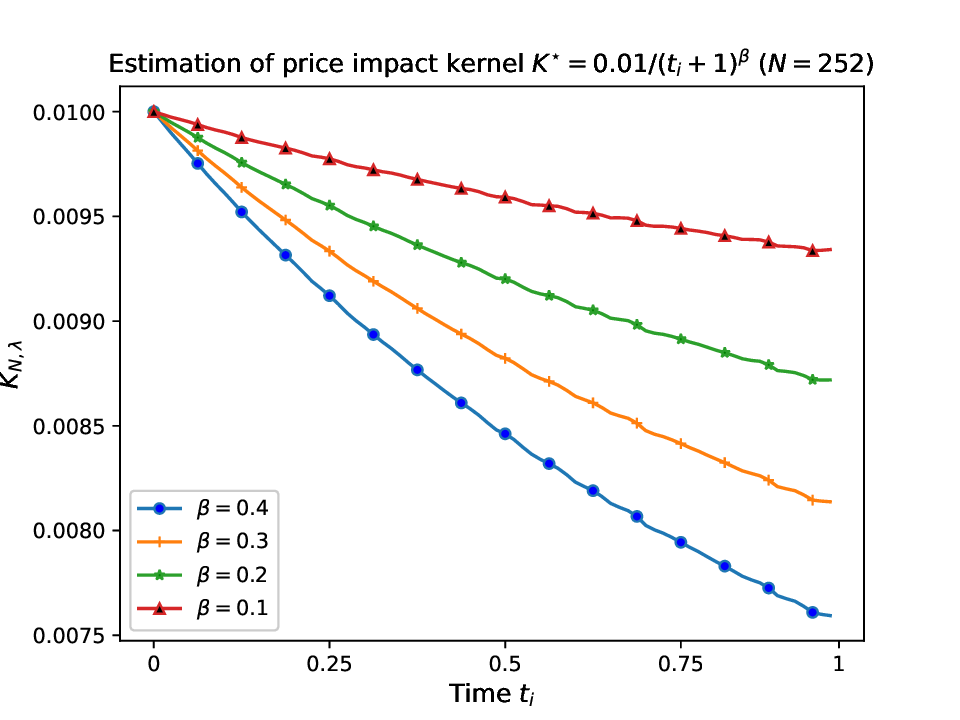}  
    \caption{}
    \label{SUBFIGURE LABEL 3}
\end{subfigure}
\begin{subfigure}{.46\textwidth}
    \centering
    \includegraphics[width=1\linewidth]{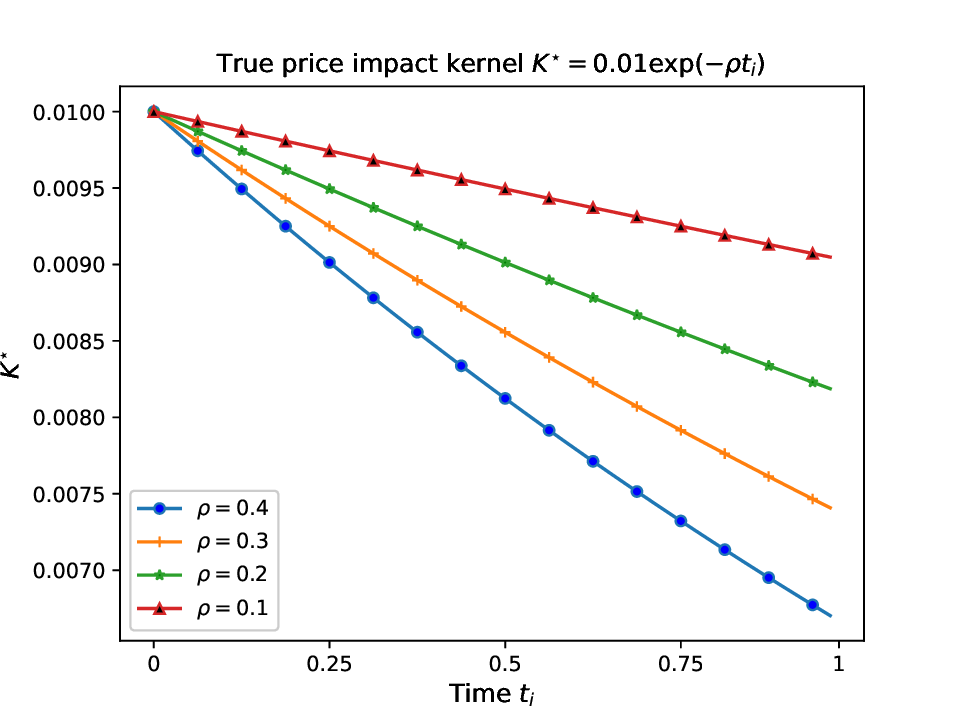}  
    \caption{}
    \label{SUBFIGURE LABEL 2}
\end{subfigure}
\begin{subfigure}{.46\textwidth}
    \centering
    \includegraphics[width=1\linewidth]{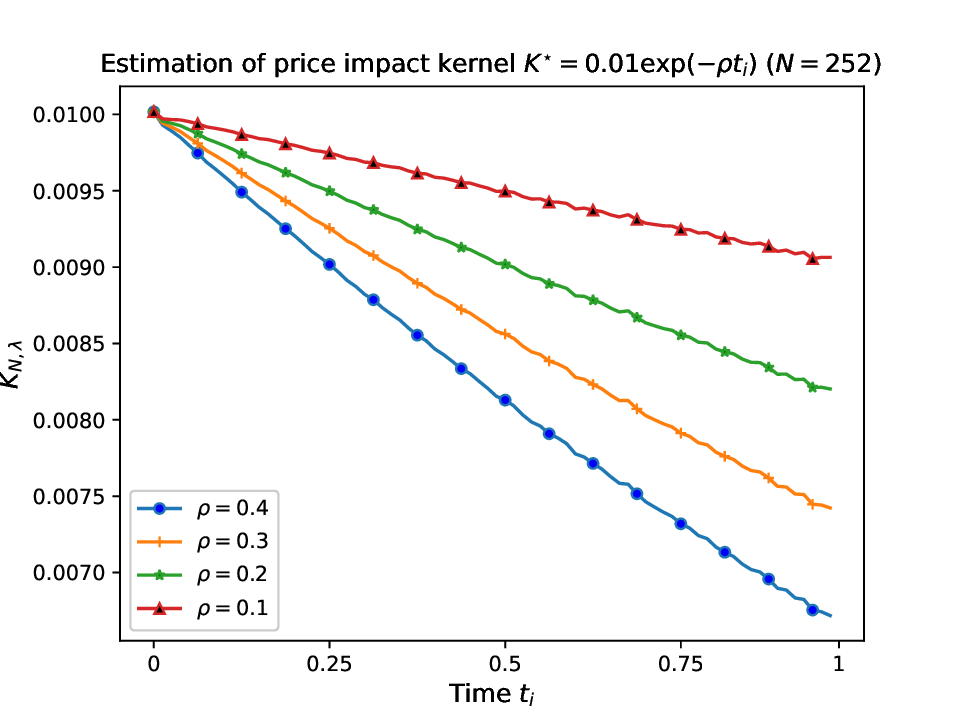}  
    \caption{}
    \label{SUBFIGURE LABEL 4}
\end{subfigure}
\caption{Comparison of the true price impact kernels to the estimated kernels in the cases of $K(t) = \kappa (t+1)^{-\beta}$ (upper panels) and $K(t) = \kappa e^{-\rho t}$ (lower panels). The sample size is $N=252$.} 
\label{FIGURE LABEL}
\end{figure}

 \begin{table}[!ht]
\caption{Accuracy of the convolution estimators
\eqref{eq:project_admissible_conv} and \eqref{w-def}
for different sample sizes} 
\centering 
\begin{tabular}{c | c | c | c | c} 
\hline\hline 
 &
 \multicolumn{2}{c|}{Power-law kernel} 
 & \multicolumn{2}{c}{Exponential   kernel} 
 \\ [0.5ex]
\hline
Sample size $N$  & 
$\textrm{err}^{(1)}$ & $\textrm{err}_{\textrm{proj}}^{(1)}$ &  $\textrm{err}^{(2)}$ 
& $\textrm{err}_{\textrm{proj}}^{(2)}$ \\ [0.5ex]
\hline
63  & $ 8.4 \times 10^{-3} $  & $  1.9 \times 10^{-3}$ & $ 9.8 \times 10^{-3} $ & $ 4.4 \times 10^{-3}$ \\
126  & $ 6.9 \times  10^{-3}$ &  $1.7 \times  10^{-3}$  & $ 5.3 \times  10^{-3}$ & $ 2.9 \times 10^{-3}$ \\
252 & $ 4.2 \times 10^{-3}$   & $ 1.2 \times 10^{-3}$  & $ 4.2 \times 10^{-3}$ &  $ 2. 2\times 10^{-3}$  \\
\hline 
\end{tabular}\label{table:2} 
\end{table}

Given the above estimated models, we proceed to investigate the performance of the pessimistic trading strategy. Our numerical results show that
the performance 
of a naive greedy strategy using the estimated model is very sensitive to the quality of the estimated model. 
In contrast, 
the pessimistic   trading strategy exhibits  a stable performance regardless of the accuracy of the estimated models, 
and achieves an execution cost close to the optimal one.

In particular, 
recall that given the  present dataset, 
the Volterra estimator \eqref{eq:G_n_lambda_volterra} 
yields a poor estimated  propagator 
as 
illustrated in Figure \ref{FIGURE LABEL_d}. 
As a result, 
a naive deployment of a greedy strategy as in Theorem \ref{prop-opt-strat}, using the estimator $G_{N,\lambda}$,  can lead to  substantially suboptimal  costs as discussed after \eqref{sub-decom}.
To illustrate this phenomenon clearly, 
 we continue with the aforementioned example 
 (for the 
power-law kernel  \eqref{g-power} with $\kappa=0.01$ and $\beta^\star= 0.4$)
 and present in Figure \ref{FIGURE LABEL_c} the trading speeds and inventories for the relevant strategies, where we neutralize the effects of exogenous trading signals on these strategies. Specifically, we plot the optimal strategy with precise information on the propagator \eqref{g-power}, the greedy strategy from Theorem \ref{prop-opt-strat} using the estimator $G_{N,\lambda}$ in \eqref{eq:G_n_lambda_volterra}, and the pessimistic strategy minimizing the cost functional \eqref{j-p1}, using the same $G_{N,\lambda}$ (with $N=252$ and $\lam = 10^{-3}$). We observe in Figure \ref{FIGURE LABEL_c} that the greedy strategy exhibits an uninterpretable behaviour as a result of oscillation in the estimator, and that these oscillations are regularized by the pessimistic strategy. We further report that the optimal costs using the true propagator $G^\star$ with zero signal, defined in Theorem \ref{prop-opt-strat} attains the value $4500.24$. Executing a greedy strategy as in Theorem \ref{prop-opt-strat} but with the estimated propagator $G_{N,\lambda}$ yields excessive costs $5216.68$, while for the pessimistic strategy with $G_{N,\lambda}$, the execution costs are significantly closer to optimality, $4537.19$ (see Table \ref{table:5}). 
 
On the other hand, 
  when  the convolution estimator 
  \eqref{eq:project_admissible_conv}  is 
  employed to estimate the propagator, 
  both the greedy strategy and the pessimistic strategy
  yield a close-to-optimal expectation cost. 
  This is due to the fact that 
the estimator \eqref{eq:project_admissible_conv} recovers the true kernel accurately,  
and hence 
stability properties of the associated optimal liquidation problem imply that
introducing a regularization in the cost functional as in \eqref{eq:K_n_lambda}, will not provide a significant improvement.
However, 
it is important to note that 
in practice, 
the true propagator  and  the  accuracy of an    estimated model are  unknown to the agent, 
and hence compared with the pessimistic strategy,  it is more challenging to 
assess   the performance of a greedy strategy before its real deployment (see the discussion after  \eqref{sub-decom}).  

\begin{figure}[!ht]
\centering
\includegraphics[width=0.48\linewidth]{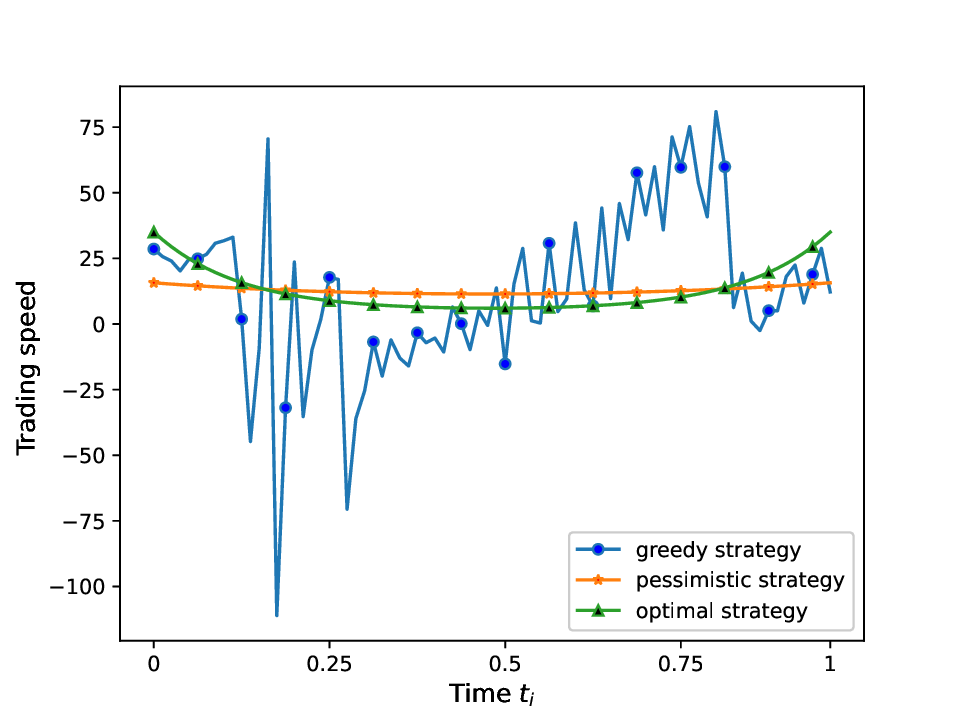}  \includegraphics[width=0.48\linewidth]{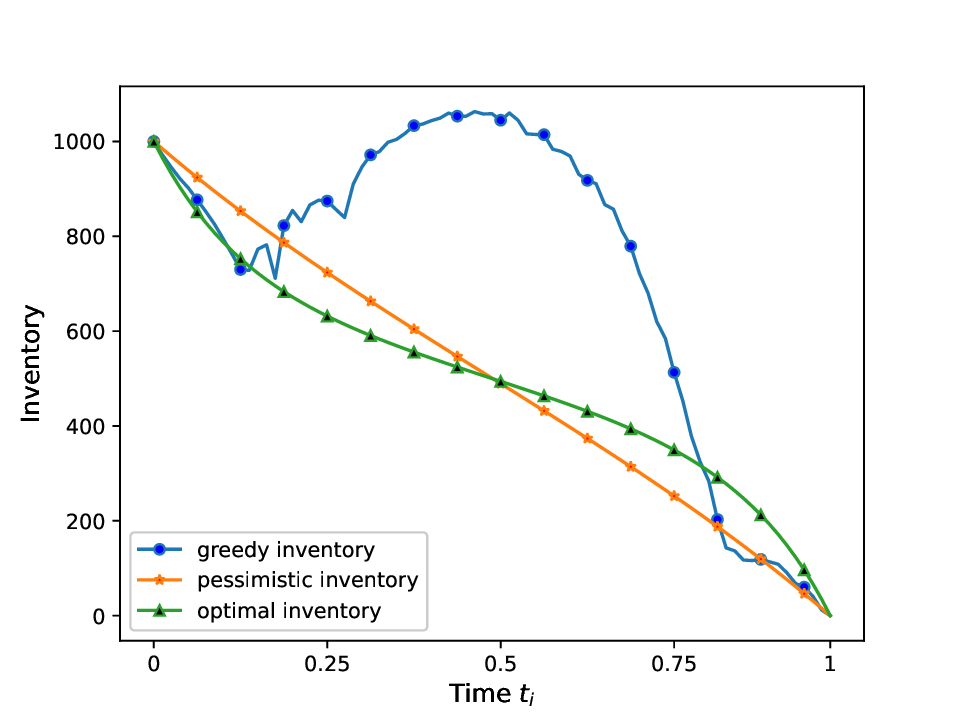}  
\caption{Trading speeds (left panel) and inventories (right panel) of the optimal trading strategy (green), greedy strategy (blue) and pessimistic strategy (orange).}
\label{FIGURE LABEL_c}
\end{figure}

\begin{table}[!ht]
\caption{Liquidation Costs} 
\centering 
\begin{tabular}{|c | c |} 
\hline\hline 
Type of strategy & liquidation costs    \\ [0.5ex]
\hline
Optimal strategy ($G^\star$)  & $ 4500.24 $   \\
Greedy strategy ($G_{N,\lambda}$) &  $5216.68$  \\
Pessimistic strategy ($G_{N,\lambda}$) & $4537.19$  \\
 \hline 
\end{tabular}\label{table:5} 
\end{table}

\section{Martingale tail inequality for least-squares estimation} \label{sec-mart} 
This section first establishes a martingale tail inequality for noise in a general finite dimensional Hilbert space. Based on this tail inequality, we derive high probability bounds for least-squares estimators resulting from correlated observations. The results of this section are of independent interest and extend the results from \cite{abbasi2011improved} from observation taking values in $\mathbb{R}$ to observations taking values in a general finite dimensional Hilbert space, which will be needed in order to prove Theorems \ref{thm:confidence_volterra} and \ref{thm:confidence_convolution}. 

We introduce some definitions and notation which are relevant to our setting. 
\paragraph{Notation:} For all real finite dimensional Hilbert spaces $(X,\langle \cdot,\cdot\rangle_X)$
and $(Y,\langle \cdot,\cdot\rangle_Y)$,
we denote by 
$\id_X$ the identity map on $X$,
and by
$\mathcal{L}_2(X,Y)$ 
be the space of  linear  maps 
$A: X\to Y$  equipped with the 
Hilbert-Schmidt 
norm  $\|A\|_{\mathcal{L}_2}
=\sqrt{\sum_{k=1}^{n_x} \|Ae_k\|^2_{Y}}<\infty$,
where 
$(e_k)_{k=1}^{n_x}$
is an orthonormal basis of $X$ of dimension $n_x$.
The norm $\|\cdot\|_{\mathcal{L}_2}$
is induced by
an  inner product 
$\langle \cdot,\cdot\rangle_{\mathcal{L}_2}$
such that 
$
\langle A,A'\rangle_{\mathcal{L}_2}
=\sum_{k=1}^{n_x} 
\langle Ae_k,A'e_k\rangle_{Y}
$ for all  $A,A'\in \mathcal{L}_2(X,Y)$.
Both $\|\cdot\|_{\mathcal{L}_2}$
and $\langle \cdot,\cdot\rangle_{\mathcal{L}_2}$ do not depend on 
 the choice of  
the    basis $(e_k)_{k=1}^{n_x}$ of $X$. 
We 
 write   $\mathcal{L}_2 (X)=\mathcal{L}_2 (X,X)$
 for simplicity.

 For each  $A\in \mathcal{L}_2(X,Y)$, we denote by $A^*:Y\to X$
  the adjoint  of $A$,
and say 
   $A$ is symmetric if $A=A^*$.
   We denote by $\mathcal{S}_0(X)$
   the space of symmetric linear  maps $A :X\to X$
   satisfying $\langle Ax, x\rangle_X\ge 0$ for all $x\in X$,
   and by
   $\mathcal{S}_+(X)$
   the space of symmetric linear  maps $A :X\to X$
   satisfying   $\langle Ax, x\rangle_X> 0$ for all $x\in X$ and $x\not =0$.
For any $A\in \mathcal{S}_0(X)  $,
we define the seminorm $\|\cdot\|_{X,A}:X\to [0,\infty)$ by
$\|u\|_{X,A}=\sqrt{\langle Au, u\rangle_X}$
for all $u\in X$.
We  write $\|\cdot\|_{X}=\|\cdot\|_{X,\id_X}$ for simplicity. 

The following definition introduces  conditional sub-Gaussian random variables. 
\begin{definition}
\label{def:conditional_subgaussian}
Let $(\Omega, \mathcal{F},\mathbb{P})$ be a probability space, let $(H,\langle\cdot, \cdot\rangle_H)$ be a finite dimensional real Hilbert space,
    and let $Z:\Omega\to H$ 
 be a  mean zero
     random variable.
     Let $\mathcal{G}\subset \mathcal{F}$ be sub-sigma field 
     and $\alpha \ge 0$ a constant. 
     We say 
     $Z$  is $\alpha$-conditionally sub-Gaussian 
     with respect to   $\mathcal{G}$  
if
\be \label{cond-prop} 
\mathbb{E}\left[\exp\left({\langle u, Z \rangle_H}\right)\mid \mathcal{G}\right]
\le \exp\left( \frac{\alpha^2\|u\|^2_H  }{2}\right),
\quad u\in H.
\ee
\end{definition}

%



We first extend the martingale tail inequality with scalar noise, which was introduced in 
\cite[Theorem 1]{abbasi2011improved}, to noise
in a general finite dimensional Hilbert space. 
\paragraph{Our setting:} For the remainder of this section we fix $X$ and $Y$ to be real finite dimensional Hilbert spaces and let
$(\Omega, \mathcal{F},\mathbb{P})$ be a probability space with a filtration
$\mathbb{G}=(\mathcal{G}_i)_{i=0}^\infty\subset \mathcal{F}$. We define 
$(A_i)_{i\in \mathbb{N}}$ as a $\mathcal{L}_2(X,Y)$-valued process 
predictable with respect to $\mathbb{G}$, and the noise process $(\eta_i)_{i\in \mathbb{N}}$ as a $Y$-valued  process 
adapted to $\mathbb{G}$
such that 
 $$
 \mathbb{E}[\eta_i\mid \mathcal{G}_{i-1}]=0, \quad  \textrm{for all } i \geq 1, 
 $$ and 
$\eta_i$ is $\alpha$-conditionally sub-Gaussian  with respect to $\cG_{i-1}$ 
(cf.~Definition \ref{def:conditional_subgaussian}). 

We further define the following stochastic processes.  Let  $M_{n}:
\Omega \to \mathcal{S}_0(X) $
and 
 $U_n:  \Omega \to X$ be such that for all $\omega\in \Omega$,
\be \label{m-proc} 
M_{n}(\omega) 
=\sum_{i=1}^n   A_i^*(\omega)A_i(\omega), \quad n \geq 1, 
\ee
and 
\be \label{u-proc} 
U_n (\omega)   =\sum_{i=1}^n A^*_i(\omega)\eta_i(\omega),  \quad n \geq 1. 
\ee
 
Now we are ready to introduce our martingale tail inequality for noise
in a general finite dimensional Hilbert space. 
\begin{theorem} \label{thm:martingale_tail}
For all   $V\in \mathcal{S}_+(X)$ and $\delta \in (0,1)$ we have 
$$
\mathbb{P}\left( \|U_{n}\|^2_{X,  (M_{n}+V)^{-1}}
\le  2\alpha^2\log\left(
\frac{\sqrt{\det \left( V^{-1} M_{n }+\id_X \right)}}{\delta}
\right),
\quad \forall n\in \mathbb{N}
\right)\ge 1-\delta,
$$
where $\det()$ denotes the Fredholm determinant (cf.~\cite[Chapter VII]{gohberg1990classes}).
\end{theorem}
In order to prove Theorem \ref{thm:martingale_tail} we introduce the following auxiliary results.

\begin{lemma}\label{lemma:super-martingale}
Let $(M_n)_{n \in \mathbb{N}}$ and $(U_n)_{n \in \mathbb{N}}$ as in \eqref{m-proc} and \eqref{u-proc}. Then for all  
$u\in X$
and $\mathbb{G}$-stopping time $\tau$,
$$
G_{\tau} (u)=\exp\left( \frac{\langle u, U_{\tau}\rangle_{X}}{\alpha}-\frac{1}{2}\langle M_{\tau} u, u\rangle_{X}\right)
$$ 
is well-defined 
and satisfies 
$\mathbb{E}[G_{\tau}(u)]\le 1$.
\end{lemma}
\begin{proof}
    For any $u\in X$ we define, 
  \be \label{g-n-def} 
  G_{n} (u)=\exp\left( \frac{\langle u, U_{n}\rangle_{X}}{\alpha}-\frac{1}{2}\langle M_{n} u, u\rangle_{X}  \right), \quad n \geq 1, 
  \ee
  with  $G_{0}(u)=1$.
From \eqref{m-proc} and \eqref{u-proc} we get, 
$$G_{n}(u)= \exp\left( \sum_{i=1}^{n} \left[ \frac{1}{\alpha}\langle u, A_{i}^*   \eta_i \rangle_{X}-\frac{1}{2} \langle  A_{i}^*A_{i}   u, u \rangle_{X}   \right] \right)
=\prod_{i=1}^n D_i, \quad \textrm{for all }n \geq 1, 
$$
with 
\begin{equation*}
    D_i \coloneqq  \exp\left(  \frac{1}{\alpha}\langle A_{i} u, \eta_i  \rangle_{Y}-\frac{1}{2} \langle    A_{i} u, A_{i} u \rangle_{Y} 
 \right).
\end{equation*}
Since $A_{i} $ is $\mathcal{G}_{i-1}$ measurable,  and $\eta_i$ is  $\alpha$-conditional sub-Gaussian with respect to $\mathcal{G}_{i-1}$ we get from \eqref{cond-prop}, 
\begin{align*}
   \mathbb{E}[D_i|\mathcal{G}_{i-1}]
   &=\mathbb{E}\left[\exp \left( 
    \frac{1}{\alpha}\langle A_{i} u, \eta_i  \rangle_{Y} 
 \right)\bigg|\mathcal{G}_{i-1}\right]
 \exp\left( -\frac{1}{2}  \langle    A_{i} u, A_{i} u \rangle_{Y} 
 \right)
 \\
 &\le 
 \exp\left( \frac{1}{2}\left  \| A_{i} u \right \|_{Y} 
 \right)\exp\left( -\frac{1}{2} \|A_{i} u\|_Y^2 
 \right)=1.
\end{align*}
Thus by the measurability of   $D_i$ with respect to $\mathcal{G}_i$,
we get for all $n \geq 1$,
\begin{align*}
    \mathbb{E}[G_{n}(u )|\mathcal{G}_{n-1}] =\left(\prod_{i=1}^{n-1}D_i  \right) \mathbb{E}[D_n|\mathcal{G}_{n-1}] = G_{n-1}(u) \mathbb{E}[D_n|\mathcal{G}_{n-1}]
    \le G_{n-1}(u),
\end{align*}
hence $(G_{n}(u))_{n \geq 0 }$ is a super-martingale with respect to $\mathbb{G}$. Using the fact that $G_{0}(u)=1$, we deduce
$\mathbb{E}[G_{n}(u)] \leq \mathbb{E}[G_{n-1}(u)]\le  1$ for all $n\in \mathbb{N}$.

Now  let $\tau$ be a stopping time with respect to the filtration $\mathbb{G}$. Then $(G_{n \land \tau }(u))_{n \in \mathbb{N}}$ is a nonnegative super-martingale. By  the optional stopping theorem,   $\mathbb{E}[G_{n \land \tau}(u)] \leq 1$ for any $n \in \mathbb{N}$. By Doob's martingale convergence theorem,  $G_{\tau}(u)=\lim_{n \to \infty} (G_{n \land \tau}(u))$ exists and is finite a.s., hence along with Fatou's lemma it follows that $\mathbb{E}[G_{\tau}(u)] \leq 1$.
\end{proof}

 \begin{proposition} \label{prop:tail_stopping}
   Let $(M_n)_{n \in \mathbb{N}}$ and $(U_n)_{n \in \mathbb{N}}$ as in \eqref{m-proc} and \eqref{u-proc}. Then, for
     any $\mathbb{G}$-stopping times $\tau$, $V\in \mathcal{S}_+(X)$ and $\delta\in (0,1)$,
     \begin{align*}
    \mathbb{P}\left(
    \exp\left( 
\frac{1}{2\alpha^2} \|U_{\tau}\|^2_{X,  (M_{\tau}+V)^{-1}}\right)
\ge \delta^{-1}\sqrt{\det \left( V^{-1} M_{\tau }+\id_X \right)}
    \right)
  &  \le   \delta^{-1}.
\end{align*}

 \end{proposition}

\begin{proof}
As $X$ is a finite dimensional space and $V$ is symmetric and positive definite, it follows that 
$V$ is invertible, and $V^{-1}\in   \mathcal{S}_+(X)$.
 A singular value decomposition of $V^{-1}$ implies that 
there exists 
$(\lambda_i)_{i=1}^{n_x}\subset (0,\infty)$ 
and 
 an orthonormal basis  $(e_i)_{i=1}^{n_x}$ of $X$
 such that 
 \be \label{v-inv-dec} 
 V^{-1}x=\sum_{i=1}^{n_x} \lambda_i \langle x,e_i\rangle_X e_i, \quad \textrm{for all }
  x\in X.
  \ee
  Let 
$\mathcal{G}_\infty=\bigcup_{n=0}^\infty\mathcal{G}_n$,
and
  $(\xi_i)_{i=1}^{n_x}$ be  standard 
normal   random variables 
that are mutually independent and are independent of $\mathcal{G}_\infty$.
Define 
\be \label{def-w} 
W =\sum_{i=1}^{n_x} \sqrt{\lambda_i}\xi_ie_i \in L^2(\Omega; X).
\ee
Let $n\in \mathbb{N}$ be fixed, 
and further define 
\be \label{h-n-def} 
H_{n}=\mathbb{E}[G_{n}(W)\mid \mathcal{G}_\infty]. 
\ee
Then, since $Vx=\sum_{i=1}^{n_x} \lambda^{-1}_i \langle x,e_i\rangle_X e_i$ for all $x\in X$, we get from \eqref{g-n-def} and \eqref{def-w}, 
\be \label{h-bnd1} 
\begin{aligned}
 H_{n}
&=\mathbb{E}\left[\exp\left(
\frac{\langle W, U_{n}\rangle_{X}}{\alpha}-\frac{1}{2}\langle M_{n} W, W\rangle_{X}
\right)
    \,\bigg\vert\, \mathcal{G}_\infty
\right]
\\
& 
=\mathbb{E}\left[\exp\left(
\frac{\langle W, U_{n}\rangle_{X}}{\alpha}-\frac{1}{2}\langle (M_{n}+ V) W, W\rangle_{X}
+
 \frac{1}{2}\langle  V  W, W\rangle_{X}
\right)
    \,\bigg\vert\, \mathcal{G}_\infty
\right]
\\
& = 
\mathbb{E}\left[\exp\left(
\frac{\langle W, U_{n}\rangle_{X}}{\alpha}-\frac{1}{2}\langle (M_{n}+V) W, W\rangle_{X}
+
 \frac{1}{2}\sum_{i=1}^m \xi^2_i
\right)
    \,\bigg\vert\, \mathcal{G}_\infty
\right]
\\
& =
\mathbb{E}\left[\exp\left( 
\frac{1}{2} \left\|\frac{U_{n}}{\alpha}\right\|^2_{X, (M_{n}+V)^{-1}}
 -\frac{1}{2} 
 \left\| W-(M_{n}+V)^{-1}\frac{U_{n}}{\alpha}\right\|_{X, M_{n}+V}^2 
+\frac{1}{2} \sum_{i=1}^m \xi_{i}^2\right)
    \,\bigg\vert\, \mathcal{G}_\infty
\right],
\end{aligned}
\ee
where the last identity used 
     $M_{n}+V\in \mathcal{S}_+(X) $
     and the fact that
     for all  $u,v\in X$
     and   $A\in \mathcal{S}_+(X)$,   
 \begin{align*}
& 
 \langle u, v\rangle_{X} -\frac{1}{2}\langle A u, u\rangle_{X}
   =
 \frac{1}{2} \|v\|^2_{X, A^{-1}}
 -\frac{1}{2} 
 \| u-A^{-1}v\|_{X, A}^2.
  \end{align*}
 Since
$U_{n}$ and $M_{n}$ are measurable with respect to $\mathcal{G}_\infty$, 
$(  \xi_i)_{i=1}^{n_x}$ are standard Gaussian random variables  independent of $\mathcal{G}_\infty$, 
by writing 
  $\bar{M}_{n}=M_{n}+ V$ we get from \eqref{h-bnd1} and \eqref{v-inv-dec} that 
\begin{align} \label{eq:H_n_m}
\begin{split}
 H_{n} 
& =
\exp\left( 
\frac{1}{2\alpha^2} \|U_{n}\|^2_{X,  \bar{M}_{n}^{-1}}\right)
\mathbb{E}\left[\exp\left( 
 -\frac{1}{2} 
 \| \bar{M}^{1/2}_{n} (W_{n}-\alpha^{-1}\bar{M}_{n}^{-1}U_{n})\|_{X}^2 
+\frac{1}{2} \sum_{i=1}^{n_x} \xi_{i}^2\right)
    \,\bigg\vert\, \mathcal{G}_\infty
\right]
\\
&=
\frac{\exp\left( 
\frac{1}{2\alpha^2} \|U_{n}\|^2_{X,  \bar{M}_{n}^{-1}}\right)}
{\sqrt{(2\pi)^{n_x}}}
\int_{\mathbb{R}^{n_x}}
 \exp\left( 
 -\frac{1}{2} 
\left\|  \bar{M}^{1/2}_{n}\left( \sum_{i=1}^{n_x} \sqrt{\lambda_i}x_ie_i  -  \alpha^{-1} \bar{M}_{n}^{-1}U_{n}\right)\right\|_{X}^2 
\right) \mathrm{d}x 
\\
& =\exp\left( 
\frac{1}{2\alpha^2} \|U_{n}\|^2_{X,  \bar{M}_{n}^{-1}}\right)
\sqrt{\det \left(V \bar{M}_{n}^{-1}\right)}
\\
&
=\left(
\frac{1}{\sqrt{\det(V^{-1}M_n+\id_X)}}
\right)\exp\left( 
\frac{1}{2\alpha^2} \|U_{n}\|^2_{X,  (M_{n}+V)^{-1}}\right),
\end{split}
\end{align}
  where the last identity used $\det(AB)=\det(A)\det(B)$.
  
Since  \eqref{eq:H_n_m} holds for all $n\in \mathbb{N}$, it also holds for all $\mathbb{G}$-stopping times $\tau$. 
By Lemma \ref{lemma:super-martingale} and \eqref{h-n-def},
for all $\mathbb{G}$-stopping times $\tau$,
$\mathbb{E}[H_{\tau}]
=\mathbb{E}[\mathbb{E}[G_{\tau}(W)\mid {\mathcal G_{\infty}}]]\le 1$.
Then for all $\delta\in (0,1)$, by Markov's inequality and $\mathbb{E}[H_{\tau}]\le 1$ we get, 
\begin{align*}
    \mathbb{P}\left(
    \frac{\exp\left( 
\frac{1}{2\alpha^2} \|U_{\tau}\|^2_{X,  (M_{\tau}+V)^{-1}}\right)}
{\delta^{-1}\sqrt{\det \left(V^{-1} M_{\tau}+  \id_{X} \right)}}
\ge 1
    \right)
  &  \le   \mathbb{E}\left[ 
    \frac{\exp\left( 
\frac{1}{2\alpha^2} \|U_{\tau}\|^2_{X,  (M_{\tau}+V)^{-1}}\right)}
{\delta^{-1}\sqrt{\det \left( V^{-1} M_{\tau}+  \id_{X} \right)}}
    \right]
    \\
 &=\delta^{-1}\mathbb{E}[H_{\tau}]\le \delta^{-1},
\end{align*}
which concludes the proof of Proposition \ref{prop:tail_stopping}. 
\end{proof}
Now we are ready to prove Theorem \ref{thm:martingale_tail}.
\begin{proof}[Proof of Theorem \ref{thm:martingale_tail}]
Let $\delta \in (0,1)$ be fixed.
  For each $n\in \mathbb{N}$,
  define the event 
  $$
B_n(\delta) = \left\{
\omega\in \Omega \,\bigg\vert\,
\exp\left( 
\frac{1}{2\alpha^2} \|U_{n}\|^2_{X,  (M_{n}+V)^{-1}}\right)
> \delta^{-1}\sqrt{\det \left( V^{-1} M_{n }+\id_X \right)}
\right\}.
  $$
Define the stopping time 
$\tau:\Omega\to\mathbb{N}$
such that 
$\tau(\omega)=\min\{n\in \mathbb{N}\mid \omega\in B_n(\delta)\}$ for all $\omega\in \Omega$.
Then by the identity 
$\bigcup_{n\in \mathbb{N}}B_n(\delta) =\{\tau <\infty\}$ and by Proposition \ref{prop:tail_stopping},
\begin{align*}
\mathbb{P}\left(\bigcup_{n\in \mathbb{N}}B_n(\delta)\right)
&=\mathbb{P}(\tau <\infty) \\
&=
 \mathbb{P}\left(
    \exp\left( 
\frac{1}{2\alpha^2} \|U_{\tau}\|^2_{X,  (M_{\tau}+V)^{-1}}\right)
> \delta^{-1}\sqrt{\det \left( V^{-1} M_{\tau }+\id_X \right)},  \, \tau<\infty
    \right)
\\
&\le 
\mathbb{P}\left(
    \exp\left( 
\frac{1}{2\alpha^2} \|U_{\tau}\|^2_{X,  (M_{\tau}+V)^{-1}}\right)
> \delta^{-1}\sqrt{\det \left( V^{-1} M_{\tau }+\id_X \right)} 
    \right)
\le \delta.
\end{align*}
This proves the desired estimate. 
\end{proof}

Based on Theorem \ref{thm:martingale_tail},
 we establish the following high probability bounds  of a projected least-squares estimator  based on correlated observations. Recall that the Hilbert spaces $X$, $Y$ and $(\Omega, \mathcal{F},\mathbb{P})$, the filtration $\mathbb{G}$ and the predictable process $(A_n)_{n \in \mathbb{N}}$ were defined before \eqref{m-proc} and that $(M_n)_{n \in \mathbb{N}}$ was introduced in \eqref{m-proc} and that $\alpha$-conditional sub-Gaussian random variables were defined in Definition \ref{def:conditional_subgaussian}.

\begin{theorem}\label{thm:confidence_interval}
Let $(y_i)_{i\in \mathbb{N}}$ be an $Y$-valued process 
adapted with respect to $\mathbb{G}$.
Assume that there exists a nonempty closed convex subset 
$\mathcal{C}\subset X$,
$x^\star \in \mathcal{C}$ and $R\ge 0$ such that 
\begin{equation}\label{regression_general_set_up}
    y_i = A_i x^\star+\eps_i ,\quad \textrm{for all }  i \geq 1,
\end{equation}
where 
$\mathbb{E}[\eps_i\mid \mathcal{G}_{i-1}]=0$, and 
$\eps_i$ is $R$-conditional sub-Gaussian  with respect to $\mathcal{G}_{i-1}$.
Define 
  \begin{equation}
\label{eq:x_n_lambda}
x_{n,\lambda}\coloneqq\argmin_{x\in \mathcal{C}}
\left(
\sum_{i=1}^n\|y_i-A_i x\|^2_Y+\lambda \|x\|^2_{X}\right) , \quad n\in \mathbb{N} ,\ \lambda>0. 
\end{equation}
  Then for any $\lambda>0$ and $\delta \in (0,1)$,
  with probability at least $1-\delta$ we have
  for all $n\in \mathbb{N}$,
  $$
\left \|   x_{n, \lambda} - x^\star
    \right\|_{X,M_n+\lambda \id_X}
\le  R
\left(
2\log\left(
\frac{\sqrt{\det \left( \lambda^{-1} M_{n }+\id_X \right)}}{\delta}
\right)\right)^{1/2}
+
\lambda     \left\|x^\star \right\|_{X,(M_n+\lambda \id_X)^{-1}}.
$$
\end{theorem}

\begin{proof}
Note that for any $n\in \mathbb{N}$ and $\lambda>0$ the map $J_{n,\lam}:X \mapsto  \mathbb{R}\cup \{\infty\} $ which is given by 
$$
J_{n,\lam}(x) = \sum_{i=1}^{n} \|y_i-A_i x\|^2_Y+\lambda \|x\|^2_{X}+
\boldsymbol{\delta}_{\mathcal{C}}(x) , 
\quad 
\textnormal{with 
$\boldsymbol{\delta}_{\mathcal{C}}(x)=\begin{cases}
    0, & x\in \mathcal{C}\\
    \infty, & x\not \in \mathcal{C}
\end{cases},$}
$$
is strongly convex. Hence $x_{n,\lambda}\in \mathcal{C}$ is well-defined
and it satisfies the following first-order condition: 
$$
\left\langle  2\sum_{i=1}^{n} A_i^*(A_i x_{n,\lambda}-y_i) +2\lambda  x_{n,\lambda}, h-x_{n,\lambda}\right\rangle_X\ge 0, \quad \textrm{for all }h\in \mathcal{C}. 
$$ 
Substituting $h=x^\star\in \mathcal{C}$ in the above inqeuality and using  $y_i =A_i x^\star+\eps_i$ gives,
\begin{align*}
 \left\langle   \sum_{i=1}^{n} A_i^* ( A_i  x_{n,\lambda}-A_i  x^\star -\eps_i ) + \lambda  (x_{n,\lambda}-x^\star)
+\lambda x^\star, x^\star-x_{n,\lambda}\right\rangle_X
\ge 0.
\end{align*}
Now let $M_n=\sum_{i=1}^{n} A_i^{*} A_i  $ 
and $\bar{M}_n=M_n+\lambda \id_X$.
Then
\begin{align*}
 \left\langle     \bar{M}_n (  x_{n,\lambda}-x^\star) -\sum_{i=1}^{n} A_i^* \eps_i   
+\lambda x^\star, x_{n,\lambda}-x^\star\right\rangle_X
\le 0,
\end{align*}
which along with the invertibility of $\bar{M}_n\in \mathcal{S}_+(X)$ and the Cauchy-Schwarz inequality implies that 
\begin{align*}
\left \|   x_{n, \lambda} - x^\star
    \right\|_{X,\bar{M}_n}^2
&\le 
 \left\langle    \sum_{i=1}^{n} A_i^* \eps_i   
-\lambda x^\star, x_{n,\lambda}-x^\star\right\rangle_X \\
&=
 \left\langle    \bar{M}_n \bar{M}_n^{-1}\left(\sum_{i=1}^{n} A_i^* \eps_i   
-\lambda x^\star\right), x_{n,\lambda}-x^\star\right\rangle_X
\\
&\le 
\left\|      \bar{M}_n^{-1}\left(\sum_{i=1}^{n} A_i^* \eps_i 
-\lambda x^\star\right)
\right\|_{X,\bar{M}_n}
\left\|      
x_{n,\lambda}-x^\star\right\|_{X,\bar{M}_n }.
\end{align*}
This together with the identitity 
 $\left \| \bar{M}_n^{-1}  x
      \right\|_{X,\bar{M}_n}
      =\left \|   x\right\|_{X,\bar{M}_n^{-1}}
      $  
for all $x\in X$ yields
\begin{align*}
\left \|   x_{n, \lambda} - x^\star
    \right\|_{X,\bar{M}_n} 
&\le 
\left\|      \sum_{i=1}^{n} A_i^* \eps_i 
-\lambda x^\star 
\right\|_{X,\bar{M}^{-1}_n}
\le 
\left(\left\|      \sum_{i=1}^{n} A_i^* \eps_i 
\right\|_{X,\bar{M}^{-1}_n}
+\lambda\left\|     
  x^\star 
\right\|_{X,\bar{M}^{-1}_n}\right).
\end{align*}
The desired estimate then follows from Theorem \ref{thm:martingale_tail}
with $V=\lambda \id_X$ and $\alpha =R$.
\end{proof}

\section{Proof of Theorems \ref{thm:confidence_volterra} and \ref{thm:confidence_convolution}} \label{sec-pf-est}

\begin{proof}[Proof of Theorem \ref{thm:confidence_volterra}] 
Let  $( \mathcal{G}_i)_{i=0}^{N-1}$ be  the filtration given   in Definition \ref{def:offline_data}, and 
for each $n=1,\ldots, N$, let     $T_n:\mathbb{R}^{M \times M}\to \mathbb{R}^M$ be such
that 
$T_n G\coloneqq  Gu^{(n)} $
for all $G\in \mathbb{R}^{M \times M}$,
and 
let $T^*_n: \mathbb{R}^M \to \mathbb{R}^{M \times M}$ be the adjoint of $T_n$.
Then by \eqref{eq:price_impact_linear_regression},
$y^{(n)}=T_nG^\star+\eps^{(n)}$
for all $n=1,\ldots, N$, and   
the least-squares estimator \eqref{eq:G_n_lambda_volterra} is equivalent to  
\begin{equation}
G_{N,\lambda}\coloneqq\argmin_{G \in \mathbb{R}^{M \times M}}
\sum_{n=1}^N \|y^{(n)}-T_n G \|^2+\lambda\|G\|_{\mathbb{R}^{M \times M}}^2.
\end{equation}
Recall  that  
$(\mathbb{R}^{M \times M}, \|\cdot\|_{\mathbb{R}^{M \times M}})$   
is a finite-dimensional Hilbert space with
  the inner product 
$\langle A,B\rangle_{\mathbb{R}^{M \times M}}\coloneqq \tr(A^\top B)$
for all $A,B\in  \mathbb{R}^{M \times M}$. 
Then by 
Theorem \ref{thm:confidence_interval}
(with $X=\mathbb{R}^{M \times M}$ and $Y=\mathbb{R}^M$),
 for all $\lambda>0$ and $\delta \in (0,1)$,
  with probability at least $1-\delta$,
\begin{align}
\label{eq:G_error_HS}
\begin{split}
&\left \|   G_{N,\lambda}  - G^\star
    \right\|_{\mathbb{R}^{M \times M},M_N+\lambda \id_{\mathbb{R}^{M \times M}}}
\\
&\quad \le  R
\left(
2\log\left(
\frac{\sqrt{\det \left( \lambda^{-1} M_{N }+\id_{\mathbb{R}^{M \times M}} \right)}}{\delta}
\right)\right)^{1/2}
+
\lambda     \left\|G^\star \right\|_{\mathbb{R}^{M \times M},(M_N+\lambda \id_{\mathbb{R}^{M \times M}})^{-1}},
\end{split}
\end{align}
where 
  $\det()$ denotes the Fredholm determinant,
$M_N =\sum_{n=1}^N T_n^*T_n: \mathbb{R}^{M \times M}\to \mathbb{R}^{M \times M}$, 
$\id_{\mathbb{R}^{M \times M}}$ is the identity map on $\mathbb{R}^{M \times M}$, 
and the norm 
$\|\cdot\|_{\mathbb{R}^{M \times M},M_N+\lambda \id_{\mathbb{R}^{M \times M}}}$
is   defined by 
\begin{equation}
\label{eq:RHH_weightnorm}
\|G\|_{\mathbb{R}^{M \times M},M_N+\lambda \id_{\mathbb{R}^{M \times M}}}
=\left(
\left\langle
(M_N+\lambda \id_{\mathbb{R}^{M \times M}} )G,G
\right\rangle_{\mathbb{R}^{M \times M}}
\right)^{1/2},
\quad G\in \mathbb{R}^{M \times M}.
\end{equation}

In the sequel, we fix $\lambda>0$ and $\delta \in (0,1)$, and simplify the estimate \eqref{eq:G_error_HS}.
Observe that 
for all $G\in \mathbb{R}^{M \times M}$,
$\id_{\mathbb{R}^{M \times M}} G=G \mathbb{I}_{M}$,
and  
for all $G\in \mathbb{R}^{M \times M}$
and $y\in \mathbb{R}^M$,
$$
\langle T_n G,y\rangle 
=\langle  Gu^{(n)},y\rangle 
 = ( Gu^{(n)})^\top y=\tr ( (u^{(n)})^\top  G^\top y)
 =\tr(G^\top y(u^{(n)})^\top )=\langle G,y(u^{(n)})^\top\rangle_{\mathbb{R}^{M \times M}},
$$
which implies that 
$T_n^*: \mathbb{R}^M\to \mathbb{R}^{M \times M}$ is given by  $T_n^* y \coloneqq y (u^{(n)})^\top$ for all 
  $y\in  \mathbb{R}^M$.
  Thus 
  $T_n^*T_n: \mathbb{R}^{M \times M}\to  \mathbb{R}^{M \times M}$
  satisfies 
  $T_n^*T_nG=Gu^{(n)}(u^{(n)})^\top$ for all $G\in \mathbb{R}^{M \times M}$.
  Then  by \eqref{eq:RHH_weightnorm},
 \begin{align*}
   \|G\|^2_{\mathbb{R}^{M \times M},M_N+\lambda \id_{\mathbb{R}^{M \times M}}}
&= 
\left\langle
(M_N+\lambda \id_{\mathbb{R}^{M \times M}} )G,G
\right\rangle_{\mathbb{R}^{M \times M}}
\\
&=\left\langle
G\left(\sum_{n=1}^N 
u^{(n)}(u^{(n)})^\top +\lambda \mathbb{I}_M
  \right),G
\right\rangle_{\mathbb{R}^{M \times M}}
.  
 \end{align*}
Let $V_N= \sum_{n=1}^N 
u^{(n)}(u^{(n)})^\top$. 
Then for all $A\in \mathbb{R}^{M \times M}$,   
  $(M_N+\lambda   \id_{\mathbb{R}^{M \times M}} )A =A(V_N+ \lambda \mathbb{I}_M)$,
  and 
  $(M_N+\lambda   \id_{\mathbb{R}^{M \times M}} )^{-1}A =A(V_N+\lambda  \mathbb{I}_M)^{-1}$. As $V_N+\lambda \mathbb{I}_M$ is symmetric and positive definite, 
for all $G\in \mathbb{R}^{M \times M}$,
\begin{align}
\label{eq:RHH_weightednorm_V}
\begin{split}
 \|G\|^2_{\mathbb{R}^{M \times M},M_N+\lambda \id_{\mathbb{R}^{M \times M}}}
& =\tr\left((V_N+\lambda \mathbb{I}_M)G^\top G\right)
=\tr\left((V_N+\lambda \mathbb{I}_M)^{\frac{1}{2}} G^\top G(V_N+\lambda \mathbb{I}_M)^{\frac{1}{2}} \right)   
\\
& 
= \left\|G(V_N+\lambda \mathbb{I}_M)^{\frac{1}{2}}   \right\|^2_{\mathbb{R}^{M \times M}},  
\end{split}
\end{align}
and similarly, 
\begin{align}
\label{eq:RHH_weightednorm_V_inverse}
\begin{split}
 \|G\|_{\mathbb{R}^{M \times M},(M_N+\lambda \id_{\mathbb{R}^{M \times M}})^{-1}}
&  
= \left\|G(V_N+\lambda \mathbb{I}_M)^{-\frac{1}{2}}   \right\|_{\mathbb{R}^{M \times M}}. 
\end{split}
\end{align}
It remains to compute the 
   Fredholm determinant
   $\det \left( \lambda^{-1} M_{N }+\id_{\mathbb{R}^{M \times M}} \right)$.
   For each $i,j=1,\ldots, M$,
   let $E_{ij}\in \mathbb{R}^{M \times M}$ be the matrix such that the  
   $(i,j)$-th entry is  $1$ and  all   remaining entries are zero.
Then  $(E_{ij})_{i,j=1}^M$ is an orthonormal basis of $(\mathbb{R}^{M \times M}, \|\cdot\|_{\mathbb{R}^{M \times M}})$, and   the   
   Fredholm determinant of 
$\lambda^{-1} M_{N }+\id_{\mathbb{R}^{M \times M}}$ can be computed using  its matrix representation 
with respect to the basis  $(E_{ij})_{i,j=1}^M$.
Indeed,  by \cite[Theorem 3.2 p.~117]{gohberg1990classes},
    \begin{align*}
 \det \left( \lambda^{-1} M_{N }+\id_{\mathbb{R}^{M \times M}} \right)
 &=\det\left( \left(\delta_{ij,i'j'}+\lambda^{-1}\langle 
 M_N E_{ij}, E_{i'j'}
 \rangle_{\mathbb{R}^{M \times M}}\right)_{i,j,i',j'=1}^M
 \right)
 \\
 &=
 \det\left( \left(\delta_{ij,i'j'}+\lambda^{-1}\langle 
  E_{ij}V_N, E_{i'j'}
 \rangle_{\mathbb{R}^{M \times M}}\right)_{i,j,i',j'=1}^M
 \right),
\end{align*}
where $\delta_{ij,i'j'}$ is the Kronecker's delta (i.e., $\delta_{ij,i'j'} =1$ if $i=i'$ and $j =j'$ and $0$ otherwise). A direct computation
shows that 
for all $\ell,k=1,\ldots, M$,
$(E_{ij}V_N)_{\ell, k}=\delta_{i,\ell}(V_N)_{j,k}$,
{\color{black}$((E_{ij}V_N)^\top)_{\ell, k}(E_{i'j'})_{k,\ell}
=\delta_{i,k}(V_N)_{j,\ell }\delta_{i',k}\delta_{j',\ell}$},
and, hence,  
$\langle 
  E_{ij}V_N, E_{i'j'}
 \rangle_{\mathbb{R}^{M \times M}}=\delta_{i,i'}(V_N)_{j,j' }$.
Thus 
$$
\left(\delta_{ij,i'j'}+\lambda^{-1}\langle 
  E_{ij}V_N, E_{i'j'}
 \rangle_{\mathbb{R}^{M \times M}}\right)_{i,j,i',j'=1}^M
=\operatorname{diag}(
\mathbb{I}_M+\lambda^{-1}V_N,\cdots ,
  \mathbb{I}_M+\lambda^{-1}V_N
  )\in \mathbb{R}^{M^2\times M^2},
 $$
which implies that 
$$ \det \left( \lambda^{-1} M_{N }+\id_{\mathbb{R}^{M \times M}} \right)
=(\det(\mathbb{I}_M+\lambda^{-1}V_N ))^M
=(\lambda^{-M}\det(\lambda\mathbb{I}_M+V_N ))^M.
$$ 
This along with 
\eqref{eq:G_error_HS} and \eqref{eq:RHH_weightednorm_V} shows that
\begin{align}
\label{eq:unconstrain_G_error}
\begin{split}
&\left\|(G_{N,\lambda}  - G^\star)(\lambda \mathbb{I}_M+V_N)^{\frac{1}{2}}   \right\|_{\mathbb{R}^{M \times M}}
\\
&\quad \le  R
\left(
2\log\left(
\frac{
\sqrt{
\lambda^{-M^2} (\det(\lambda\mathbb{I}_M+V_N ))^{M}}
}{\delta}
\right)\right)^{\frac{1}{2}}
+\lambda  \left \|    G^\star
     (V_N+\lambda \mathbb{I}_M)^{-\frac{1}{2}}\right\|_{\mathbb{R}^{M \times M}}
\\
&\quad  =  R
\left(
 \log\left(
\frac{
\lambda^{-M^2} (\det(\lambda\mathbb{I}_M+V_N ))^{M}
}{\delta^2}
\right)\right)^{\frac{1}{2}}
+\lambda  \left \|    G^\star
  (V_N+\lambda \mathbb{I}_M)^{-\frac{1}{2}}  \right\|_{\mathbb{R}^{M \times M}}.
\end{split}
\end{align}
This 
proves the desired estimate.
\end{proof}

\begin{proof}[Proof of Theorem \ref{thm:confidence_convolution}] 
Let  $(\mathcal{G}_i)_{i=0}^{N-1}$ be  the filtration given   in Definition \ref{def:offline_data}.
As $u^{(n)}$ is $\mathcal{G}_{n-1}$-measurable, 
$U_n $ is $\mathcal{G}_{n-1}$-measurable    
and $y^{(n)}$ is $\mathcal{G}_{n}$-measurable.   
 Hence by  
Theorem \ref{thm:confidence_interval}
(with $X=Y =\mathbb{R}^M$),
for each $n\in \mathbb{N}$ and $\lambda>0$,
\begin{align*}
&\left \|  W_{N,\lambda}^{\frac{1}{2}} ({K}_{N,\lambda}  - K^\star) 
    \right\|
    \\
&\quad 
\le  R
\left(
 \log\left(
\frac{ \det \left( \lambda^{-1} \sum_{n=1}^N U^\top_n U_n+\mathbb{I}_H \right) }{\delta^2}
\right)\right)^{\frac{1}{2}}
+
\lambda    \left\|W_{N,\lambda}^{\frac{1}{2}} K^\star \right\|
\\
&\quad   
\le  R
\left(
 \log\left(
\frac{ \lambda ^{-M}\det \left( W_{N,\lambda} \right) }{\delta^2}
\right)\right)^{\frac{1}{2}}
+
\lambda    \left\|W_{N,\lambda}^{\frac{1}{2}} K^\star \right\|.
\end{align*}
 This proves the desired estimate.  
\end{proof}

\section{Proof of Theorems  \ref{THM:performance_bound} and \ref{thm-l-bound} and Corollary \ref{cor-quant}}  \label{sec-pf-pessim} 

\begin{proof}[Proof of Theorem \ref{THM:performance_bound}] 
(i) 
Let $\lambda>0$ and $\dl \in (0,1)$. By Theorem \ref{thm:confidence_volterra}, \eqref{bnd-err-1} holds with probability $1-\delta$.  Using the invertibility  of $V_{N,\lambda}$ (recall \eqref{eq:project_admissible}), and Cauchy-Schwarz inequality we get for all $u\in \mathcal{A}_{b}$, 
\be \label{g-diff-est1} 
\begin{aligned}
 \left   \| (G^{\star} - G_{N,\lambda})   u \right \|  
    & = \left \| (G^{\star} - G_{N,\lambda})V_{N,\lambda}^{1/2} V_{N,\lambda}^{-1/2}   u \right\| \\ 
    & \leq \left \| (G^{\star} - G_{N,\lambda})V_{N,\lambda}^{1/2} \right \|_{\mathbb{R}^{M\times M}} \left \| V_{N,\lambda}^{-1/2} u  \right \|  .
\end{aligned}
\ee
From \eqref{bnd-err-1}, \eqref{l-bnd}, \eqref{c-const} and \eqref{g-diff-est1}, we obtain with probability at least $1 - \delta$,
\be \label{g-diff-est2} 
\begin{aligned}
    & \| (G^{\star} - G_{N,\lambda}) u  \| \| u \| \\
    &\leq L^\mathcal{A}_{\eqref{l-bnd}}  \left[ R\left(\log\left(\frac{\lambda^{-M^2} (\det(V_{N\lambda} ))^{M}}{\delta^2}\right)\right)^{\frac{1}{2}}+\lambda  \left \|     G^\star (V_{N,\lambda})^{-\frac{1}{2}}  \right\|_{\mathbb{R}^{M \times M}} \right] \left \| V_{N,\lambda}^{-1/2} u \right \|
     \\ 
     & \leq {\color{black}L^\mathcal{A}_{\eqref{l-bnd}}C_{\eqref{c-const}}(N)  \left \| V_{N,\lambda}^{-1/2} u \right \|} \\
     &= \ell_1(V_{N,\lambda},u),
\end{aligned}
\ee
where we used \eqref{fun_regularization} in the last equality. 

From \eqref{per-func}, \eqref{j-p1}, \eqref{g-diff-est2} and Cauchy-Schwarz inequality, we get (with probability at least $1-\delta$)
\be \label{j-diff1}
 \begin{aligned}
  & J(\hat{u}^{(1)};G^{\star})- J_{P,1} (\hat{u}^{(1)};G_{N,\lam}) \\
  &= \mathbb{E}_{\mathcal D} \left[ \sum_{i=1}^{M}   \sum_{j=1 }^{i} G^{\star}_{i,j}  \hat{u}^{(1)}_{t_j} \hat{u}^{(1)}_{t_i}  \right] - \mathbb{E}_{\mathcal D}\left[ \sum_{i=1}^{M}   \sum_{j=1 }^{i} (G_{N,\lambda})_{i,j} \hat{u}^{(1)}_{t_j} \hat{u}^{(1)}_{t_i} +\ell_1(V_{N,\lambda},\hat{u})    \right] \\
    & = \mathbb{E}_{\mathcal D}\left[  \left \langle (G^{\star} - G_{N,\lambda}) \hat{u}^{(1)},\hat{u}^{(1)} \right \rangle  \right]  - \mathbb{E}_{\mathcal D}\left[  \ell_1(V_{N,\lambda}\hat{u}^{(1)}) \right]\\
    & \leq  \mathbb{E}_{\mathcal D}\left[ \| (G^{\star} - G_{N,\lambda}) \hat{u}^{(1)} \|  \| \hat{u}^{(1)} \| \right]
     -\mathbb{E}_{\mathcal D}\left[  \ell_1(V_{N,\lambda},\hat{u}^{(1)}) \right] \\
    & \leq  \mathbb{E}_{\mathcal D}\left[\ell_1(V_{N,\lambda},\hat{u}^{(1)}) \right]    -\mathbb{E}_{\mathcal D} \left[  \ell_1(V_{N,\lambda},\hat{u}^{(1)}) \right] \\
    &=  0.
\end{aligned}
\ee
Further, we deduce from \eqref{j-diff1} and the optimality of $\hat u^{(1)}$ with respect to  $J_{P,1}( \cdot \,;G_{N,\lam})$ that with probability  $1 - \delta$,
\be \label{i01} 
\begin{aligned}
   J(\hat u^{(1)};G^{\star})- J(u;G^{\star})   & \leq J_{P,1}(\hat u^{(1)};G_{N,\lambda}) -J(u;G^{\star})  \\
   & \leq  J_{P,1}(u;G_{N,\lambda}) -J(u;G^{\star}) \\
   & =  \mathbb{E}_{\mathcal D}\left[ \sum_{i=1}^{M}   \sum_{j=1 }^{i} (G_{N,\lambda})_{i,j} u_{t_j} u_{t_i} +\ell_1(V_{N,\lambda},u)    \right] -\mathbb{E}_{\mathcal D}\left[ \sum_{i=1}^{M}   \sum_{j=1 }^{i} G^{\star}_{i,j}  u_{t_j} u_{t_i}  \right]  \\
   & =   \mathbb{E}_{\mathcal D}\left[  \ell_1(V_{N,\lambda},u) \right] - \mathbb{E}_{\mathcal D}\left[  \left \langle (G^{\star} - G_{N,\lambda}) u, u \right \rangle  \right] \\
& \leq   \mathbb{E}_{\mathcal D}\left[  \ell_1(V_{N,\lambda},u) \right] +\left| \mathbb{E}_{\mathcal D}\left[  \left \langle (G^{\star} - G_{N,\lambda}) u, u \right \rangle  \right] \right|\\
   &\leq 2 \mathbb{E}_{\mathcal D} \left[  \ell_1(V_{N,\lambda},u) \right], \quad \textrm{for all } u \in \mathcal A_b, 
\end{aligned}
\ee
where we used the Cauchy-Schwarz inequality and \eqref{g-diff-est2} to derive the last estimate. 
This completes the proof of (i).

(ii) Let $\lambda>0$ and $\dl \in (0,1)$. By Theorem \ref{thm:confidence_convolution}, \eqref{diff-k}  holds with probability $1-\delta$.  Using the invertibility of $W_{N,\lambda}$ (recall \eqref{w-def}), \eqref{diff-k}, \eqref{l-bnd}, \eqref{eq:cost_conv}, \eqref{c-2-const}  and Cauchy-Schwarz inequality we get for all $u\in \mathcal{A}_{b}$, 
\be \label{ey1}
\begin{aligned}
\| (\mathcal T u) (K^{\star} - K_{N,\lambda})  \|  \|u\|  &\leq  L^\mathcal{A}_{\eqref{l-bnd}} \left   \| \mathcal ( \mathcal{T}u) W_{N,\lambda}^{-1/2} W_{N,\lambda}^{1/2} (K^{\star} - K_{N,\lambda}) \right \| \\ 
    & \leq  L^\mathcal{A}_{\eqref{l-bnd}}  \left \| W_{N,\lambda}^{1/2} (K^{\star} - K_{N,\lambda}) \right \|  \left \| \mathcal Tu W_{N,\lambda}^{-1/2}   \right \|_{\mathbb{R}^{M\times M}} \\ 
    & \leq \ell_2(W_{N,\lambda},u),
\end{aligned}
\ee
where we define a linear map 
$\mathcal T : \mathbb R^{M} \mapsto \mathbb R^{M\times M}$ 
such that 
  $U=\mathcal T u$, with
$U$ being defined as in \eqref{t-op}.

From \eqref{per-func}, \eqref{j-p2}, \eqref{t-op}, \eqref{ey1} and Cauchy-Schwarz inequality, it follows that with probability at least $1 - \delta$, 
\begin{align*}
  & J(\hat{u}^{(2)};K^{\star})- J_{P,2} (\hat{u}^{(2)};K_{N,\lam}) \\
 &= \mathbb{E}_{\mathcal D}\left[ \sum_{i=1}^{M}   \sum_{j=1 }^{i} K^{\star}_{i-j}  \hat{u}^{(2)}_{t_j} \hat{u}^{(2)}_{t_i}  \right] - \mathbb{E}_{\mathcal D}\left[ \sum_{i=1}^{M}   \sum_{j=1 }^{i} (K_{N,\lambda})_{i-j}  \hat{u}^{(2)}_{t_j}\hat{u}^{(2)}_{t_i} +\ell_2(W_{N,\lambda},\hat{u}^{(2)})    \right] \\
    & = \mathbb{E}_{\mathcal D}\left[  \left \langle \mathcal T \hat{u}^{(2)} (K^{\star} - K_{N,\lambda}) ,   \hat{u}^{(2)}\right \rangle \right]  - \mathbb{E}_{\mathcal D}\left[  {\ell}_2(W_{N,\lambda},\hat{u}^{(2)}) \right]\\
    & \leq  \mathbb{E}_{\mathcal D}\left[  \| \mathcal T \hat{u}^{(2)}  (K^{\star} - K_{N,\lambda})  \|   \|  \hat{u}^{(2)}  \| \right]
     -\mathbb{E}_{\mathcal D}\left[  {\ell}_2(W_{N,\lambda},\hat{u}^{(2)}) \right] \\
    & \leq  \mathbb{E}_{\mathcal D}\left[{\ell}_2(W_{N,\lambda},\hat{u}^{(2)}) \right]    -\mathbb{E}\left[ {\ell}_2(W_{N,\lambda},\tilde{u}) \right] =  0.
\end{align*}
Then (ii) follows by applying similar steps as in \eqref{i01}. 
\end{proof}

\begin{proof}[Proof of Corollary \ref{cor-quant}]
(i) Repeating similar steps as in \eqref{j-diff1} we have with $\mathbb P'$-probability at least $1-\dl$, for any $u\in \mathcal A_b$, 
\be  
 \begin{aligned}
  & J(u;G^{\star})- J  (u;G_{N,\lam}) \\
  &= \mathbb{E}_{\mathcal D}\left[ \sum_{i=1}^{M}   \sum_{j=1 }^{i} G^{\star}_{i,j}  u_{t_j}{u}_{t_i}  \right] - \mathbb{E}_{\mathcal D}\left[ \sum_{i=1}^{M}   \sum_{j=1 }^{i} (G_{N,\lambda})_{i,j} u_{t_j} u_{t_i}    \right] \\
    & = \mathbb{E}_{\mathcal D}\left[  \left \langle (G^{\star} - G_{N,\lambda}) u,u \right \rangle  \right]   \\
    & \leq  \mathbb{E}_{\mathcal D}\left[ \| (G^{\star} - G_{N,\lambda})u \|  \| u \| \right]\\
    & \leq  \mathbb{E}_{\mathcal D}\left[\ell(V_{N,\lambda},u) \right], 
\end{aligned}
\ee
which completes the proof. 

(ii) the proof follows similar lines as the proof of (i), hence it is omitted. 
\end{proof}

\begin{proof}[Proof of Theorem \ref{thm-l-bound}] 
(i) Let $\lam >0$ and $ \delta \in (0,1)$. 
Then there exists an event $A_\delta$
with $\mathbb{P}'(A_\delta)\ge 1-2\delta$,
on which both 
Theorem \ref{THM:performance_bound}(i)
and Assumption \ref{assum:concentration_inequ}(i) hold. 
We work on the event $A_\delta$ in the subsequent analysis.
From Theorem \ref{THM:performance_bound}(i) it follows that we need to bound $\mathbb{E} \left[  \ell_1(V_{N,\lambda},u) \right]$.  
We note that by using \eqref{fun_regularization} and Jensen's inequality, we can estimate $\E[ \ell_1(V_{N,\lambda},u)]$ as follows, 
\be\label{g1}
 \begin{aligned}
 \E_{\mathcal D}\big[  \ell_1(V_{N,\lambda},u) \big]&  =  L^\mathcal{A}_{\eqref{l-bnd}}C_{\eqref{c-const}}(N) \mathbb{E}_{\mathcal D} \left[ \sqrt{u^{\top} V_{N,\lambda}^{-1} u} \right]  \\
   & \leq L^\mathcal{A}_{\eqref{l-bnd}}C_{\eqref{c-const}}(N) \sqrt{\mathbb{E}_{\mathcal D} \left[ u^{\top} V_{N,\lambda}^{-1} u  \right]}. 
\end{aligned}
\ee
By \eqref{eq:project_admissible}
 and Assumption \ref{assum:concentration_inequ}(i), on the event $A_\delta$,
\be \label{g2} 
V_{N,\lambda}^{-1} =    \left( \sum_{n = 1}^{N} u^{(n)} (u^{(n)})^{\top}  + \lambda \mathbb{I}_{M} \right)^{-1} \leq  \left( N \Sigma - C_{\ref{assum:concentration_inequ}}(\dl) \sqrt{N} \mathbb{I}_{M} + \lambda  \mathbb{I}_{M} \right)^{-1} \leq  N^{-1} \Sigma^{-1},
\ee
where we used $\lam =C_{\ref{assum:concentration_inequ}}(\delta) \sqrt{N} $ by the hypothesis of the theorem. Inequality \eqref{g2} is to be interpreted as follows, 
\begin{align*}
    x^{\top} V_{N,\lambda}^{-1} x \leq  x^{\top} \left( N \Sigma \right)^{-1} x, \quad \textrm{for all } x \in \mathbb{R}^M. 
\end{align*}
Note that $\Sigma^{-1}$ exists since $\Sigma$ is a symmetric positive-definite matrix. 

By using the Rayleigh quotient bounds on the symmetric positive-definite matrix $\Sigma^{-1}$ we have,  
 \be \label{g3} 
 \|\Sigma^{-1} x \|  \leq  \ul \xi^{-1}_{\Sigma} \|x\|  , \quad \textrm{for all } x \in \mathbb{R}^M, 
\ee
where $ \ul \xi_{\Sigma}$ is the minimal eigenvalue of $\Sigma$, hence $ \ul \xi_{\Sigma}^{-1}$ it is the maximal eigenvalue of $\Sigma^{-1}$. 

From \eqref{g2}, Cauchy-Schwarz inequality and \eqref{g3}, we deduce that with probability $1-\delta$,
\be \label{tf1} 
\begin{aligned}
   \mathbb{E}_{\mathcal D}\left[  u^{\top} V_{N,\lambda}^{-1} u\right] 
& \leq \mathbb{E}_{\mathcal D}\left[ u^{\top} \left( N \Sigma \right)^{-1} u \right]  \\
&=  \frac{1}{N} \mathbb{E}_{\mathcal D} \left[\left \langle   \Sigma ^{-1} u, u \right \rangle \right] \\ 
&\leq  \frac{1}{N} \mathbb{E}_{\mathcal D} \left[ \| \Sigma ^{-1} u \| \| u \| \right] \\ 
&\leq  \frac{1}{N} \ul \xi^{-1}_{\Sigma} \mathbb{E}_{\mathcal D}[\| u \|^2],  \quad \textrm{for all } u \in \mathcal A_b.
 \end{aligned}
\ee

From \eqref{g1}, \eqref{tf1} and \eqref{l-bnd}, we get with probability $1-\delta$,
\be \label{g4} 
\E_{\mathcal D}\big[  \ell_1(V_{N,\lambda},u) \big]    \leq  (L^{\mathcal A}_{\eqref{l-bnd}})^2  \frac{C_{\eqref{c-const}}(N)}{\sqrt{N}} \ul \xi_{\Sigma}^{-1/2}. 
\ee

We conclude from \eqref{g4} that in order to derive an explicit bound on $\E\big[  \ell_1(V_{N,\lambda},u) \big]$ in terms of $N$ we need to further analyse the constant $C_{\eqref{c-const}} (N)$, which was defined in \eqref{c-const}. To do so, we first bound the so called \emph{information gain},  $$I(M,N,\Sigma) \coloneqq \log\left(
\lambda^{-M^2} (\det V_{N,\lam} ))^{M}
\right).$$ Denote the eigenvalues of $\Sigma$ by $\xi_1 \geq \ldots \geq \xi_i \geq \ldots   \geq  \xi_M > 0$. Using \eqref{eq:project_admissible}, Assumption \ref{assum:concentration_inequ}(i) and the monotonicity of the determinant of positive definite matrices we obtain
\be \label{i-m} 
\begin{aligned}
I(M,N,\Sigma) & = M\log \left(
(\det(\lambda \mathbb{I}_{M}))^{-1}\det( V_{N, \lam} )  \right) \\
& \leq M\log \left(
(\det(\lambda \mathbb{I}_{M}))^{-1}\det( N\Sigma + (C_{\ref{assum:concentration_inequ}}(\dl) \sqrt{N} + \lambda) \mathbb{I}_M) \right) \\
&= M    \log \left( \frac{ \prod_{i=1}^M (\lam + C_{\ref{assum:concentration_inequ}}(\dl) \sqrt{N} +N\xi_i) }{\lambda^M}  \right)  \\
&= M  \sum_{i=1}^M  \log \left( \frac{ \lam + C_{\ref{assum:concentration_inequ}}(\dl) \sqrt{N} +N\xi_i }{\lambda}  \right). 
\end{aligned}
\ee
Plugging in our choice of $\lambda  = C_{\ref{assum:concentration_inequ}}(\delta) \sqrt{N}$, we get 
\be \label{g7} 
\begin{aligned}
 I(M,N,\Sigma) &\leq  M^2  \log \left( 2 + \frac{ \sqrt{N} \xi_1}{C_{\ref{assum:concentration_inequ}}(\dl)}  \right)  \\
  &\leq  M^2 \left(  \log \left( 2 + \frac{   \xi_1}{C_{\ref{assum:concentration_inequ}}(\dl)}  \right) +\frac{1}{2} \log N \right) \\
& = M^2 \left(  \log \left(2 + \frac{  \ol \xi_\Sigma}  {C_{\ref{assum:concentration_inequ}}(\dl) } \right)+\frac{1}{2} \log N \right),
\end{aligned}
\ee
  where we recall that $\ol \xi_\Sigma =\xi_1$ is the maximal eigenvalue of $\Sigma$. 
From \eqref{g2} and \eqref{frob} we get that, 
\be \label{g8} 
\left \|    (V_{N,\lam}  )^{-1/2}  \right\|^2_{\mathbb{R}^{M \times M} } = 
\textrm{tr}({V_{N,\lam}}^{-1})  \leq N^{-1}\textrm{tr} (\Sigma^{-1}) = N^{-1}\sum_{k=1}^M \xi^{-1}_k\\
\leq M N^{-1}\ul\xi^{-1}_{\Sigma}, 
\ee
where we recall that $ \ul \xi_{\Sigma} = \xi_M$ is the minimal eigenvalue of $\Sigma$.
From \eqref{c-const},  \eqref{i-m},  \eqref{g7} and \eqref{g8} and $\lambda  = C_{\ref{assum:concentration_inequ}}(\delta) \sqrt{N}$, it follows that,
\be   \label{g10} 
 \begin{aligned}
    C_{\eqref{c-const}} (N)  &= R\big( I(M,N,\Sigma) +\log(\dl^{-2})\big)^{1/2} + \lambda   L^\mathscr{G}_{\eqref{l-bnd}}\left \|    (V_{N,\lam}  )^{-1/2}  \right\|_{\mathbb{R}^{M \times M} } \\ 
    & \leq   M R \left(   \log \left(2 + \frac{
     \ol \xi_\Sigma
    }  {C_{\ref{assum:concentration_inequ}}(\dl) } \right)+\frac{1}{2} \log(N) + \log(\delta^{-2}) \right)^{1/2}    \\
&   \quad  \    +L^\mathscr{G}_{\eqref{l-bnd}} C_{\ref{assum:concentration_inequ}}(\delta)  M^{1/2}   \ul\xi^{-1/2}_{\Sigma}.
\end{aligned}
\ee
Item (i) is then a consequence of \eqref{g4} and \eqref{g10}.

(ii)  
Using \eqref{w-def}, Assumption \ref{assum:concentration_inequ}(ii) with $\lambda  = C_{\ref{assum:concentration_inequ}}(\delta) \sqrt{N}$, it follows that $ N\hat \Sigma \leq W_{N,\lam}$.
Together with the Cauchy-Schwarz inequality and \eqref{frob-id}, we deduce that with probability $1-\delta$,
\be \label{tf31} 
\begin{aligned}
   \mathbb{E}_{\mathcal D}\left[  U^{\top} W_{N,\lambda}^{-1} U\right] 
& \leq \mathbb{E}_{\mathcal D}\left[ U^{\top} \left( N \hat \Sigma \right)^{-1} U \right]  \\
&=  \frac{1}{N} \mathbb{E} _{\mathcal D}\left[\left \langle   \hat \Sigma ^{-1} U, U \right \rangle_{\mathbb{R}^{M\times M}} \right] \\ 
&\leq  \frac{1}{N} \mathbb{E}_{\mathcal D} \left[ \| \hat \Sigma ^{-1} U \|_{\mathbb{R}^{M\times M}} \| U \|_{\mathbb{R}^{M\times M}}\right] \\ 
 &\leq  \frac{M}{N} \ul \xi^{-1}_{\Sigma} \mathbb{E}_{\mathcal D}[\| u \|^2],  \quad \textrm{for all } u \in \mathcal A_b,
 \end{aligned}
\ee
where we used $\|U\|_{\mathbb{R}^{M\times M}} \leq M^{1/2} \|u\|$ (cf.~\eqref{t-op}) and \eqref{g3}, which gives $\|\hat \Sigma ^{-1} U \|_{\mathbb{R}^{M\times M}} \leq M^{1/2} \ul \xi^{-1}_{\hat \Sigma} \|u\|$  in the last inequality. 

Following similar lines as in \eqref{g1}--\eqref{g4}, using Theorem \ref{THM:performance_bound}(ii), Assumption \ref{assum:concentration_inequ}(ii) with $\lambda  = C_{\ref{assum:concentration_inequ}}(\delta) \sqrt{N}$ and \eqref{tf31}, we get with probability $1-\delta$,
 \be \label{g11} 
\E_{\mathcal D}\big[  \ell_2(W_{N,\lambda},u) \big]    \leq M^{1/2} (L^{\mathcal A}_{\eqref{l-bnd}})^2  \frac{C_{\eqref{c-2-const}}(N)}{\sqrt{N}} \ul \xi_{\hat \Sigma}^{-1/2}. 
\ee

From \eqref{w-def}, Assumption \ref{assum:concentration_inequ}(ii) with $\lambda  = C_{\ref{assum:concentration_inequ}}(\delta) \sqrt{N}$, it follows that $ \hat \Sigma \leq W_{N,\lam}$. Together with \eqref{frob} we get, 
\be \label{g11.5} 
\left \|    (W_{N,\lam}  )^{-1/2}  \right\|^2_{\mathbb{R}^{M \times M} } = 
\textrm{tr}({W_{N,\lam}}^{-1})  \leq N^{-1}\textrm{tr} (\hat \Sigma^{-1})  
\leq M N^{-1}\ul\xi^{-1}_{\hat \Sigma}, 
\ee
where we recall that $ \ul \xi_{\hat \Sigma}  $ is the minimal eigenvalue of $\hat \Sigma$.

By using \eqref{g11.5} and repeating similar steps that were leading to \eqref{g10}, we get the following bound on the left-hand side of \eqref{c-2-const}, 
\be \label{g12} 
\begin{aligned}
C_{\eqref{c-2-const}} (N) &\leq   M^{1/2} R \left(   \log \left(2 + \frac{  
 \ol \xi_{\hat \Sigma}
}  {C_{\ref{assum:concentration_inequ}}(\dl) } \right)+\frac{1}{2} \log(N) + \log(\delta^{-2}) \right)^{1/2}    \\
&   \quad  \    +L^\mathscr{K}_{\eqref{k-bnd}}C_{\ref{assum:concentration_inequ}}(\delta)  M^{1/2}    \ul\xi^{-1/2}_{\hat \Sigma}. 
\end{aligned}
\ee 
We note that there is a difference in a factor of $\sqrt{M}$ between \eqref{g10} and \eqref{g12}, due to a difference between the powers of $\lam$ in the logarithm terms of \eqref{c-const} and \eqref{c-2-const}, respectively.  From \eqref{g11} and \eqref{g12}, (ii) follows.
\end{proof}

\section{Proof of Theorem \ref{prop-opt-strat}} \label{sec-pf-opt}  

\begin{proof} [Proof of Theorem \ref{prop-opt-strat}] Note that the uniqueness of the optimizer of \eqref{opt-prog} follows immediately from the convexity of the cost functional, see e.g., Theorem 2.3 in \cite{Lehalle-Neum18}. 

Next we derive the solution to \eqref{opt-prog}. Let $\lambda$ be $\mathcal F_{T}$-measurable, square-integrable random variable. We write the Lagrangian of the performance functional in \eqref{per-func} as follows 
\begin{equation} \label{opt-lagran} 
\begin{aligned} 
\mathcal L(u;\lambda) 
=  \mathbb{E}\left[ \frac{1}{2} u^{\top}  (G+G^{\top})u   + u^{\top} A + \lambda  (u^\top \boldsymbol{1}  - x_0) \right].
\end{aligned} 
\end{equation} 
The first order condition with respect to \eqref{opt-lagran} is given by 
\begin{align} 
\nabla_{u}  \mathcal L(u;\lambda) =\mathbb{E}\left[  (G+G^{\top})u   +  A + \lambda \boldsymbol{1}   \right] &=0,  \label{foc}  \\
u^\top \boldsymbol{1}  &= x_0.  \label{constr}
\end{align} 
 Let $ t_i \in \mathbb{T}$, then from \eqref{foc} and the tower property we have 
 $$
 \mathbb{E}\left[   \mathbb{E}_{t_i} \left[ (G+G^{\top})u   +  A + \lambda \boldsymbol{1}   \right] \right] =0.
 $$
 So if 
\begin{equation} \label{proj} 
 \mathbb{E}_{t_i} \left[ (G+G^{\top})u   +  A + \lambda \boldsymbol{1}   \right] =0 ,
\end{equation} 
holds then \eqref{foc} is satisfied. 

Recall that $G_{i,j}=0$ for $j>i$. We rewrite explicitly for each entry in \eqref{proj} (conditioned on $\mathcal F_{t_i}$) to get, 
\begin{equation} \label{gf1} 
 \sum_{ j \leq i }  G_{i,j}u_{t_j} +   \sum_{ i  \leq j \leq  M  }  G_{j,i} \mathbb{E}_{t_i} [u_{t_j}]    +   A_{t_i} +  \mathbb{E}_{t_i}[\lambda]     =0,  \quad \textrm{for all }  i =1,\ldots,M.  
\end{equation} 
Using \eqref{g-dec} we obtain the following condition for the optimal strategy, 
\begin{equation} \label{gf1.1} 
 2\eta  u_{t_i} +  \sum_{ j \leq i }  \wt G_{i,j}u_{t_j} +  \sum_{ i  \leq j \leq  M   }  \wt G_{j,i} \mathbb{E}_{t_i} [u_{t_j}]    +   A_{t_i} +  \mathbb{E}_{t_i}[\lambda]     =0,  \quad \textrm{for all }  i =1,\ldots,M.  
\end{equation}

In the following we fix $t_\ell\in \mathbb T$ such that $t_\ell \leq t_i$. Taking conditional expectation on both sides of \eqref{gf1.1}, using the tower property we get for all $ \ell \leq i$, 
\begin{equation} \label{gf2} 
 2\eta  \mathbb{E}_{t_\ell}[u_{t_i}]  + \sum_{ j < \ell }  \wt G_{i,j} u_{t_j} +   \sum_{ \ell \leq  j \leq i }  \wt G_{i,j}  \mathbb{E}_{t_\ell}[u_{t_j}]     +    \sum_{ i  \leq j \leq  M   } \wt G_{j,i} \mathbb{E}_{t_\ell} [u_{t_j}]    +  \mathbb{E}_{t_\ell}[ A_{t_i}] +  \mathbb{E}_{t_\ell }[\lambda]     =0.  
\end{equation} 
Define $m_{t_\ell} = (m_{t_\ell}(t_1),\ldots,m_{t_\ell}(t_M))^{\top}$ such that 
$$
m_{t_\ell}(t_i)= 1_{\{ t_\ell \leq  t_i \}}   \mathbb{E}_{t_\ell} [ u_{t_i}]. 
$$
By multiplying both sides of \eqref{gf2} by $1_{\{t_\ell \leq  t_i\}}$ we get for all $i=1,\ldots,M$: 
\begin{equation} \label{gf3} 
\begin{aligned} 
0&= 2\eta  m_{t_\ell}(t_i) +1_{\{t_\ell \leq t_i\}}  \sum_{ j < \ell }  
 \wt G_{i,j} u_{t_j} +   1_{\{t_\ell \leq t_i\}}  \sum_{ \ell \leq j \leq i } \wt  G_{i,j} m_{t_\ell}(t_j)    \\ 
&\quad + 1_{\{t_\ell \leq t_i\}}     \sum_{ i  \leq j \leq  M   } \wt G_{j,i} m_{t_\ell}(t_j)   + 1_{\{t_\ell \leq t_i\}} \left( \mathbb{E}_{t_\ell}[ A_{t_i}] +  \mathbb{E}_{t_\ell }[\lambda] \right)   \\
&=f_{t_\ell}(t_i) + 2\eta  m_{t_\ell}(t_i) + 1_{\{t_\ell \leq t_i\}}  \sum_{ \ell \leq j \leq i } \wt G_{i,j} m_{t_\ell}(t_j)    +1_{\{t_\ell \leq t_i\}}     \sum_{ i  \leq j \leq  M   } \wt G_{j,i} m_{t_\ell}(t_j)   \\
&=f_{t_\ell}(t_i) + 2\eta  m_{t_\ell}(t_i) +    \sum_{ \ell \leq j \leq i }  \wt G^{(\ell)}_{i,j} m_{t_\ell}(t_j)    +    \sum_{ i  \leq j \leq  M   } ((\wt G^{\top})^{(\ell)})_{i,j} m_{t_\ell}(t_j) \\
&=f_{t_\ell}(t_i) + 2\eta  m_{t_\ell}(t_i) +    
 \wt G^{(\ell)}_{i, \cdot} m_{t_\ell}    + ((\wt G^{\top})^{(\ell)})_{i,\cdot} m_{t_\ell}, 
\end{aligned} 
\end{equation} 
where $\wt G^{(\ell)}_{i, \cdot}$ is the $i$-th row of the matrix $ \wt G^{(\ell)}$ which was introduced in \eqref{tile-g-i} and
\be \label{f-proc} 
f_{t_\ell}(t_i) :=   1_{\{t_\ell \leq t_i\}}  \sum_{ j < \ell }  \wt G_{i,j} u_{t_j}  + 1_{\{t_\ell \leq t_i\}}  \left( \mathbb{E}_{t_\ell}[ A_{t_i}] +  \mathbb{E}_{t_\ell }[\lambda]  \right).
\ee

Let $f_{t_\ell} =  (f_{t_\ell}(t_1),\ldots,f_{t_\ell}(t_M))^{\top}$, then from \eqref{gf3} we get,
\begin{equation} \label{gf4} 
\left( 2\eta \mathbb I_M+ \wt G^{(\ell)}    +       (\wt G^{\top})^{(\ell)}  \right) m_{t_\ell}  = -f_{t_\ell}.
 \end{equation} 
We present the following technical lemma, which will be proved at the end of this section.  
\begin{lemma} \label{lemma-d-inv} 
The matrix 
$$D^{(\ell)} =  2\eta \mathbb I_M+ \big(\wt G    +     \wt G^{\top} \big)^{(\ell)}  $$ 
is invertible for any $\ell=1,\ldots,M$. 
\end{lemma} 

From \eqref{gf4} and Lemma \ref{lemma-d-inv} we get, 
\begin{equation} \label{gf5} 
 m_{t_\ell}  = - \left( D^{(\ell)} \right)^{-1}  f_{t_\ell}, \quad  \textrm{for all } \ell =1,\ldots,M,  
 \end{equation}
We now plug-in \eqref{gf5} to \eqref{gf1.1} to get an equation for $u_{t_i}$  
\begin{equation} \label{gf6} 
2\eta  u_{t_i} +  \sum_{ j \leq i }  \wt G_{i,j}u_{t_j} -   \sum_{ i  \leq j \leq  M   }  \wt G_{j,i} (D^{(i)})^{-1}_{j, \cdot} f_{t_i}    +   A_{t_i} +  \mathbb{E}_{t_i}[\lambda]     =0.
\end{equation} 

We first consider the second and third terms on the left-hand side of \eqref{gf6}. Using \eqref{f-proc} we get, 
\allowdisplaybreaks
\begin{equation} \label{gf7} 
\begin{aligned}
& \sum_{ j \leq i } \wt G_{i,j}u_{t_j} - \sum_{ i  \leq j \leq  M   }   \wt G_{j,i} (D^{(i)})^{-1}_{j, \cdot} f_{t_i} \\
 &= \sum_{ j \leq i }  \wt G_{i,j}u_{t_j}  - \sum_{ i  \leq j \leq  M   } \wt G_{j,i} \sum_{k=1}^M (D^{(i)})^{-1}_{j, k} f_{t_i}(t_k) \\
& =   \sum_{ j \leq i }  \wt G_{i,j}u_{t_j}  -  \sum_{ i  \leq j \leq  M   } \wt  G_{j,i} \sum_{k=1}^M (D^{(i)})^{-1}_{j, k} \left( 1_{\{t_i \leq t_k\}}  \sum_{ \ell < i }  \wt G_{k,\ell} u_{t_\ell}  + 1_{\{t_i \leq t_k\}} \big(  \mathbb{E}_{t_i}[ A_{t_k}] +  \mathbb{E}_{t_i }[\lambda] \big)\right) \\
&=  \sum_{ j \leq i }  \wt G_{i,j}u_{t_j}  -  \sum_{ i  \leq j \leq  M   }  \wt G_{j,i} \sum_{k=1}^M (D^{(i)})^{-1}_{j, k}  1_{\{t_i \leq t_k\}}  \sum_{ \ell < i }  \wt G_{k,\ell} u_{t_\ell}  \\
&\quad - \sum_{ i  \leq j \leq  M   }  \wt G_{j,i} \sum_{k=1}^M (D^{(i)})^{-1}_{j, k} 1_{\{t_i \leq t_k\}} \left(     \mathbb{E}_{t_i}[ A_{t_k}] +  \mathbb{E}_{t_i }[\lambda] \right) \\
&=\sum_{j\leq i} K_{i,j}u_j  -\sum_{ i  \leq j \leq  M   }   \wt G_{j,i} \sum_{k=1}^M (D^{(i)})^{-1}_{j, k} 1_{\{t_i \leq t_k\}}   \left(   \mathbb{E}_{t_i}[ A_{t_k}] +   \mathbb{E}_{t_i }[\lambda] \right), \\
\end{aligned}
\end{equation} 
where $K$ is the following lower triangular matrix, 
\be \label{k-def} 
\begin{aligned}
K_{i,\ell} & \coloneqq  \wt G_{i,\ell} -1_{\{t_\ell < t_i\}}   \sum_{i \leq q \leq M} \wt G_{q,i} \sum_{k=1}^{M} (D^{(i)})^{-1}_{q,k} \mathrm{1}_{\lbrace t_i \leq t_k \rbrace} \wt G_{k,\ell}  \\
 & =  \wt G_{i,\ell} - \mathrm{1}_{\lbrace t_\ell <  t_i \rbrace}  \sum_{k=1}^{M}  \sum_{1 \leq q \leq M} \wt G_{q,i}(D^{(i)})^{-1}_{q,k} \wt G^{(i)}_{k,\ell} \\
     & =   \wt G_{i,l} -   \mathrm{1}_{\lbrace t_\ell <  t_i \rbrace}  \sum_{k=1}^{M}(\wt G^{\top} (D^{(i)})^{-1})_{i,k}\wt G^{(i)}_{k,\ell}  \\
          & =  \wt G_{i,\ell} -   \mathrm{1}_{\lbrace t_\ell <  t_i \rbrace} (\wt G^{\top} (D^{(i)})^{-1} \wt G^{(i)})_{i,\ell },
          \end{aligned}
\ee
which agrees with \eqref{k-def-og}.

From \eqref{gf6} and \eqref{gf7} we get, 
\begin{equation} \label{gf8} 
\begin{aligned}
2\eta  u_{t_i}+    \sum_{j\leq i} K_{i,j} u_{t_j}& =  \sum_{k=1}^M \sum_{ 1 \leq j \leq  M   }  \wt G_{j,i} (D^{(i)})^{-1}_{j, k} 1_{\{t_i \leq  t_k\}}  \left(    \mathbb{E}_{t_i}[ A_{t_k}] +    \mathbb{E}_{t_i }[\lambda] \right)  - A_{t_i} - \mathbb{E}_{t_i}[\lambda] \\
  & = \sum_{k=1}^M  ( \wt G^{\top}(D^{(i)})^{-1})_{i,k}  1_{\{t_i \leq  t_k\}} \left(   \mathbb{E}_{t_i}[ A_{t_k}] +  \mathbb{E}_{t_i }[\lambda] \right)  - A_{t_i} - \mathbb{E}_{t_i}[\lambda] \\    
        & =  -A_{t_i} + \sum_{k=1}^M  ( \wt G^{\top} (D^{(i)})^{-1})_{i,k}    1_{\{t_i \leq  t_k\}}   \mathbb{E}_{t_i}[ A_{t_k}]  \\
        &\quad   - \left(1-\sum_{k=1}^M  ( \wt G ^{\top} (D^{(i)})^{-1})_{i,k}1_{\{t_i \leq  t_k\}} \right) \mathbb{E}_{t_i}[\lambda] . 
\end{aligned}
\end{equation}
We define $g=(g_{t_1},\ldots,g_{t_M})^\top$ and $a=(a_{t_1},\ldots,a_{t_M})^\top$ as follows,   
\begin{equation} \label{} 
\begin{aligned}
g_{t_i} &=  -A_{t_i} +\sum_{k=1}^M  ( \wt G^{\top} (D^{(i)})^{-1})_{i,k}   1_{\{t_i \leq  t_k\}}  \mathbb{E}_{t_i}[ A_{t_k}],    \\ 
a_{t_i} &=   -1+\sum_{k=1}^M  ( \wt G^{\top} (D^{(i)})^{-1})_{i,k}1_{\{t_i \leq   t_k\}}.
\end{aligned}
\end{equation}
Writing \eqref{gf8} in matrix form gives   
\begin{equation} \label{gf9} 
  (2\eta  \mathbb{I}_M +  K) u  =  g + a^\top \mathbb{I}_M \mathbb{E}_{\cdot }[\lambda],  \\
 \end{equation}
where we introduce the notation 
$$ 
\mathbb{E}_{\cdot }[\lambda]  = (\mathbb{E}_{t_1}[\lambda], \ldots,\mathbb{E}_{t_M}[\lambda])^\top, 
$$
so that, 
$$
a^\top  \mathbb{I}_M \mathbb{E}_{\cdot }[\lambda] =  (a_{t_1}  \mathbb{E}_{t_1 }[\lambda], \ldots, a_{t_M}  \mathbb{E}_{t_M }[\lambda])^{\top}.
$$
From \eqref{k-def} it follows that $2\eta  \mathbb{I}_M +  K$ is lower triangular matrix and 
$$
(2\eta  \mathbb{I}_M +  K)_{j,j} = 2\eta + \tilde G_{j,j} > 0,   \quad \textrm{for all } j=1,\ldots,M, 
$$
hence it is invertible. We therefore get from \eqref{gf9},
\begin{equation} \label{gf10} 
    u =  (2\eta  \mathbb{I}_M +  K)^{-1} (g + a^\top  \mathbb{I}_M  \mathbb{E}_{\cdot }[\lambda]). 
\end{equation}
Let 
$$
\wt A  = (2\eta  \mathbb{I}_M +  K)^{-1}  g, 
$$
and note that $\wt A_{t_i}$ is $\mathcal F_{t_i}$ measurable since $(2\eta  \mathbb{I}_M +  K)$ and hence $(2\eta  \mathbb{I}_M +  K)^{-1}$ are lower triangular matrices. 
We can rewrite \eqref{gf10} as follows: 
$$
u_{t_i} = \wt A_{t_i} +   \sum_{j \leq i} (2\eta  \mathbb{I}_M +  K)^{-1}_{i,j} a_{t_j}  \mathbb{E}_{t_j }[\lambda]. 
$$
By using the tower property we get for any $r< i$, 
\begin{align*} 
\E_{t_r} [u_{t_i}]  &= \E_{t_r} [\wt A_{t_i}] +   \sum_{r \leq j \leq i} (2\eta  \mathbb{I}_M +  K)^{-1}_{i,j} a_{t_j}  \mathbb{E}_{t_r }[\lambda]+\sum_{j < r  } (2\eta  \mathbb{I}_M +  K)^{-1}_{i,j} a_{t_j}  \mathbb{E}_{t_j }[\lambda].
\end{align*} 

Summing over $i$ and using $\sum_{i=1}^M\E_{t_r} [u_{t_i}] =x_0$, which follows from \eqref{constr}, we get  
\begin{align*} 
x_0
 &= \sum_{i=1}^M \E_{t_r} [\wt A_{t_i}] +    \mathbb{E}_{t_r }[\lambda]\sum_{i=1}^M {\color{black} \sum_{r  \leq j \leq M} }(2\eta  \mathbb{I}_M +  K)^{-1}_{i,j} a_{t_j} +\sum_{i=1}^M \sum_{  j <r  }(2\eta  \mathbb{I}_M +  K)^{-1}_{i,j}  a_{t_j}  \mathbb{E}_{t_j }[\lambda]. 
\end{align*} 
It follows that $ \mathbb{E}_{t_r } [\lambda]$ satisfies the following forward equation: 
\begin{align*} 
  c_r +b_r \mathbb{E}_{t_r }[\lambda]+ \sum_{ j <r  } h_j   \mathbb{E}_{t_j }[\lambda] =x_0, \quad r= 1, \ldots, M.
\end{align*} 
 where 
\begin{align*} 
c_r &=   \sum_{i=1}^M \E_{t_r} [\wt A_{t_i}] ,  \quad b_{r}  =\sum_{i=1}^N {\color{black}\sum_{r  \leq j \leq M}} (2\eta  \mathbb{I}_M +  K)^{-1}_{i,j} a_{t_j}, \quad  h_j = \sum_{i=1}^M(2\eta  \mathbb{I}_M +  K)^{-1}_{i,j} a_{t_j},
\end{align*} 
which admits the following solution: 
$$
\mathbb{E}_{t_{r+1}}[\lam]= \frac{(b_r-h_r) \mathbb{E}_{t_{r}}[\lam]- (c_{r+1}-c_r)}{b_{r+1}} ,  \quad   \mathbb{E}_{t_1}[\lam] = \frac{x_0-c_1}{b_1}.
$$

\end{proof}

\begin{proof} [Proof of Lemma \ref{lemma-d-inv}]
From \eqref{g-def} and \eqref{g-dec} if follows that $\wt G + \wt G^{\top}$ is symmetric nonnegative definite matrix. Therefore, $D^{(1)}$ is symmetric positive definite, hence it is invertible, or equivalently its determinant is nonzero. Now, $D^{(2)}$ is the same as $D^{(1)}$ except that the off-diagonal entries of the first row, $D^{(2)}_{1,j}$, for $j =2, \ldots, M$  are zero. In general, for $\ell \geq 2$, the entries $D^{(\ell)}_{k,j}$, for $k= 1, \ldots, \ell-1$ and $j \neq  k$, of $D^{(\ell)}$ are zero and equal to $D^{(1)}$ otherwise. Therefore, the determinant of $D^{(2)}$ can be expressed as $$\mathrm{det}(D^{(2)}) = D^{(2)}_{1,1}\mathrm{det}(D^{(2)}_{2:M,2:M}),$$
where $D^{(2)}_{2:M,2:M}$ is the symmetric $\mathbb{R}^{(M-1) \times (M-1)}$-matrix obtained by deleting the first row and first column of $D^{(2)}$ (or equivalently the first row and first column of $D^{(1)}$). Now, observe that $$D^{(2)}_{1,1}\mathrm{det}(D^{(2)}_{2:M,2:M}) \neq 0,$$
since $D^{(2)}_{1,1} \neq 0$ and $\mathrm{det}(D^{(2)}_{2:M,2:M}) \neq 0$, which is a consequence of 
$ x^{\top}D^{(1)} x > 0$, for all $x \in \mathbb{R}^{M}$. To see this, set the first component of $x \in \mathbb{R}^M$ equal to zero which implies $\tilde{x}^{\top} D^{(2)}_{2:M,2:M} \tilde{x} > 0$ for any $\tilde{x} \in \mathbb{R}^{M-1}$, and thus $D^{(2)}_{2:M,2:M}$ is invertible. A similar argument holds for the matrices $D^{(\ell)}$, $\ell >2$.
\end{proof} 

\appendix

\section{Illustrative examples} \label{sec-examples} 
In the following example, we demonstrate how spurious correlation (i) in \eqref{sub-decom} could lead to unfavourable costs. 
\begin{example} [Spurious correlation for propagator estimation]   \label{exmp-corr} 
As described in Section \ref{sec-mot-res}, we consider the scenario of underestimating the price impact, and we capture the estimation error between the true propagator $G^{\star}$ to the estimator $\hat G$ as follows, 
\be \label{g-err} 
 G^{\star}- \hat G >  \Delta \mathbb{I}_M,
\ee
 for some constant $\Delta>0$. The above inequality is understood as partial order defined by the convex cone of positive semi-definite matrices (see Definition \ref{ineq-matrix}). Note that in \eqref{g-err} we neglect the error of off-diagonal elements for convenience. Including these error terms would clearly make the spurious correlation larger. 

From \eqref{sub-decom} it follows that the cost created by the spurious correlation of a greedy strategy $\hat u$ is given by, 
\be\label{diff-j}
\begin{aligned}
\textrm{SpurCor} &= 
 J(\hat{u};G^{\star}) - J(\hat{u};\hat{G}) \\ 
 &=   \mathbb{E}\left[ \sum_{i=1}^{M}  \sum_{j=1 }^{i} ( G^{\star}_{i,j} - \hat{G}_{i,j})   \hat{u}_{t_j} \hat{u}_{t_i}  \right] \\
 &= \E \left[\langle ( G^{\star}  - \hat{G} ) \hat u, \hat u \rangle \right] \\
 & \geq \Delta {\color{black} \mathbb{E}[\|\hat{u}\|^2]   }.
 \end{aligned}
 \ee
Note however that underestimating the propagator $ G^{\star}$ will lead to underestimation of price impact costs and will result in a faster execution strategy and hence create additional trading costs. This monotonicity relation is demonstrated in Figure 4 of \cite{AJ-N-2022}, where it is shown that smaller propagators allow for faster execution rate.  Next, we derive a lower bound for the spurious correlation contribution in \eqref{diff-j}, by considering a special case that will help to simplify the computations and keep this example tractable. We assume that the price impact is temporary, i.e., $\hat G =   \hat \kappa \mathbb{I}_M$, and $ G^\star =     \kappa \mathbb{I}_M$, for some constants $0<\hat \kappa < \kappa$. Note that in \eqref{diff-j} we have $\Delta =    \kappa - \hat \kappa$. We further make the assumption that the unobserved asset price has the following dynamics, 
$$
P_t = \mu t + \sigma W_t, 
$$
where $W_t$ is a Brownian motion and $\mu, \sigma> 0$ are constants. Specifically $\mu$ is a trend that can be regarded as a deterministic signal. In this case one can solve the continuous time analog of the execution problem \eqref{per-func} (see \cite[Exercise E.6.3, Chapter 6.9]{cartea15book}) to get, 
$$
u_t^{\star} = \frac{x_0}{T} +  \frac{\mu}{4\kappa}(T-2t), \quad 0\leq t \leq T. 
$$
In \cite[Table 2 , Section 4.3]{collin.al.20}, the price impact coefficient was estimated from proprietary dataset of real transactions executed by a large investment bank, where it was found that $\kappa \approx 10^{-10}$. Since in a bullish market a daily stock return of $2\%$ is quite common we have $\mu T \approx  0.02$. Therefore, if we wish to execute $x_0=100$ stocks within one day of $8$ hours of trading and the basic time unit is seconds, we have $T=28800$ and it is clear that $\frac{x_0}{T} \ll  \frac{\mu}{2\hat \kappa}(T-2t)$, except for a small interval $T/2 \pm O(\kappa/\mu)$. Outside this interval we have, 
\be \label{ratio-strat} 
\frac{\hat u_t}{u_t^{\star}} = \frac{ \frac{x_0}{T} + \frac{\mu}{2\hat \kappa}(T-2t)}{ \frac{x_0}{T} + \frac{\mu}{2\kappa}(T-2t)}  \approx \frac{\kappa}{\hat \kappa}. 
\ee
Hence from \eqref{ratio-strat} and \eqref{diff-j} it follows that 
\be\label{exmp-spur}
\begin{aligned}
\textrm{SpurCor} &= \E \left[\langle ( G^{\star}  - \hat{G} ) \hat u, \hat u \rangle \right] 
 & \geq (\kappa -\hat \kappa) \int_{0}^T (\hat u_t)^2 \mathrm{d}t 
  &\approx (\kappa -\hat \kappa) \left( \frac{\kappa}{\hat{\kappa}}\right)^2 \int_{0}^T (u^{\star}_t)^2 \mathrm{d}t,
  \end{aligned}
 \ee
 where we have neglected the aforementioned small interval in which $\hat u_t/u_t^{\star} \approx 1$ with a smaller order contribution. 

 We learn from \eqref{exmp-spur} that underestimating the price impact has an amplifying effect of order $(\kappa/\hat\kappa)^2 >1$ on the transaction costs when a greedy strategy is implemented. 

In fact in this simple case we can compare the pessimistic strategy $u^{(1)}$ to the greedy strategy $\hat u$. Recall that, 
$$
 J(u;\kappa \mathbb{I}_M  ) = \kappa \int_0^Tu^2_t \mathrm{d}t +\mu \int_0^T u_t \mathrm{d}t ,  
$$
so the (optimal) uncertainty quantifier in Definition \ref{def-quant} takes the form  $\Gamma(u)= \Delta \int_0^T u_t^2 \mathrm{d}t$ and the pessimistic cost functional in \eqref{j-h-def} is 
$$
J(u;\hat\kappa \mathbb{I}_M ) = (\hat\kappa +\Delta)  \int_0^Tu^2_t \mathrm{d}t+  \mu \int_0^T u_t \mathrm{d}t. 
$$
The minimiser $u^{(1)}$ of $J(u;\hat\kappa \mathbb{I} )$ is given by 
$$
u_t^{(1)} = \frac{x_0}{T} + \frac{\mu}{4(\hat \kappa+\Delta)}(T-2t), \quad 0\leq t \leq T. 
$$
Since $\hat \kappa+\Delta = \kappa$ we get $u^{(1)}= u^{\star}$ so the pessimistic strategy outperforms the greedy strategy $\hat u$ due to the optimality of  $u^{\star}$ with respect to $J( \cdot ;\kappa \mathbb{I}_M)$. Note that in practice we do not know $\kappa$ (or $\Delta$). Therefore, we need to choose a $\Gamma$ which may be not be as sharp as in this example. 
In Section \ref{sec-numerics} we present further examples for spurious correlation costs in the presence of trading signals and more general propagators.

\end{example}

We recall that $u^{(1)}$ denotes the optimal strategy with respect to the pessimistic cost functional \eqref{j-h-def}. In the following simple example, we show how our choice of $\ell_1(V_{N,\lambda}, \cdot)$ in \eqref{ell-1-def} as a $\dl$-uncertainty quantifier (see \eqref{j-h-def}) helps to keep the pessimistic optimal strategy $u^{(1)}$ close to trajectories which are frequently observed in the dataset. 
 
\begin{example} [Effect of $\ell_1$-penalty function on order execution] \label{exmp-l1-pen} The role of the penalty $\ell_1$ in  \eqref{ell-1-def} is to discourage the strategies to visit out-of-distribution actions. We will illustrate this effect in the following simple example.   
Consider a dataset of $N$ buy metaorders, such that each order is executed over two even time bins, that is $M=2$. The dataset contains the following strategies $u^{(i)} = (u^{(i)}_1, u_2^{(i)})^\top$, for $i=1,\ldots,N$. It is well known that in buy execution problems without signals we should have 
\be \label{asmp} 
u^{(i)}_1 \gg u^{(i)}_2 \geq 0,  \quad \textrm{for all } i=1,\ldots,N, 
\ee
 due to risk aversion terms and inventory constraints (see e.g., \cite[Chapter 6.5, Figure 6.2]{cartea15book}). For our example we choose for simplicity $\lam=1$ in $V_{N,\lam}$ and hence 
$$
V_{N, 1 } = \sum_{i=1}^N u^{(i)}(u^{(i)})^\top + \mathbb{I}_M  = \begin{pmatrix}
     \sum_{i=1}^N (u_1^{(i)})^2 +1  &  \sum_{i=1}^N u_1^{(i)}  u_2^{(i)}    \\
     \sum_{i=1}^N u_1^{(i)}  u_2^{(i)}   &  \sum_{i=1}^N (u_2^{(i)})^2 +1  
  \end{pmatrix}.
$$
It follows that we can write, 
$$
V_{N,1} =   \begin{pmatrix}
    a_{1} & a_2   \\
     a_2   &a_3 
  \end{pmatrix}, \quad \textrm{for some } 0< a_3 <a_2 < a_1,
$$
where from the fact that $ a_3 <a_2 $ follows from $u^{(i)}_1  \gg u^{(i)}_2$. Our assumptions also imply $a_1a_3 >a_2^2$. Hence,  $V_{N,1}$ is invertible and we have, 
$$
V^{-1}_{N,1} =  \frac{1}{a_1a_3 - a_2^2}   \begin{pmatrix}
    a_{3} &  - a_2   \\
    - a_2   & a_1 
  \end{pmatrix} =:  \begin{pmatrix}
    b_{1} & b_2   \\
     b_2   &b_3 
  \end{pmatrix} \quad \textrm{ where } b_2<0<b_1<b_3.
$$
Now let $v=(v_1,v_2)^\top$ and consider 
$$
\ell^2_1(v, V^{-1}_{N,1}) = v^\top V^{-1}_{N,1} v = b_1 v_1^2 + 2b_2v_1v_2 + b_3v_2^2.
$$
Since according to our assumption we expect $b_1\ll b_3$, it follows that $\ell_1$ will penalize strategies $v$ where $v_1 \ll v_2$, which are not covered by the dataset (see \eqref{asmp}). 
 \end{example}

\section{Volterra kernel estimation for a noisy dataset} \label{sec-examples_noisy} 
In this example, we construct the trading dataset as follows: the strategies $u_{t_i}^{(n)}$ for each $n \in \lbrace 1, \ldots, N\rbrace$ and $i \in \lbrace 1, \ldots, M \rbrace$ are i.i.d. samples from a normal distribution with mean 50 and standard deviation $9$. We consider the propagator $G^{\star}_{i,j} = \frac{\kappa}{(t_i - t_j +1)^{\beta} }$ $i \geq j$, with $\kappa = 0.01$ and $\beta = 0.4$, 
and apply \eqref{eq:G_n_lambda_volterra} 
to estimate it using observed price trajectories.   For the estimation procedure, we set $\lambda = 10^{-3}$, $N= 252$, $M =78$.  
Numerical results are presented in Figure \ref{FIGURE LABEL_e},
which indicate that \eqref{eq:G_n_lambda_volterra} recovers the true propagator reasonably well.

\begin{figure}[!ht]
\centering
\includegraphics[trim=30 5 30 30, clip, width= 0.48\linewidth]{heatmap_shadow_1_final}
\includegraphics[trim=30 5 30 30, clip, width=0.48\linewidth]{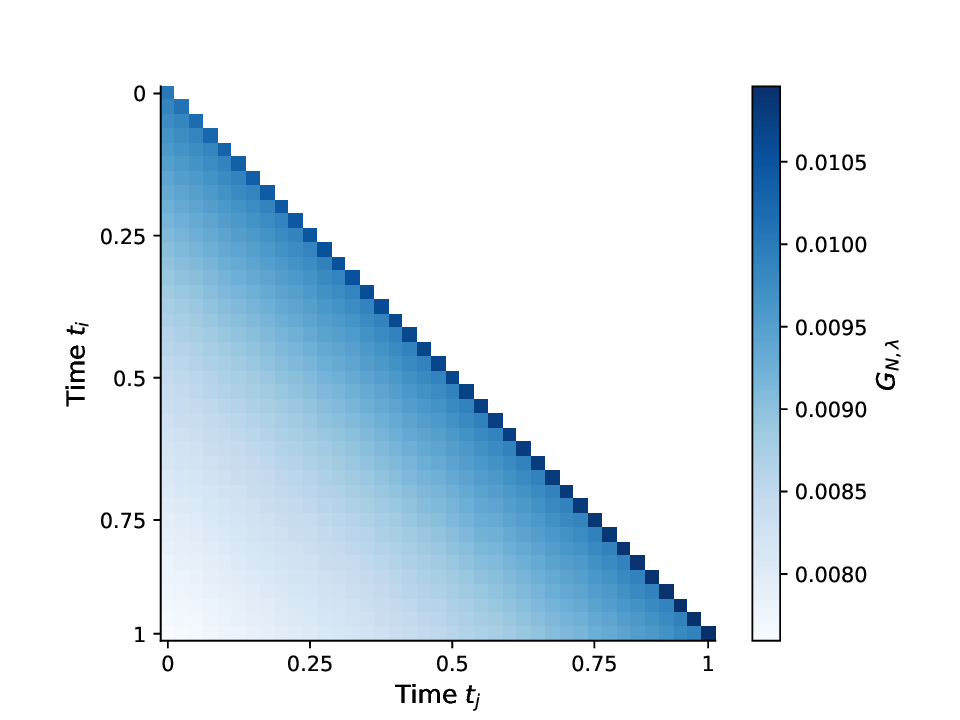}  
\includegraphics[trim=30 5 30 30, clip, width=0.48\linewidth]{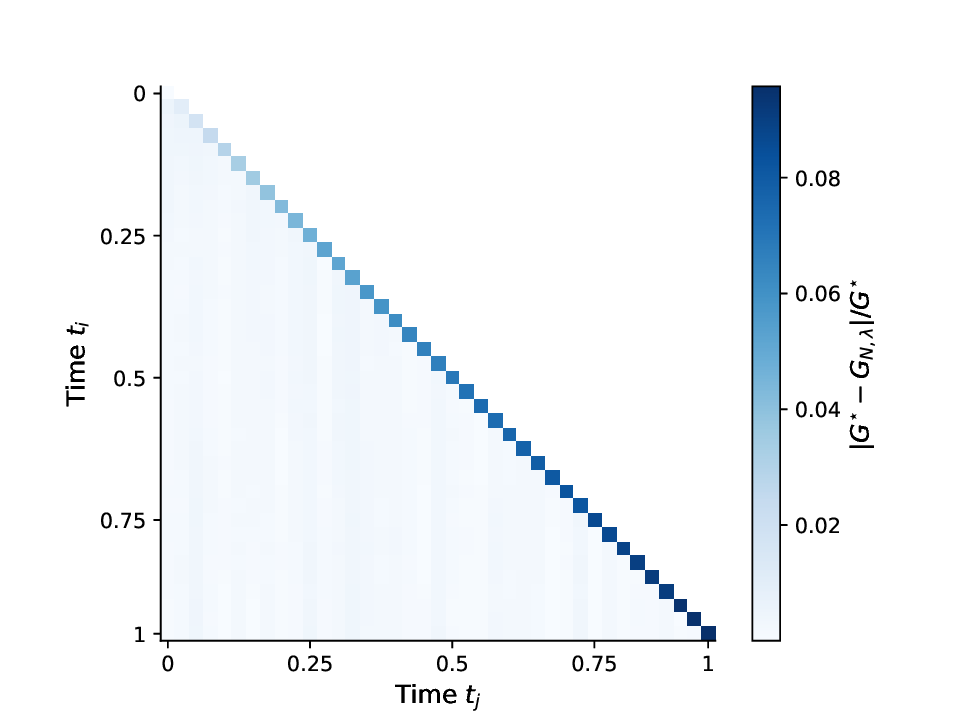}
\caption{
The exact propagator  $G^{\star}$ (left),   the estimated propagator (right) and the relative error (bottom) of the Volterra estimator by observing  noisy trading strategies.
}
\label{FIGURE LABEL_e} 
\end{figure}


\end{document}